\documentclass[12pt]{amsart}
\usepackage{color}
\usepackage[all]{xy}
\usepackage{mathrsfs}
\usepackage{enumitem}

\usepackage{tikz}
\usetikzlibrary{calc}
\newcommand\irregularline[2]{%
  let \n1 = {rand*(#1)} in
  +(0,\n1)
  \foreach \a in {0.1,0.2,...,#2}{
    let \n1 = {rand*(#1)} in
    -- +(\a,\n1)
  } 
}
\usetikzlibrary{decorations.pathreplacing}

\usepackage{amscd,amssymb,verbatim,graphicx,stmaryrd,amsthm,color,amsmath} 
\usepackage{hyperref}

\calclayout

\date{February 28, 2023}

\usepackage{mathpazo} 

\newtheorem{dummy}{anything}[section]
\newtheorem{theorem}[dummy]{Theorem}
\newtheorem{lemma}[dummy]{Lemma}
\newtheorem{proposition}[dummy]{Proposition}
\newtheorem{corollary}[dummy]{Corollary}

\newtheorem{innercustomgeneric}{\customgenericname}
\providecommand{\customgenericname}{}
\newcommand{\newcustomtheorem}[2]{%
  \newenvironment{#1}[1]
  {%
   \renewcommand\customgenericname{#2}%
   \renewcommand\theinnercustomgeneric{##1}%
   \innercustomgeneric
  }
  {\endinnercustomgeneric}
}

\newcustomtheorem{customthm}{Theorem}
\newcustomtheorem{customcor}{Corollary}
\newcustomtheorem{customhyp}{Hypothesis}

\theoremstyle{definition}

  \newtheorem{example}[dummy]{Example}
  \newtheorem{remark}[dummy]{Remark}

\newcommand
{\eqncount}{\setcounter{equation}{\value{dummy}}%
\addtocounter{dummy}{1}}

\DeclareMathOperator{\End}{\mathrm{End}}
\newcommand{\sumdd}[2]{\displaystyle \sum_{#1}^{#2}}
\newcommand{\sumd}[1]{\displaystyle \sum_{#1}}
\newcommand{\limd}[1]{\displaystyle \lim_{#1}}
\newcommand{\intd}[1]{\displaystyle \int_{#1}}

\newcommand{\bb}[1]{\mathbb{#1}}

\newcommand{\defeq}{\mathrel{\mathpalette{\vcenter{\hbox{$:$}}}=}}
\newcommand{\SU}{\mathrm{SU}}
\newcommand{\SO}{\mathrm{SO}}

\newcommand{\A}{{\mathcal{A}}}
\newcommand{\G}{{\mathcal{G}}}

\newcommand{\C}{{\mathcal{C}}}

\newcommand{\CS}{{\mathcal{CS}}}

\newcommand{\eps}{\epsilon}

\newcommand{\Hom}{{\mathrm{Hom}}}

\newcommand{\F}{{\mathcal{F}}}

\newcommand{\M}{{\mathcal{M}}}
\newcommand{\cH}{{\mathcal{H}}}
\newcommand{\E}{{\mathcal{E}}}
\renewcommand\L{\mathcal{L}}
\newcommand{\T}{\mathcal{T}}

\newcommand{\mmr}{{morgan-mrowka-ruberman1}}
\newcommand{\Sl}{\mathcal{SL}}
\newcommand{\Stab}{\textrm{Stab}}

\title[Existence of $\mASD$\ connections]{Existence of $\mASD$\ connections on $4$-manifolds with cylindrical ends}

\author{David L. Duncan}
\address{Department of Mathematics \& Statistics, James Madison University,  Harrisonburg, Virginia 22807}
\email{duncandl@jmu.edu}

\author{Ian Hambleton}
\address{Department of Mathematics \& Statistics, McMaster University, Hamilton, Ontario L8S 4K1, Canada}
\email{hambleton@mcmaster.ca}
\thanks{Research partially supported by NSERC Discovery Grant A4000.
The authors would like  to thank the Fields Institute for Research in Mathematical Sciences in Toronto,  and the Max Planck Institut f\"ur Mathematik in Bonn for their hospitality and support while parts of this work were done.}

\newcommand{\cM}{\mathcal M}

\newcommand{\ASD}{{\textup{ASD}}}
\newcommand{\mASD}{{\textup{mASD}}}
\newcommand{\rM}{\widehat{\cM}}

\newcommand{\St}{S}
\newcommand{\rt}{r}
\newcommand{\Rt}{R}
\newcommand{\Qt}{Q}

\begin{document}

\begin{abstract} 
Taubes' gluing theorems establish the existence of $\ASD$ connections on closed, oriented $4$-manifolds. We extend these gluing results to the $\mASD$ connections of Morgan--Mrowka--Ruberman on oriented $4$-manifolds with cylindrical ends. As a corollary, we obtain an $\ASD$-existence result in the presence of degenerate asymptotic flat connections.
\end{abstract}

\maketitle

\section{Introduction}\label{sec:introduction}

The results of Taubes \cite{taubes4,taubes3} on ``gluing" establish the existence of non-trivial anti-self dual ($\ASD$) connections on closed, oriented 4-manifolds, provided one works with an $\SU(2)$-bundle with sufficiently high second Chern class. This was extended by Donaldson \cite{donaldsonb} to a general gluing theorem for connected sums; see also \cite{donaldsona,freed-uhlenbeck1}. These gluing results have direct extensions to cylindrical end 4-manifolds, provided one works with $\ASD$ connections having a \emph{non-degenerate} flat connection as an asymptotic limit \cite{donaldson10}. However, in the absence of such non-degeneracy assumptions, the space of $\ASD$ connections on a cylindrical end 4-manifold is generally not well-controlled (e.g., the $\ASD$ operator is not Fredholm) and this now-standard gluing formalism breaks down. Nevertheless, the question of existence for $\ASD$ connections in this degenerate cylindrical end setting remains well-posed. One of our main results, Theorem \ref{thm:A} below, establishes one such $\ASD$-existence result in the degenerate setting.

To state this, suppose $X$ is a connected, oriented $4$-manifold with cylindrical ends. Thus, we can write $X = X_0 \cup \End X$, where $X_0$ is a compact $4$-manifold with boundary $N$, and $\End  X \cong  [0,\infty) \times N$ is diffeomorphic to a cylinder. We refer to $X_0$ as the \emph{compact part} and to $\End  X$ as the \emph{cylindrical ends}. Unless otherwise stated, we allow the case where $N$ has multiple components, or is empty. Fix a metric $g$ on $X$ that is asymptotically cylindrical in the sense described in Section \ref{sec:AuxiliaryChoices}.

\begin{customthm}{A}\label{thm:A}
Assume $b^+(X) \leq 1$. Assume further that the 3-manifold $N$ is connected and satisfies one of the following:
\begin{itemize}
\item[(i)] $N$ is a circle bundle over a surface with positive Euler class: $e(N) > 0$; or
\item[(ii)] $N$ has first Betti number at most one: $b_1(N) \leq 1$. 
\end{itemize}
Then, for any integer $\ell \geq b^+(X) + 1$, the manifold $X$ admits an irreducible $\ASD$-connection $A$ on a principal $SU(2)$-bundle over $X$, and $A$ satisfies
$$  \int_{X} \vert F_A \vert^2 \; d \mathrm{vol} = 8 \pi^2 \ell.$$
\end{customthm}

We prove this in Section \ref{sec:ProofOfThmA}. As a concrete example, the hypotheses of Theorem \ref{thm:A} hold when $X_0$ is diffeomorphic to the total space of a positive Euler class disk bundle over a surface. To the authors' knowledge, Theorem \ref{thm:A} (and its extension, Theorem \ref{prop:exists3}) is the first general $\ASD$ existence result for cylindrical end manifolds that allows for a degenerate flat limit down the end.

Our approach to Theorem \ref{thm:A} is to (locally) embed the space of $\ASD$ connections into the larger space of \emph{modified ASD ($\mASD$)} connections of Morgan, Mrowka, and Ruberman \cite{\mmr}. This larger space is obtained by modifying the $\ASD$ operator in such a way that one obtains a Fredholm operator whose zero set contains an open set in the space of finite-energy $\ASD$ connections; it may also contain some new solutions. It is shown in \cite{\mmr} that, by allowing the auxiliary choices in this construction to vary, every finite-energy $\ASD$ connection belongs to some $\mASD$ space of connections defined in this way. The other main results of the present paper, stated below, show that the gluing results of Taubes and Donaldson for connected sums have extensions to this $\mASD$ setting. We then arrive at Theorem \ref{thm:A} as a consequence of these $\mASD$-gluing results; the topological hypotheses on $N$ imply that the $\mASD$ connections thus obtained are in fact $\ASD$.

Before stating these $\mASD$-gluing results, we give several remarks to help provide further context for this $\mASD$ setting.

\begin{remark} 
(a) Our primary motivation for developing these gluing results was to use the Morgan--Mrowka--Ruberman  ``moduli space'' of $\mASD$ connections to study the action on $X$ of a finite group $\pi$. Even in the $\ASD$ setting, generic perturbations are usually not $\pi$-equivariant, so the standard transversality arguments are not available, and one must appeal to some other approach to handle singularities in the moduli space. As a sequel to this paper, we planned to study the $\pi$-equivariant compactification of the ``$\mASD$ moduli space'' as was done in \cite{hlee2}, \cite{hlee3}, and \cite{htanase1} for the $\ASD$ moduli space. 

Unfortunately, the $\mASD$ operator \emph{fails} to be gauge equivariant in any reasonable sense (see Remark \ref{rem:noneq}). This appears to be an oversight in the original text \cite{\mmr} (e.g., see \cite[p.~125]{\mmr}), and at present we do not know how to define a suitable gauge quotient of the space of $\mASD$ connections that one might call the ``$\mASD$ moduli space''. It is a fundamental and interesting open problem to construct an appropriate $\mASD$-replacement for the $\ASD$ moduli space. See Section \ref{sec:mmrerror} for more information about this issue.

\medskip

(b) The foundational work of  Mrowka \cite{mrowka}, Morgan--Mrowka--Ruberman \cite{morgan-mrowka-ruberman1}, and Taubes \cite{taubes1, taubes2} concerning instantons on cylindrical end $4$-manifolds was done shortly before the Seiberg--Witten revolution in gauge theory. One of their striking results in this setting is that a finite-energy $\ASD$ connection has a well-defined limiting flat connection on the $3$-manifold $N$ ``at infinity''.

At that time, a central problem was to understand the behaviour of $\ASD$ connections under neck-stretching within a closed $4$-manifold, as well as the reverse operation in which $\ASD$ connections on non-compact $4$-manifolds with matching data on their cylindrical ends could be glued together. Indeed, the authors of \cite[p.~12]{morgan-mrowka-ruberman1} state:   ``The [\mASD] moduli space seems to provide the correct geometric context for a general gluing theorem for $\ASD$ connections, although we do not treat this topic in this book" (see Remark \ref{rem:endglue}). This point of view was a main ingredient in a paper of Fintushel and Stern \cite{fintushel-stern94} (and in unpublished work of Morgan and Mrowka \cite{morgan-mrowka94}). An account of gluing along cylindrical ends from the perspective of Floer homology was later provided by Donaldson \cite{donaldson10}, simplified by assuming  the presence of a perturbation to avoid degeneracies (see (c), below). We note, however, that the gluing results of the present paper take place on the compact part $X_0$, and not on the ends.

\medskip

(c) Researchers have worked around the technical issues involved in gluing in the degenerate setting by various methods. Of these methods, one of the most popular is to perturb the $\ASD$ equation on the ends in such a way that all perturbed-$\ASD$ connections are asymptotic to non-degenerate perturbed-flat connections \cite{floer1}, \cite{donaldson10}. However, this approach has several drawbacks. For one, $\ASD$ connections are generally not solutions of perturbed-$\ASD$ equations of this type; this can obscure the geometric information one can infer from an abstract existence result for perturbed-$\ASD$ connections (e.g., to what extent do these connections depend on the perturbation?). Another drawback is that these perturbation schemes are not well-behaved in the presence of reducible flat connections (e.g., the trivial flat connection), and this limits the applicability of such approaches. For example, a full $\SU(2)$-instanton Floer theory for 3-manifolds $N$ with $b_1(N) \geq 1$ is still lacking, and even the existing instanton Floer theory for integer homology spheres handles the trivial flat connection separately. In summary, a more in-depth understanding of $\ASD$ connections with degenerate limits is desired, and we view the results of this paper as being a step in that direction.
\end{remark}

To state our gluing results for $\mASD$ connections, let $G$ be a compact Lie group and fix a principal $G$-bundle $E \rightarrow X$. We assume that $E$ is translation-invariant on the end; that is, we assume the diffeomorphism $\End  X \cong \left[0, \infty \right) \times N$ is covered by a bundle isomorphism $E \vert_{\End X} \cong  [0, \infty ) \times Q $ for some principal $G$-bundle $Q \rightarrow N$. We also fix a flat connection $\Gamma$ on $Q$.

Given a connection $A$ on $E$ that converges sufficiently fast down the end, one can define a quantity
 $$\kappa(E, A\vert_{\End  X}) \defeq - \frac{1}{8 \pi^2} \intd{X} \langle F_A \wedge F_A \rangle  \in \bb{R}$$ 
that we call the \emph{relative characteristic number of the adapted bundle $(E, A\vert_{\End  X})$}; see Section \ref{sec:AdaptedBundles} for more details. If $A$ is $\ASD$, then $\kappa(E, A\vert_{\End  X}) = (8 \pi^2)^{-1} \int_X \vert F_A \vert^2$ equals the usual energy of the connection $A$. The upshot for us is that the quantity $\kappa(E, A\vert_{\End  X})$ is well-defined for a much larger class of connections than those with finite energy. Indeed, this relative characteristic number depends only on the topological type of the adapted bundle $(E,  A\vert_{\End  X})$, and it is a lift of the Chern--Simons value of the connection on $Q$ to which $A$ is asymptotic. Note that if $\kappa(E, A\vert_{\End  X}) \neq 0$, then $A$ is not flat. When $X$ is closed, then this relative characteristic number is actually an integer that depends only on $E$, and we will simply write it as $\kappa(E)$ (e.g., if $G = \SU(r)$, then $\kappa(E) = c_2(E)\left[X \right]$ is the second Chern number). We will primarily use $\kappa(E, A\vert_{\End  X})$ to keep track of the topological data in our gluing operations, just as the second Chern class keeps track of the underlying bundle type when gluing in the standard $\SU(2)$-setting for $\ASD$ connections on closed 4-manifolds.

By making several auxiliary choices, collectively called \emph{thickening data}, one can define the \emph{modified $\ASD$ ($\mASD$) operator}, which is a non-linear Fredholm map $s$ defined on a suitable space of connections on {$E$ (see Section \ref{sec:Connections} for definitions)}. In particular, we note that this space of connections is defined so that all elements are asymptotic to connections close to the flat connection $\Gamma$ fixed above. By definition, the $\mASD$ connections are those in the zero set of $s$, and we say that an $\mASD$ connection $A$ is \emph{regular} if the linearization of $s$ at $A$ is surjective {when restricted to a Coulomb slice}.

For $k = 1, 2$, suppose $X_k$ is an oriented, cylindrical end 4-manifold equipped with a principal $G$-bundle $E_k \rightarrow X_k$ and thickening data, as above. Let $X = X_1 \# X_2$ be a connected sum of these manifolds, taken at points in the compact parts of the $X_k$. Then the $E_k$ can be used to form a connected sum bundle $E \rightarrow X$, and we equip this with the thickening data induced from those of the $E_k$; see Section \ref{sec:Setupforgluing}. Our basic gluing result can be stated as follows.

\begin{customthm}{B}\label{thm:B}
For $k = 1, 2$, suppose $A_k$ is a regular $\mASD$ connection on $E_k$. Then for any $\epsilon > 0$, the bundle $E = E_1 \# E_2$ admits an $\mASD$ connection $A$ with the property that the distance between $A \vert_{X_k \cap X}$ and $A_k \vert_{X_k \cap X}$ is less than $\epsilon$ for $k = 1, 2$. Moreover, 
\eqncount\begin{equation}\label{eq:relclassineq}
\Big| \kappa(E, A \vert_{\End  X}) - \sumdd{k = 1}{2} \kappa(E_k, A_k \vert_{\End  X}) \Big| < \epsilon.
\end{equation}
\end{customthm}

{In the statement of Theorem \ref{thm:B}, the distance is relative to a $L^2_2(N) \times L^{p^*}_{ \delta}(X)$-metric on the space of connections; see (\ref{eq:close}) for a precise statement (the connection $A$ of Theorem \ref{thm:B} is what is called $\mathcal{J}(A_1, A_2)$ in (\ref{eq:close})).} Theorem \ref{thm:B} is a special case of Theorem \ref{thm:1}, which works in the broader setting where the $A_k$ are not necessarily regular. In this broader setting, the connection $A$ need not be $\mASD$, but its failure to be $\mASD$ is expressed through an obstruction map. In Theorem \ref{thm:2}, we extend Theorem \ref{thm:B} to a gluing result for \emph{families} of regular $\mASD$ connections. These results are $\mASD$-extensions of results familiar from the $\ASD$ setting; see \cite[Section 7.2]{donaldson-kronheimer1}.

As an application of Theorems \ref{thm:B} and \ref{thm:2}, we establish the following existence result, extending that of Taubes \cite{taubes4,taubes3} to the present cylindrical end $\mASD$ situation.

\begin{customthm}{C}\label{thm:C}
Assume $G = \SU(2)$ and $b^+(X) \leq 1$, and fix an integer $\ell \geq b^+(X) + 1$. Then for every $\epsilon > 0$, there is a principal $\SU(2)$-bundle $E \rightarrow X$ and an $\mASD$ connection $A$ on $E$ that is irreducible, and satisfies
\eqncount\begin{equation}\label{eq:relclassineq2}
\vert \kappa(E, A \vert_{\End X}) - \ell \vert < \epsilon.
\end{equation} 
If $b^+(X) = 0$, then the connection $A$ is regular. 
\end{customthm}

The cases $b^+(X) = 0$ and $b^+(X) = 1$ are special cases of Theorem \ref{thm:exist1} and Theorem \ref{thm:exist2}, respectively. Structurally, our proof strategies for these are very similar to the analogous statements in the closed case \cite{taubes4,taubes3} by realizing $X$ as a trivial connected sum $X \cong X \# S^4$. Under the assumption that $b^+(X) = 0$, it follows that the trivial flat connection on $X$ is regular as an $\mASD$ connection (see Remark \ref{rem:RegularityRemark}). It is well-known that the 4-sphere admits irreducible $\ASD$ connections of every positive second Chern class, and these are necessarily regular for topological reasons. Then Theorem \ref{thm:C} for $b^+(X) = 0$ follows from the general gluing result of Theorem \ref{thm:B} and adjacent results designed to handle gauge transformations (more below). We note also that Theorem \ref{thm:exist1} (the more general version of Theorem \ref{thm:C}) is proved for an arbitrary compact Lie group $G$, under mild hypotheses on $\ell$ and $G$.

\begin{remark} If $X$ is simply-connected and $b^+(X) = 0$, we expect that a modified gluing construction will produce an open subset of the space of \mASD\ 
connections that are in Coulomb gauge relative to some fixed connection. The issues involved in carrying out this improvement are briefly indicated in Remark \ref{rem:G2point}(b).
\end{remark}

The strategy for our proof of Theorem \ref{thm:C} when $b^+(X) = 1$ is similar, albeit more involved since the trivial flat connection on $X$ is no longer regular. Thus a careful analysis of the obstruction map of Theorem \ref{thm:1} is required. Just as in \cite{taubes3}, we glue $\ASD$ connections on $S^4$ at several sites instead of one, and this is sufficient to show that the obstruction vanishes for some choice of gluing parameters. In this analysis, we use the assumption that $G = \SU(2)$. As Taubes mentions \cite[p.~518]{taubes3}, it is likely that the restriction to $G =SU(2)$ can be removed, but that would call for a different approach. 

We prove our general existence results only for $b^+  \leq 1$ because (i) these are the cases of interest for our applications, and (ii) extending the discussion to higher values of $b^+$ would add considerable length to the paper (this can already be seen in \cite{taubes3}). {For similar reasons we also carry out our analysis with connections that are locally in Sobolev class $L^p_1$ as opposed to, say, $L^p_k$ for $k \geq 1$. We do not see any inherent obstruction to extending our results to higher Sobolev spaces and presumably such extensions would recover our $L^p_1$-results by elliptic regularity. We leave the details of such extensions to interested parties.}

The appearance of $\epsilon > 0$ in the statements of Theorems \ref{thm:B} and \ref{thm:C} is new to this $\mASD$ setting. To explain it, we note that in the standard set-up of gluing $\ASD$ connections on a closed 4-manifold, the inequality (\ref{eq:relclassineq}) would be replaced by the equality $\kappa(E_1 \# E_2) =  \kappa(E_1) + \kappa(E_2)$; likewise (\ref{eq:relclassineq2}) would be replaced by $\kappa(E) = \ell$. The presence of an \emph{inequality} for us reflects a need to freely vary the asymptotic values in order to obtain the $\mASD$ connection $A$. Indeed, as discussed further in Remark \ref{rem:RegularityRemark}, the ability to freely vary these asymptotic values is at the heart of what makes the $\mASD$ set-up a viable candidate for the type of existence statement in Theorem \ref{thm:C} and thus Theorem \ref{thm:A}. For example, when $b^+(X) = 0$, the trivial flat connection is regular only because the $\mASD$ operator allows for this variation in the asymptotic values.

 In Section \ref{sec:PartialCompactification}, we have included a discussion of how Theorem \ref{thm:C} for $b^+(X) = 0$ provides a ``partial compactification'' of the space of $\mASD$ connections. We also discuss why this compactification is only partial, and what a more complete compactification would require. 

As mentioned above, the lack of gauge-equivariance for the $\mASD$ operator means that we are \emph{not} free to pass to the quotient modulo gauge. Indeed, to obtain a Fredholm problem for the gluing constructions, we work entirely within a fixed Coulomb slice. Since the natural Coulomb slice varies as the connections vary, this dependence becomes relevant when we glue over families of connections, which is necessary for Theorem \ref{thm:C}. This is a central obstacle with which we must contend in the present paper: In the usual $\ASD$ setting, one could apply suitable gauge transformations that put all nearby $\ASD$ connections into the same slice. However, in this $\mASD$ setting, the gauge-transformed $\mASD$ connections would no longer be $\mASD$. To account for this, we establish a pair of gauge fixing results, Proposition \ref{prop:1} and Theorem \ref{thm:3}, that show that, by a making an additional perturbation, an $\mASD$ connection in one Coulomb slice can be perturbed to an $\mASD$ connection in a nearby Coulomb slice.

Apart from the failure of gauge equivariance in the $\mASD$ setting, the main difference between the $\mASD$ and $\ASD$ settings is that we now need to handle the additional nonlinearities that arise from the term modifying the $\ASD$ operator. The key observation we use for handling this term is that it factors through a finite-dimensional manifold.

Finally, we mention that if $\Gamma$ is non-degenerate, then every $\mASD$ connection with asymptotic value near $\Gamma$ is in fact $\ASD$. E.g., this non-degeneracy hypothesis is satisfied when $N$ is a rational homology 3-sphere and $\Gamma$ is the trivial connection. As such, our results recover standard gluing results for $\ASD$ connections on cylindrical end manifolds with non-degenerate asymptotic limits; see Sections \ref{sec:Non-Deg} and \ref{sec:ProofOfThmA} for more details. More interestingly, there are situations for which $\Gamma$ is degenerate, but for which every $\mASD$ connection with asymptotic value near $\Gamma$ is $\ASD$. In such cases, our $\mASD$ gluing theorem produces an $\ASD$ connection. Theorem \ref{thm:A} is one result of this type.

\tableofcontents

\section{Background on the thickened moduli space}\label{sec:Connections}

In this section we give a rapid review of the relevant background material from \cite{\mmr}; we also expand on some of the results of \cite{\mmr}, which will assist in our discussion of gluing below. With a few exceptions, we use much of the same notation and set-up established in \cite{\mmr}. To allow for a more streamlined discussion, we assume throughout that the 3-manifold end $N$ is nonempty; however, see Section \ref{sec:SpecialCases} for an extension to the case $N = \emptyset$. 

We will write $\A(E)$ and $\G(E)$ for, respectively, the spaces of smooth connections and gauge transformations on $E \rightarrow X$. When the bundle is clear from context, we will simply write $\A(X)$ and $\G(X)$. Given a connection $A$, we denote by $F_A$ its curvature, which is a 2-form on $X$ with values in the adjoint bundle $\mathfrak{g}_E$. We will write $\Omega^\ell(X)$, and sometimes $\Omega^\ell$, for the space of smooth adjoint bundle-valued $\ell$-forms on $X$ that are rapidly decaying. 

To touch base with constants associated with characteristic classes below, we work relative to an inner product on $\mathfrak{g}$ obtained as follows. Fix a Lie group homomorphism 
\eqncount\begin{equation}\label{eq:rep}
G \longrightarrow \SU(r)
\end{equation} 
that is also an immersion. Then the induced map $\mathfrak{g} \hookrightarrow \mathfrak{su}(r)$ is an embedding of Lie algebras. Let $\langle \cdot, \cdot \rangle: \mathfrak{g} \otimes \mathfrak{g} \rightarrow \bb{R}$ denote the inner product on $\mathfrak{g}$ obtained by pulling back the inner product $A \otimes B \mapsto - \mathrm{tr}(A B)$ on $\mathfrak{su}(r)$. This inner product is $\mathrm{Ad}$-invariant, and so determines a metric on the adjoint bundle $\mathfrak{g}_E$. 

 Notation such as $L^p_{k}(\Omega^\ell(X), g)$ will denote the $L^p_k$-Sobolev completion of $\Omega^\ell(X)$, relative to a metric $g$ on $X$ and the above-defined metric on $\mathfrak{g}_E$. When $X$ or $g$ are clear from context, or not relevant, we may drop them from the notation.

\subsection{Auxiliary choices}\label{sec:AuxiliaryChoices}

\subsubsection{The center manifold}\label{sec:TheCenterManifold}

Recall from the introduction that we have fixed a bundle $Q \rightarrow N$ as well as a smooth flat connection $\Gamma$ on $Q$. Fix a metric $g_N$ on $N$. Let $U_\Gamma \subseteq L^{2}_{2}(\A(N))$ be a coordinate patch centered at $\Gamma$, in the sense of \cite[Def. 2.3.1]{\mmr}; for our purposes, it suffices to know that $U_\Gamma$ is a small open neighborhood of $\Gamma$ in the Coulomb slice $\left\{\Gamma \right\} + \ker(d_\Gamma^*)$. As in \cite[Lemma 2.5.1]{\mmr}, there is a unique $\mathrm{Stab}(\Gamma)$-equivariant map 
$$\Theta\colon  U_\Gamma \longrightarrow L^2_2(\Omega^0(N))$$
with $\Theta(a) \in (\ker \Delta_\Gamma)^\perp$ and 
$$d_\Gamma^*(* F_{a} - d_{a} \Theta(a)) = 0.$$
It follows from this last equation, and uniqueness, that if $a$ has higher regularity then so too does $\Theta(a)$. 

We will be interested in the densely-defined vector field
$$\nabla f_\Gamma \colon  U_\Gamma \longrightarrow T U_\Gamma \hspace{2cm}a \longmapsto \nabla f_\Gamma(a) \defeq - * F_a + d_a \Theta(a).$$
Note that the zeros of $\nabla f_\Gamma$ are precisely the flat connections in $U_\Gamma$. (As described in \cite[Lemma 2.5.1(1)]{\mmr}, this vector field is the (negative) gradient of the restriction to $U_\Gamma$ of the Chern--Simons functional, where the gradient is taken relative to a certain inner product that takes into account the possibility of a non-trivial stabilizer of $\Gamma$.)

For $m \geq 2$, let $\cH = \cH_\Gamma \subseteq U_\Gamma$ be a $\mathrm{Stab}(\Gamma)$-invariant $\C^m$-center manifold for the vector field $\nabla f_\Gamma$, as in \cite[Cor. 5.1.4]{\mmr}. In particular, this means that
\begin{itemize}
\item $\cH_\Gamma$ is a finite-dimensional $\C^m$-manifold containing $\Gamma$,
\item the tangent space to $\cH_\Gamma$ at $\Gamma$ is the $\Gamma$-harmonic space 
$$H^1_\Gamma \defeq \ker(d_\Gamma \oplus d_\Gamma^*) \subseteq \Omega^1(N),$$
\item $\nabla f_\Gamma$ is tangent to $\cH_\Gamma$, and
\item every zero of $\nabla f_\Gamma$ sufficiently close to $\Gamma$ is contained in $\cH_\Gamma$.
\end{itemize}
We denote by $\Xi = \Xi_\Gamma$ the restriction of $\nabla f_\Gamma$ to $\cH_\Gamma$.

Fix a compactly supported cutoff function $\beta\colon  \cH \rightarrow \left[0, 1 \right]$ that is identically 1 near $\Gamma$. The \emph{trimmed vector field} is given by
$$\Xi^{tr}(h) \defeq \beta(h) \Xi(h).$$
Set $\cH_{in} = \beta^{-1}(1)$ and $\cH_{out} = \beta^{-1}((0, 1])$.

Fix a real number $T \geq 1$. The trimmed vector field is complete and so, for each $h \in \cH$, there is a unique solution $h_T\colon  \bb{R} \rightarrow  \cH$ to the flow
$$\frac{d}{dt} h_T(t) = \Xi^{tr}(h_T(t)) \hspace{2cm} h_T(T) = h.$$
We then set
$$\alpha(h) = \alpha_{T}(h) \defeq  h_T(t) + \Theta(h_T(t)) dt.$$
Depending on context, we may view $\alpha$ as a connection on the submanifold $\End X \cong \left[0, \infty \right) \times N$, or on the cylinder $\bb{R} \times N$.

\begin{lemma}\label{lem:regularityofalpha}
For all $h \in \cH$, the connection $\alpha(h)$ is in $L^2_{2, loc}(\bb{R} \times N) \cap \C^0(\bb{R} \times N)$, and hence in $L^p_{1, loc}( \bb{R} \times N) \cap \C^0(\bb{R} \times N)$ for any $1 \leq  p < 4$. Moreover, the map $\alpha: h \mapsto \alpha(h)$ is $\C^m$ relative to the $L^2_2(N)$-topology on the domain and the $\C^0(\bb{R} \times N)$-topology on the codomain. 
\end{lemma}

\begin{proof}
The initial condition $h$ is in $L^2_2(N) \subset \C^0(N)$, by assumption. It then follows from standard regularity arguments for flows that the path $h_T$ is in $L^2_{2, loc}\cap \C^0$ on $\bb{R} \times N$. Hence $\alpha(h)$ is in the same space as well, since the regularity of $\Theta(h_T)$ is controlled by that of $h_T$. That $\alpha$ is $\C^m$ relative to these topologies follows from a similar argument applied to its $k$th derivative for $1 \leq k \leq m$. The assertion about $L^p_1$ follows from the embedding $L^2_{2, loc} \hookrightarrow  L^p_{1, loc}$, which holds provided $1 \leq p < 4$.
\end{proof}

In $\cH_{in}$, the flow defining $h_T(t)$ is gauge equivalent to the Chern--Simons gradient flow. It follows that if $h_T(t) \in \mathrm{int}(\cH_{in})$ lies in the interior for some $t$, then $\alpha(h)$ is $\ASD$ in a neighborhood of $\left\{t \right\} \times N$. This has the following useful linear analogue for the linearization $(D \alpha)_\Gamma$ of $\alpha$ at $\Gamma$.

\begin{lemma}\label{lem:derofalpha}
If $\eta \in T_\Gamma \cH = H^1_\Gamma$, then 
$$(D \alpha)_\Gamma \eta = \pi^* \eta$$
where $\pi: \bb{R} \times N \rightarrow N$ is the projection.
\end{lemma}

\begin{proof}
Since $\Gamma$ is flat, we have $\Theta(\Gamma) =0$ and so $\alpha(\Gamma) = \pi^*  \Gamma$. This is a flat connection on $\bb{R} \times N$ and so it is $\ASD$:
$$F^+_{\alpha(\Gamma)} = F^+_{\pi^* \Gamma} = 0.$$
Linearizing this in the direction of $\eta \in T_\Gamma \cH$ implies
$$d^+_{\pi^* \Gamma} ((D \alpha)_\Gamma \eta) = 0.$$
Note that, since $\eta$ is $\Gamma$-harmonic, we also have
$$d^+_{\pi^* \Gamma} (\pi^* \eta) = 0.$$
In general, we can view the kernel of $d^+_{A}$ on the cylinder $\bb{R} \times N$ as defining a flow on the space $\Omega^1(N) \times \Omega^0(N)$. In the case of $(D \alpha)_\Gamma \eta$ and $\pi^* \eta$, these both take values in the graph
$$\mathrm{Graph}((D\Theta)_\Gamma \vert_{T_\Gamma \cH}) \subseteq T_\Gamma \cH \times L^2_2(\Omega^0(N))$$
of the linearization of $\Theta$. Observe two things: (i) this graph is finite-dimensional, and (ii) the 1-forms $(D \alpha)_\Gamma \eta$ and $\pi^* \eta$, viewed as paths in the graph, both equal $\eta$ at time $T$. Then $(D\alpha)_\Gamma \eta = \pi^* \eta$ follows by the uniqueness for flows on finite-dimensional spaces. 
\end{proof}

\subsubsection{The choice of metric}\label{sec:TheChoiceOfMetric}

We will use $t\colon  \End X \rightarrow [0, \infty)$ to denote the projection relative to the identification $\End  X \cong [0, \infty) \times N$. With the use of a cutoff function, we can view $t$ as a smooth real-valued function defined on all of $X$, which we will denote by the same symbol.

Fix a smooth cylindrical end metric $g_0$ on $X$; this means that the restriction
$$g_0 \vert_{\End  X} = dt^2 + g_N$$
is a product metric, where $g_N$ is the fixed metric on $N$. Let $B$ be a $\C^3$-neighborhood of $g_0$ in the space of $\C^{\max(m, 3)}$-metrics on $X$ so that the conclusions of \cite[Theorem 2.6.3]{\mmr} hold (the proof of Theorem 2.6.3 shows that such a set exists; the details of the theorem will not play an active role in the discussion that follows). Let $\mu_\Gamma^\pm$ be the smallest positive eigenvalue of $\mp * d_\Gamma : \Omega^1(N) \rightarrow \Omega^1(N)$. The sign convention here is to agree with that of \cite[Def. 2.1.1]{\mmr}. Then we will say that a metric $g$ on $X$ is \emph{asymptotically cylindrical} if $g \in B$ and 
$$\Vert  g - g_0 \Vert_{\C^1(\left\{t \right\} \times N)} \leq  e^{- \max(\mu_\Gamma^-, \mu_\Gamma^+) t}$$
for all $t \geq 0$. (This is effectively Condition A3 of \cite[p.~116]{\mmr}.) Throughout, we will always assume our metrics are asymptotically cylindrical in this sense. Note that every cylindrical end metric is automatically asymptotically cylindrical.

\subsubsection{Thickening data}\label{sec:ThickeningData}

Fix data as in \cite[Section 7.2]{\mmr}; we will refer to this as the \emph{thickening data} and denote it by $\mathcal{T}_\Gamma$. In particular, this includes the choice of positive numbers $\eps_0$ and $\delta$, that we will describe momentarily. The details of the remaining data in $\mathcal{T}_\Gamma$ will not play an active role in our discussion. For convenience, we also assume that $\mathcal{T}_\Gamma$ includes the choice of the fixed $T \geq 1$ from above.

The key feature for us regarding $\eps_0$ is that $\vert \CS(h_2)  - \CS(h_1) \vert < \eps_0 / 2$ for all $h_1, h_2 \in \mathrm{supp}(\beta)$, where $\CS$ is the Chern--Simons function. For any $\epsilon_0 > 0$, this inequality can be arranged by shrinking the support of $\beta$, if necessary. {The remaining requirements for $\eps_0$ will not be directly relevant to us, but see \cite[Definition 4.3.2]{\mmr} for more details.} As for $\delta$, we assume $\delta > \mu_\Gamma^-$ and that $\delta/2$ is not an eigenvalue of $-* d_\Gamma$. By shrinking the size of the coordinate patch $U_\Gamma$, if necessary, we may assume further that $\delta / 2$ is not an eigenvalue of $-* d_{a}$ for any $a \in U_\Gamma$. At various times, we may place additional restrictions on $\delta$.

\subsubsection{Weighted spaces}

We define the space $L^p_{k, \delta}(X)$ to be the completion of the set of compactly supported smooth forms $f$ on $X$, relative to the weighted Sobolev norm:
$$\Vert f \Vert_{L^p_{k,\delta}} \defeq \Vert e^{\delta t / 2} f \Vert_{L^p_k}$$
When $p = 2$, this recovers the family of norms used in \cite{\mmr}. The subspace of $\ell$-forms will be denoted by $L^p_{k, \delta}(\Omega^\ell)$ or  $L^p_{k, \delta}(\Omega^\ell(X))$. Following standard conventions, when $k = 0$, we will write $L^p_{\delta} $ for $ L^p_{0,\delta}$.  Note that the norm $\Vert f \Vert_{L^p_{k,\delta}} $ is equivalent to the norm:
$$\sum_{0 \leq j \leq k} \Vert e^{\delta t/ 2} \nabla^j f \Vert_{L^p}$$
In particular, we can use this equivalence to transfer Sobolev embedding results for $L^p_k$ to the weighted setting; e.g., see the proof of Lemma \ref{lem:quad1}.

\subsection{Gauge theory}\label{sec:GaugeTheory}

\subsubsection{The space of connections}\label{sec:TheSpaceOfConnections}

For $1 \leq p < 4$, define $\A^{1,p}(\mathcal{T}_\Gamma)  $ to consist of the connections $A$ on $E$ satisfying the following:
\begin{itemize}
\item $A$ has regularity $L^p_{1,{loc}}$,
\item there is some $h \in \cH_{out}$ so that $A - \alpha(h) \in L^p_{1, \delta}(\Omega^1(\End  X))$,
\item for each $t \geq T$, the connection $A \vert_{\left\{t \right\} \times N}$ is gauge equivalent to a connection in the coordinate patch $U_\Gamma$ centered at $\Gamma$.
\end{itemize}
This space of connections is generally \emph{not} an affine subspace of $L^p_{1, loc}(\A(E))$; this reflects the nonlinearities in the definition of the map $h \mapsto \alpha(h)$. We give $\A^{1,p}(\mathcal{T}_\Gamma)$ the structure of a $\C^m$-Banach manifold, as in \cite[Section 7.2.2]{\mmr}. (Equivalently, this $\C^m$-Banach manifold structure is precisely the one for which the map $\iota$, defined in (\ref{eq:iotamapdiff}) below, is a $\C^m$-diffeomorphism.) By \cite[Prop.~7.2.3]{\mmr} (see also \cite[p.~120]{\mmr}), given $A \in \A^{1,p}(\mathcal{T}_\Gamma) $, the element $h \in \cH_{out}$ from the second bullet point is uniquely determined; this uses the assumption that $\delta > \mu_\Gamma^-$. As such, there is a well-defined map
$$p_T\colon  \A^{1,p}(\mathcal{T}_\Gamma)  \longrightarrow \cH_{out}$$
that is $\C^m$-smooth. 

\begin{remark}\label{rem:pbound}
(a) Note that our space $\A^{1,p}(\mathcal{T}_\Gamma)$ consists of connections with weaker regularity than the one in \cite[Ch. 7]{\mmr}, which is modeled on $L^2_2$ instead of $L^p_1$. This changes little as far as the exposition of \cite{\mmr} is concerned; the only significant exception to this is the gauge group, which we will discuss in the next section.

\medskip

(b) The restriction to $p$ less than 4 is, at the moment, coming from Lemma \ref{lem:regularityofalpha}. A much deeper reason for restricting to $p$ less than 4 will appear in our gluing analysis of Section \ref{sec:GenGluing} (in particular, (\ref{eq:estyp0})), where this condition ensures we have a right inverse to our linearized operator that is uniformly bounded; see also \cite[p.~293]{donaldson-kronheimer1}.
\end{remark}

Fix a smooth cutoff function $\beta''$ on $X$ supported on $\left[T - 1/2, \infty \right) \times N$ and identically 1 on $[T, \infty) \times N$. Fix also a smooth reference connection $A_{ref}$ on $E$; we assume this belongs to the space $\A^{1,p}(\mathcal{T}_\Gamma)  $. Using these objects, we can form the map
$$\begin{array}{rcl}
i  \colon  \cH_{out} & \longrightarrow & \A^{1,p}(\mathcal{T}_\Gamma)\\
h & \longmapsto  & A_{ref} + \beta''(\alpha(h) - A_{ref})
\end{array}$$
where $\alpha(h) = h_T(t) + \Theta(h_T(t)) dt$ is as above. This map $i$ is $\C^m$-smooth. As in \cite[Lemma 10.1.1]{\mmr}, it is convenient to introduce the map
\eqncount\begin{equation}\label{eq:iotamapdiff}
\begin{array}{rcl}
\iota \colon  \cH_{out} \times L^p_{1, \delta}(\Omega^1(X)) & \longrightarrow & \A^{1,p}(\mathcal{T}_\Gamma)\\
(h, V) & \longmapsto & \iota(h, V) \defeq i(h) + V
\end{array}
\end{equation}
This map $\iota$ is a $\C^m$-diffeomorphism with inverse given by $A \mapsto (p_T(A), A - i(p_T(A)))$. It follows immediately from the definitions that 
$$p_T(i(h)) = p_T(\iota(h, V))  =  h$$
for all $h \in \cH_{out}$ and $V \in L^p_{1, \delta}(\Omega^1)$. We view $\iota$ as providing something of a coordinate system on the space of connections.

The tangent space to $\A^{1,p}(\mathcal{T}_\Gamma)$ at $A$ is the space of all 1-forms $W \in L^p_{1, loc}(\Omega^1(X))$ so that there is some $\eta \in T_{p(A)} \cH$ with
$$W - (D i)_{p_T(A)} \eta \in L^p_{1, \delta}(X)$$
where $(D i)_h $ is the linearization at $h \in \cH_{out}$ of the map $i\colon  \cH_{out} \rightarrow\A^{1,p}(\mathcal{T}_\Gamma)  $. Linearizing $\iota$ at $(h, V)$, we obtain a Banach space isomorphism
\eqncount\begin{equation}\label{eq:Dnorma}
\begin{array}{rcl}
(D \iota)_{(h, V)}\colon  T_h \cH_{out} \times L^p_{1, \delta}(\Omega^1(X)) & \longrightarrow & T_A \A^{1,p}(\mathcal{T}_\Gamma) \\
(\eta, W)&  \longmapsto & (Di)_h \eta +  W
\end{array}
\end{equation}
The 1-form $(Di)_h \eta$ vanishes on $X_0$, so the operator norm of $(D \iota)_{(h, V)}$ is independent of the metric on $X_0$.

\subsubsection{The gauge group}\label{sec:GaugeGroup}

When $2 < p < 4$, we will write $\G^{2,p}_\delta(\Gamma)$ for the set of bundle automorphisms $u$ of $E$ with the property that $u^* A \in \A^{1, p}(\mathcal{T}_\Gamma)$ for all $A \in \A^{1,p}(\mathcal{T}_\Gamma)$. The condition that $p$ be less than 4 is a carry-over from Remark \ref{rem:pbound} (b). The condition that $p$ be larger than 2 ensures that $\G^{2,p}_\delta(\Gamma)$ is a well-defined Banach Lie group that acts $\C^m$-smoothly on $\A^{1,p}(\mathcal{T}_\Gamma)$; these claims follow from the Sobolev multiplication maps $W^{2, p}_{loc} \times W^{2, p}_{loc} \rightarrow W^{2, p}_{loc}$ and $W^{2, p}_{loc} \times W^{1, p}_{loc} \rightarrow W^{1, p}_{loc}$ being well-defined in dimension 4 for $p > 2$. See, for example, \cite[Lemma A.6]{Wehrheim}. We will only consider $\G^{2, p}_\delta(\Gamma)$ for $p$ satisfying $2 < p < 4$. 

The proof of \cite[Lemma 7.2.7]{\mmr} carries over to this setting to imply that the group $\G^{2, p}_\delta(\Gamma)$ is equal to the space of $L^{p}_{2, loc}$-gauge transformations with the property that there is some $\tau_u \in \Stab(\Gamma)$, viewed as a $t$-invariant gauge transformation on $\End  X$, so that $u \vert_{\End  X} \circ \tau^{-1}_u$ is in $L^{p}_{2, \delta}(\End  X)$. The gauge transformation $\tau_u$ is uniquely determined by $u$, and we denote by $\G^{2, p}_\delta \subseteq \G^{2,p}_\delta(\Gamma)$ the (normal) subgroup of all gauge transformations $u$ with $\tau_u = \mathrm{I}$ equal to the identity. Thus, we have a short exact sequence of groups:
\eqncount\begin{equation}\label{eq:gaugesequence}
\left\{e \right\} \longrightarrow \G^{2, p}_\delta \longrightarrow \G^{2, p}_{\delta}(\Gamma) \longrightarrow \mathrm{Stab}(\Gamma) \longrightarrow \left\{e \right\}
\end{equation}

We will write $\Stab(A)$ for the stabilizer of $A$ under the action of $\G^{2,p}_\delta(\Gamma)$. 
The center $Z(G)$ of $G$ embeds into $\G^{2, p}_\delta(\Gamma)$ as the set of constant maps $X \rightarrow Z(G)$, and we will identify $Z(G)$ with its image in the gauge group. Note that $Z(G)$ is also the center of $\G^{2,p}_\delta(\Gamma)$ and $Z(G) \subseteq \Stab(A)$. We will say that $A$ is \emph{irreducible} if $Z(G)$ and $\Stab(A)$ have the same dimension (equivalently, if they have isomorphic Lie algebras). Note that the term ``irreducible'' is only defined when $2 < p < 4$.

\begin{lemma}\label{lem:irreducible}
Fix $2 < p  < 4$, and assume $A \in \A^{1, p}(\mathcal{T}_\Gamma)$ is irreducible. Then there is a neighborhood $U \subseteq \A^{1, p}(\mathcal{T}_\Gamma)$ of $A$ so that $A'$ is irreducible for all $A' \in U$. 
\end{lemma}

\begin{proof}
We begin with a few preliminaries. Set $\A \defeq \A^{1,p}(\mathcal{T}_\Gamma)$ and $\G \defeq \G^{2, p}_\delta(\Gamma)$. Linearizing the gauge group action at $A \in \A$, we obtain a map
$$d_A: \mathrm{Lie}(\G) \longrightarrow T_A \A \hspace{2cm} \phi \longmapsto d_A \phi.$$
Then a connection $A \in \A$ is irreducible if and only if the kernel of $d_A$ equals the Lie algebra $\mathfrak{z} \defeq \mathrm{Lie}(Z(G))$ of the center of $\G$. It follows form the definition of the topologies on $\A$ and $\G$, as well as from standard elliptic estimates for $\delta$-decaying spaces, that the operator $d_A$ is bounded. Moreover, this operator has a range that is closed and admits a complement in $T_A \A$. 

Let $H^0_\Gamma \defeq \ker(d_\Gamma) \subseteq \Omega^0(N)$ be the Lie algebra of $\Stab(\Gamma)$. The center $\mathfrak{z}$ is naturally a subalgebra of $H^0_\Gamma$, so we can write
$$H^0_\Gamma = \mathfrak{z} \oplus \mathfrak{z}^\perp$$
where $\mathfrak{z}^\perp$ is the $L^2(N)$-orthogonal complement of $\mathfrak{z}$. Then for $\tau \in H^0_\Gamma$, we will write $\tau^\perp \in \mathfrak{z}^\perp$ for its projection.

Just as the map $\iota$ provides ``coordinates'' for $\A$, there is an analogous Banach space isomorphism
$$\iota_{\Omega^0}: H^0_\Gamma \times L^p_{2, \delta}(\Omega^0) \longrightarrow \mathrm{Lie}(\G) \hspace{2cm} (\tau, \xi) \longmapsto (\tau- \tau^\perp) + \beta'' \tau^\perp + \xi$$
where we are viewing $\tau - \tau^\perp \in \mathfrak{z}$ as a 0-form on $X$. This map $\iota_{\Omega^0}$ takes $\mathfrak{z} \times \left\{0 \right\}$ isomorphically to $\mathfrak{z} \subseteq \mathrm{Lie}(\G)$. Let $\Upsilon \subseteq \mathrm{Lie}(\G)$ be the image under $\iota_{\Omega^0}$ of the complement $\frak{z}^\perp \times  L^p_{2, \delta}(\Omega^0) $ to $\mathfrak{z} \times \left\{0 \right\}$. Then we have a direct sum decomposition
$$\mathrm{Lie}(\G) = \mathfrak{z} \oplus \Upsilon.$$
The key point is that $A$ is irreducible if and only if the restriction
 $$d_A^\perp \defeq d_A \vert_{\Upsilon} : \Upsilon \longrightarrow T_A \A $$
 is injective. We will want to view this operator as a function of $A$, and for this it would be convenient if $d_A^\perp$ were to have a codomain that is independent of $A$. Though this is not the case presently, we can arrange for $A$-independence of the codomain as follows: Let $(D \iota)_A: T_h \cH \times L^p_{1, \delta}(\Omega^1)\rightarrow T_A \A $ be the linearization of the coordinate map $\iota$. The $L_2^2$-inner product for 1-forms on $N$ provides a Riemannian metric on the finite-dimensional subspace $\cH \subseteq L^2_2(\A(N))$ of connections on $N$. Use this Riemannian metric to define the parallel transport map $\mathrm{PT}_h: T_h \cH \rightarrow T_\Gamma \cH$. Letting $I$ denote the identity on $L^p_{1, \delta}(\Omega^1)$, we will be interested in the operator
 $$D_A \defeq (\mathrm{PT}_h \times I) \circ (D \iota)_A^{-1} \circ d_A^\perp: \Upsilon \longrightarrow T_\Gamma \cH \times L^p_{1, \delta}(\Omega^1).$$
This is a bounded linear map, and expansions of the form $d_{A} = d_{A'} + \left[ (A -A'), \cdot \right]$ show it depends continuously on $A \in \A$ in the operator norm topology on the space $\mathcal{B}(\Upsilon, T_\Gamma \cH \times L^p_{1, \delta}(\Omega^1))$ of bounded linear maps from $\Upsilon$ to $T_\Gamma \cH \times L^p_{1, \delta}(\Omega^1)$. Since $\Upsilon$ has finite codimension, and $d_A$ has closed range, the operator $D_A$ has closed range as well.

It follows from the construction that $D_A$ is injective if and only if $A$ is irreducible. Assume that $A$ is irreducible. Then the fact that $\mathrm{im}(d_A)$ has a complement in $T_A \A$ implies that $D_A$ admits a bounded left inverse, which we denote by $L_A$. In summary, the map 
 $$\A \longrightarrow \mathcal{B}(\Upsilon, \Upsilon) \hspace{2cm}A' \longmapsto L_A D_{A'}$$ 
 is a continuous map into the space of bounded linear operators on the Banach space $\Upsilon$. It is clearly invertible at $A' = A$. Since the set of invertible bounded linear maps on a Banach space is open, there is some neighborhood $U \subseteq \A$ of $A$ so that $L_A D_{A'}$ is invertible for all $A' \in U$. Thus if $A' \in U$, then $A'$ is irreducible. 
\end{proof}

\begin{remark}\label{rem:irreduciblesmallp}
Completing $L^p_{2, \delta}(\Upsilon)$ to $L^p_{1, \delta}(\Upsilon)$, the map $D_A$ extends to a bounded linear operator of the form
$$D_A: L^p_{1, \delta}(\Upsilon) \longrightarrow T_\Gamma \cH \times L^p_\delta(\Omega^1(X)).$$
Let $p^* = 4p / (4 - p)$ be the Sobolev conjugate of $p \in (2, 4)$. Then one can show that the map $A \mapsto D_A \in \mathcal{B}(L^p_{1, \delta}(\Upsilon) , T_\Gamma \cH \times L^p_\delta(\Omega^1(X)))$ is continuous in $A = i(h) + V$ relative to the topology $(h, V) \in \C^0(N) \times L^{p^*}_{\delta}(X)$. The proof we gave for Lemma \ref{lem:irreducible} carries over to show that $A' = \iota(h', V')$ is irreducible whenever $A = \iota(h, V)$ is irreducible, and $\Vert h - h' \Vert_{\C^0} + \Vert V - V' \Vert_{L^{p^*}_\delta}$ is sufficiently small.
\end{remark}

\subsubsection{The $\mASD$ equation}

Fix a cut off function $\beta' $ on $X$ that is identically 1 on the cylinder $[T+1/2, \infty) \times N$ and supported on the slightly larger cylinder $[T, \infty) \times N$. Consider the map
$$s  \colon  \A^{1,p}(\mathcal{T}_\Gamma)   \longrightarrow L^p_\delta(\Omega^+(X)) \hspace{2cm} A \longmapsto F^+_A - \beta' F_{i(p_T(A))}^+.$$
We will call $s$ the \emph{modified $\ASD$ ($\mASD$) operator}. The equation $s(A) = 0$ is the \emph{modified $\ASD$ ($\mASD$) equation}, and any $A$ satisfying $s(A) = 0$ will be called \emph{modified $\ASD$ ($\mASD$)}. The map $s$ is $\C^{m}$ in the specified topologies; see \cite[Lemma 7.1.1]{\mmr} and use the fact that the composition of $\C^m$ functions is again $\C^m$.

The following will help us understand the linearization of $s$.

\begin{lemma}
If $A = \iota(h, V)$ for $(h, V)  \in \cH_{out} \times L^p_{1, \delta}(\Omega^1(X)) $, then
\eqncount\begin{equation}\label{eq:formfors}
s(A)  = s(\iota(h, V)) = (1-\beta') F^+_{i(h)} + d_{i(h)}^+ V + \frac{1}{2} \left[ V \wedge V \right]^+.
\end{equation}
\end{lemma}

\begin{proof}
This follows from the identity $F_{i(h) + V}^+ = F_{i(h)}^+ + d_{i(h)}^+ V + \frac{1}{2} \left[ V \wedge V \right]^+$ and the fact that $p_T(\iota(h, V)) = h$. 
\end{proof}

\begin{remark}\label{rem:noneq}
Unfortunately, when $G$ is not abelian, the $\mASD$ operator $s$ is not generally well-behaved under any suitable gauge group; e.g., it is not equivariant relative to the action of the gauge group of Section \ref{sec:GaugeTheory}. The issue is that the term $F_{i(p_T(A))}$ is gauge equivariant relative to the \emph{trivial} $G$-action on $\mathfrak{g}$, while the term $F_A$ is gauge equivariant relative to the \emph{adjoint} $G$-action on $\mathfrak{g}$. Consequently, any non-trivial linear combination of these (e.g., as in the above formula for $s$) is not equivariant relative to either $G$-action. This issue is apparent even in the smooth compactly supported setting, and hence persists regardless of which Sobolev completion we choose. See Section \ref{sec:mmrerror} for some comments about the effect of this issue on the results of \cite{\mmr}.
\end{remark}

\subsubsection{A Coulomb slice}

To obtain a Fredholm operator, we will restrict the operator $s$ to a Coulomb (gauge) slice
$$\Sl(A') \defeq \left\{\iota(h, V) \; \Big| \; h \in \mathcal{H} \; \; \textrm{and}\;\; V \in \ker(d_{A'}^{*, \delta})\subseteq L^p_{1, \delta}(\Omega^1)  \right\}$$
for some fixed connection $A'$. Here $d^{*, \delta}_{A} = e^{-t \delta} d_{A}^* e^{t \delta}$ is the adjoint relative to the $L^2_\delta$-inner product. For more details on this slice, see \cite[Prop.~10.3.1]{\mmr}. Since $\delta/2$ is not in the spectrum of $-*d_\Gamma$ on 1-forms, it follows from \cite[Lem. 8.3.1]{\mmr} that the restriction $s \vert_{\Sl(A')}$ of $s$ to this slice is a Fredholm map.

 We set
$$\rM = \rM(\mathcal{T}_\Gamma, A') \defeq s^{-1}(0) \cap \Sl(A')$$
which we refer to as the \emph{space of $\mASD$-connections}. For us, this will play the role that the $\ASD$ moduli space usually plays in the closed setting (though, as discussed in the introduction, this is less than satisfying for global considerations due to its dependence on $A'$). Elliptic regularity implies that any element of $\rM$ has regularity $\C^m$.

Consider the restriction of the $\mASD$ operator $s$ to this slice $\Sl(A')$. Then the linearization at $A \in \Sl(A')$ of this restriction is a bounded linear map 
$$\big(D s \vert_{\Sl(A')}\big)_A \colon  T_A \Sl(A') \longrightarrow L^p_\delta(\Omega^+).$$
We will say that an $\mASD$ connection $A$ is \emph{{$A'$-}regular} if this operator is surjective. {When $A' = A$, we drop the $A'$ and say that $A$ is \emph{regular} if it is $A$-regular.} We will also be interested in connections that are not regular, and for these we will need to consider the cokernel
\eqncount\begin{equation}\label{eq:H+def}
H^+_{A, \delta} \defeq \mathrm{coker} \; \big(D s \vert_{\Sl(A)}\big)_A.
\end{equation}
Clearly $H^+_{A, \delta} = 0$ if and only if $A$ is regular. 

We will write
$$\rM_{reg} = \rM_{reg}(\mathcal{T}_\Gamma, A') \subseteq \rM(\mathcal{T}_\Gamma, A')$$
for the subset of $A'$-regular $\mASD$ connections. It follows from the implicit function theorem that $\rM_{reg}$ is a $\C^m$-smooth manifold. 

\subsubsection{Navigating Morgan--Mrowka--Ruberman \cite{\mmr}}\label{sec:mmrerror} 

Since the operator $s$ is not gauge equivariant, we do not currently know how to define a suitable gauge quotient of $s^{-1}(0) \subset \A^{1,p}(\mathcal{T}_\Gamma) $ in order to obtain the \mASD\ moduli space $\cM_{\ell, w}(\mathcal{T}_\Gamma, T_0, g)$ envisaged in \cite[p.~125]{\mmr}. Though this appears to be a significant error in \cite{\mmr}, the major results of \cite{\mmr} remain intact. For the reader's benefit, we will now review \cite{\mmr}, highlighting those areas that need adjustment. In this section, we refer freely to the notation established in \cite{\mmr}, and the phrase ``thickened moduli space'' will refer to any of the following spaces
$$\cM_{\ell, w}(\mathcal{T}_\Gamma, T_0, g), \hspace{1cm} \cM_{\ell, w}^0(\mathcal{T}_\Gamma, T_0, g), \hspace{1cm} \cM_{\ell, w}(h, \mathcal{T}_\Gamma, T_0, g)$$
(note that sometimes the $w$ or the $g$ are dropped from the notation in \cite{\mmr}).

The thickened moduli space does not appear in any significant way in Chapters 1--6, 10--13, 15--16, and the results therein remain intact as stated. In Chapters 7--9, it is often the case that claims about the thickened moduli space remain intact if one interprets the term as meaning an object defined by a gauge slice as opposed to a quotient by the gauge group (e.g., replace the based version $\cM_{\ell, w}^0(\mathcal{T}_\Gamma, T_0, g)$ by what we called $\rM(\mathcal{T}_\Gamma, A')$ above). Such is the case with the patching results of \S~7.4, the index calculations of Ch. 8 (more details below), and most of the generic metrics results of Ch.~9. However, it is not immediately clear to us how best to interpret the claims in Lemma 7.5.3 and \S 9.4 since they refer to the $\mu$-map on the thickened moduli space, which is an inherently global object.

At first sight, Chapter 14 appears to have issues but, upon closer inspection, the results of this chapter remain intact as well, with the possible exception of the claims of \cite[Remark 14.0.5]{\mmr} (which don't appear to be used elsewhere in \cite{\mmr}). Here are some more details regarding Chapter 14: The main goal of this chapter is to prove Theorem 14.0.1, which is a structural result that gives dimension formulas for various $\ASD$ moduli spaces. This theorem has two cases. In the situation of Case I, as the authors point out explicitly (p.~202), every $\mASD$ connection is actually $\ASD$ and so the ``thickened moduli space'' is indeed well-defined because it is really (an open set of) the $\ASD$ moduli space (this coincidence is one that we too exploit, and is discussed more in Section \ref{sec:ProofOfThmA}). Thus, there is no problem with this case. The remaining Case II in Theorem 14.0.1 does not appeal to the thickened moduli space at all, and instead works with the space $\mathcal{B}^0_{\ell, \delta}(X)$ (despite being a gauge group quotient, the space $\mathcal{B}^0_{\ell, \delta}(X)$ is indeed well-defined as it does not make reference to the $\mASD$ operator $s$). Thus, the statement and proof of Theorem 14.0.1 need no adjustment and appear to be correct as written. However, we are not sure how to interpret the argument of Remark 14.0.5, which uses the thickened moduli space to deduce a Whitney stratification on the $\ASD$-moduli space. In the setting of this remark, there are $\mASD$ connections that are not $\ASD$, so the above-referenced coincidence does not apply.

\medskip

We end this section with a more detailed discussion of Chapter 8 in \cite{\mmr}, with the aim of salvaging the index calculation of Proposition 8.5.1. As in \cite[Ch. 8]{\mmr}, in addition to our usual hypotheses on $\delta$, we also assume that $\delta < 2 \mu_\Gamma^-$. We also restrict to the case where $G = \SO(3)$ or $G = \SU(2)$; the index for more general compact $G$ can be computed using the strategy outlined in \cite[Section 7.1]{donaldson-kronheimer1} (e.g., when $G$ is simple and simply-connected, use the data from \cite[Table 8.1]{AHS} to pin down the constants specific to $G$).

The first issue appears on p.~139, since the image of the map $D_\omega m$ is not generally contained in the kernel of $D_\omega s$ (this is the linear version of the fact that $s$ is not gauge-equivariant). As a consequence, $E_\delta(\omega)$ is not a complex in the usual sense. Nevertheless, much of what is desired of $E_\delta(\omega)$ can be salvaged by ``wrapping it up" and considering the operator
$$D\defeq (d_\omega^{*, \delta}, D_\omega s): T_\omega \A_{\ell, w} (\T_\Gamma, T_0) \longrightarrow T_e \G_\delta(\Gamma) \times \Omega^2_{+, 1, \delta}(X)$$ 
where $d_\omega^{*, \delta} = e^{-t \delta} d_\omega^* e^{t \delta}$ is the $L^2_\delta$-adjoint of $d_\omega$ and $D_\omega s$ is the notation used in \cite{\mmr} for the linearization of $s$ at a connection $\omega$ (what they call $\omega$ is what we call $A$). Then Proposition 8.5.1 can be interpreted as saying that this operator $D$ is Fredholm with index given by the formula:
$$Ind(D) = 8 \ell - \frac{3}{2}(\chi(X) + \sigma(X)) + \frac{h_\Gamma^1 - h_\Gamma^0}{2} + \frac{\rho(\Gamma)}{2}.$$
There is a ``based'' version of this that replaces $\G_\delta(\Gamma)$ by the normal subgroup $\G_\delta$. Wrapping in this case produces a map
$$D^0\defeq (d_\omega^{*, \delta}, D_\omega s): T_\omega \A_{\ell, w} (\T_\Gamma, T_0) \longrightarrow T_e \G_\delta \times \Omega^2_{+, 1, \delta}(X).$$
This operator is also Fredholm and, in light of the sequence (\ref{eq:gaugesequence}), its index is given by
$$Ind(D^0) = Ind(D) + h_\Gamma^0 = 8 \ell - \frac{3}{2}(\chi(X) + \sigma(X)) + \frac{h_\Gamma^1 + h_\Gamma^0}{2} + \frac{\rho(\Gamma)}{2}$$
since $h_\Gamma^0 = \mathrm{dim}(\mathrm{Stab}(\Gamma))$. Restricting $D^0$ to a slice $\Sl(A')$ and projecting to the $\Omega^2_{+, 1, \delta}(X)$-component, we obtain the operator that we called $(D s \vert_{\Sl(A')})_A$ above. When $A = A'$ and this connection is flat down the end, it follows readily that the index of $(D s \vert_{\Sl(A)})_A$ equals that of $D^0$. Since the index remains unchanged under addition of compact operators, it follows that 
\eqncount\begin{equation}\label{eq:indexcomp}
Ind((D s \vert_{\Sl(A')})_A) = 8 \ell - \frac{3}{2}(\chi(X) + \sigma(X)) + \frac{h_\Gamma^1 + h_\Gamma^0}{2} + \frac{\rho(\Gamma)}{2}
\end{equation}
for all $A \in \Sl(A')$. The right-hand side of (\ref{eq:indexcomp}) is the expected dimension of the $\mASD$ space $\rM(\mathcal{T}_\Gamma, A')$, and it is the actual dimension of the $\C^m$-manifold $\rM_{reg}(\mathcal{T}_\Gamma, A')$ of $A'$-regular $\mASD$ connections.

\subsection{Special cases}\label{sec:SpecialCases}

\subsubsection{Flat connections}\label{sec:FlatConnections}

In this section, we study the linearized operator $(Ds)_A$ and its cokernel in the special case when $A$ is flat {and asymptotic to $\Gamma$}. To simplify the discussion, we assume $A$ is in temporal gauge on the end \cite[p.~15]{donaldson10} (though we continue to work in the general setting where the metric is asymptotically cylindrical). It follows that, for each $t \geq T$, the restriction $A \vert_{\left\{t \right\} \times N} = \Gamma$ is constantly equal to $\Gamma$ on the end. Then $A \in\A^{1,p}(\mathcal{T}_\Gamma) $ and $p(A) = \Gamma$. The associated flow $\alpha(\Gamma) = A$ recovers the flat connection $A$ on $\End  X$. This implies $A$ is $\mASD$.

The operator $s$ is defined in terms of the map $\iota$, and we recall that the definition of $\iota$ required the choice of a reference connection $A_{ref}$ on $X$. {Likewise, the space of $\mASD$ connections is defined by restricting $s$ to a slice $\Sl(A')$ for some choice of connection $A'$. It is convenient to take $A_{ref} \defeq A$ and $A' \defeq A$}; the reader can check that any other choice of $A_{ref}$ does not affect the outcome of the discussion that follows, though different choices of $A'$ may. In particular, our choice of $A_{ref}$ gives
$$A =i(\Gamma) = \iota(\Gamma, 0)$$

Let $b^+(X, A)$ be the dimension of a maximal positive definite subspace for the pairing map $q_A \colon  \hat{H}^2(\overline{X}, \mathrm{ad}(A)) \otimes \hat{H}^2(\overline{X}, \mathrm{ad}(A)) \rightarrow \bb{R}$, as in \cite[Section 8.7]{\mmr}, where $\overline{X}$ is the natural compactification of $X$ obtained by adding a copy of $N$ at infinity. For example, when $A= A_{\mathrm{triv}}$ is the trivial connection on the trivial $G$-bundle, then $b^+(X, A_{\mathrm{triv}}) = \dim(G) b^+(X)$ is a multiple of the usual self-dual Betti number of $X$. We will need the following result.

\begin{proposition}\label{prop:index}
Assume $0<\delta/2 < \mu_\Gamma^-$, where $\mu_\Gamma^-$ is as in Section \ref{sec:AuxiliaryChoices}. Then the cokernel $H^+_{A, \delta} = (Ds \vert_{\Sl(A)})_A$ has dimension $b^+(X, A)$.
\end{proposition} 

This is proved in \cite[Prop.~8.7.1(4)]{\mmr}, however the discussion there does not deal with the linearized operator $(Ds)_A$ directly. In preparation for our gluing arguments below, we will summarize the argument given in \cite[Prop.~8.7.1(4)]{\mmr}, but from the present perspective. Our proof is sketched below, after we give some preliminary computations that will be useful in their own right. 

The restriction $\iota\vert\colon  \cH \times \ker(d^{*, \delta}_{A}) \rightarrow \Sl(A)$ is a diffeomorphism, essentially by definition. To understand the cokernel $H^+_{A, \delta}$ it suffices to understand the cokernel of the linearization of $s \circ \iota \vert$. Towards this end, differentiating (\ref{eq:formfors}) at $A = \iota(\Gamma, 0)$ in the direction of $(\eta, V) \in T_\Gamma \cH_{out} \times \ker(d^{*, \delta}_{A})$ gives
\eqncount\begin{equation}\label{eq:dsa}
(D s)_A \circ (D \iota)_{(\Gamma, 0)}(\eta, V) = (1- \beta') d_A^+ (D i)_\Gamma \eta + d_A^+ V.
\end{equation}
Next, it follows from Lemma \ref{lem:derofalpha} that $(Di)_\Gamma \eta = \beta'' \eta$. Since we also have $d_\Gamma \eta = 0$ for $\eta \in T_\Gamma \cH = H^1_\Gamma$, it follows that
\eqncount\begin{equation}\label{eq:daformulafordi}
d_A^+(D i)_\Gamma \eta  = (\partial_t \beta'') (dt \wedge \eta)^+ .
\end{equation}
This is zero everywhere except on $\left(T-1/2, T\right) \times N$, where it vanishes if and only if $\eta = 0$. Combining this with (\ref{eq:dsa}), we therefore have the formula
\eqncount\begin{equation}\label{eq:formforlin}
(D s)_A \circ (D \iota)_{(\Gamma, 0)}(\eta, V)  = (1-\beta') (\partial_t \beta'')(dt \wedge \eta)^+ + d_A^+ V.
\end{equation}
This shows that, relative to the coordinates afforded by $\iota$, the leading order term of $\big(D s \vert_{\Sl(A)}\big)_A$ is $d_A^+ \vert_{\ker(d_A^{*, \delta})}$. The remaining term is compactly-supported and of order zero.

\begin{proof}[Proof of Proposition \ref{prop:index} (sketch)]
A maximal positive definite subspace for $q_A$ can be realized as the space $H^+(X, \mathrm{ad}(A))$ of self-dual 2-forms $W \in L^2(\Omega^+(X, \mathrm{ad}(A)))$ satisfying $d_A W = 0$ and so that the restriction to any slice $\left\{t \right\} \times N$ has trivial $\Gamma$-harmonic part; note that this definition does not depend on $\delta$, but see also Lemma \ref{lem:deltadecay}. Setting 
$$D \defeq \big(D s \vert_{\Sl(A)}\big)_A$$ 
we can similarly represent the cokernel $H^+_{A, \delta}$ of $D$ as the $L^{2}_\delta$-orthogonal complement $(\mathrm{im}\:D)^{\perp, \delta}$ to the image. Then Proposition \ref{prop:index} follows by showing that the map
$$j\colon  (\mathrm{im}\: D)^{\perp, \delta} \longrightarrow H^+(X, \mathrm{ad}(A)) \hspace{2cm} W \longmapsto e^{\delta t} W$$
is well-defined and bijective. That the map $j$ is well-defined follows from the formula in (\ref{eq:formforlin}). Indeed, if $W \in  (\mathrm{im}\; D)^{\perp, \delta}$, then the identity (\ref{eq:dsa}) implies 
$$0 = (W, D \circ (D \iota)_{(\Gamma, 0)}(0, V))_\delta = (W, d_A^+ V)_\delta = (d_A^{*, \delta} W, V)_\delta$$
where $(\cdot, \cdot)_\delta$ is the $L^2_\delta$-inner product. This holds for all $V \in L^2_{1, \delta}(\Omega^1)$, so it follows that $d_A j(W) =  -e^{t \delta} * d_{A}^{*, \delta} W =   0$. Similarly, we have
$$ 0 = (W, D \circ (D \iota)_{(\Gamma, 0)}(\eta, 0))_\delta = (W, (1-\beta') (\partial_t \beta'')(dt \wedge \eta)^+)_\delta.$$ 
Since $(1- \beta') \partial_t \beta''$ is non-zero on the cylinder $\left(T-1/2, T\right) \times N$, and since $\eta \in T_\Gamma \cH = H^1_\Gamma$ is allowed to roam freely over the $\Gamma$-harmonic space, it follows that the harmonic part of $W \vert_{\left\{t\right\} \times N}$ must vanish for any $t \in \left(T-1/2, T\right)$. 

It is clear that the map $j$ is injective; this already gives $\dim(H^+_{A, \delta}) \leq b^+(X, A)$. Surjectivity of $j$ is equivalent to the reverse inequality holding, and this follows from a dimension count: By \cite[Prop.~8.7.1(4)]{\mmr} (which uses the assumption $0 < \delta/2 < \mu^-_\Gamma$), the integer $b^+(X, A)$ is equal to the dimension of the cokernel of $(Ds)_A: T_A \A^{1, p}(\mathcal{T}_\Gamma) \rightarrow L^p_\delta(\Omega^+)$. The result follows because cokernels are non-decreasing under domain restriction:
$$\begin{array}{rcl}
b^+(X, A) & = & \dim\; \mathrm{coker}\big((Ds)_A: T_A \A^{1, p}(\mathcal{T}_\Gamma) \rightarrow L^p_\delta(\Omega^+)\big) \\
& \leq & \dim \; \mathrm{coker}\big((Ds \vert_{\Sl(A)})_A: T_A \Sl(A) \rightarrow L^p_\delta(\Omega^+)\big)\\
& = & \dim \; H^+_{A, \delta}.
\end{array}$$
\end{proof}

\begin{remark}\label{rem:RegularityRemark}
Proposition \ref{prop:index} is the observation that makes the $\mASD$-operator---\emph{and not the $\ASD$-operator}---a viable candidate for our existence results, including Theorem \ref{thm:A} which is purely in the $\ASD$ setting. Indeed, consider the case where $b^+(X) = 0$ and $A$ is the trivial flat connection. Then $b^+(X, A) = b^+(X) = 0$ and so Proposition \ref{prop:index} implies $A$ is regular as an $\mASD$ connection. However, when $\Gamma$ is degenerate, the trivial flat connection $A$ is \emph{not} regular as an $\ASD$ connection: The proof just given shows that, for the operator (\ref{eq:formforlin}), the image of $\eta \mapsto (1-\beta') (\partial_t \beta'')(dt \wedge \eta)^+$ is \emph{not} contained in the image of $d^+_A$. Thus, the additional degree of freedom afforded by $\eta \in T_\Gamma \cH$ in (\ref{eq:formforlin}) is necessary (and sufficient) to obtain surjectivity of $(Ds \vert_{\Sl(A)})_{A}$ when $A$ is the trivial flat connection, $\Gamma$ is degenerate, and $b^+(X) = 0$.
\end{remark}

We end with the following exponential decay estimate that we will use in Section \ref{sec:b^+=1}.

\begin{lemma}\label{lem:deltadecay}
For each $W \in H^+(X, \mathrm{ad}(A))$, there is some $C$ so that the restriction $W\vert_{\left\{t \right\} \times N}$ satisfies
$$\Vert W\vert_{\left\{t \right\} \times N} \Vert_{\mathcal{C}^0(N)} \leq C e^{- \mu_\Gamma^{+} t}$$
for all $t \geq 0$. In particular, $H^+(X, \mathrm{ad}(A)) \subseteq L^2_\delta(X)$ for any $\delta < 2\mu_\Gamma^{+}  $. 
\end{lemma}

\begin{proof}
It suffices to establish the estimate of the lemma under the assumption that the metric $g$ is cylindrical. Since $A$ is in temporal gauge on the end, its covariant derivative decomposes as $d_A = dt \wedge \partial_t + d_\Gamma$. Standard elliptic estimates for the operator $\Delta_\Gamma = d_\Gamma^* d_\Gamma + d_\Gamma d_\Gamma^*$ on $N$ provide a uniform constant $C$ so that
$$\Vert v \Vert_{\C^0(N)} \leq C \big( \Vert v \Vert_{L^2(N)}  + \Vert \Delta_\Gamma v \Vert_{L^2(N)} \big)$$
for all $v \in \Omega^1(N)$. 

Fix $W \in H^+(X, \mathrm{ad}(A))$. On $\End  X$, we can write $W = dt \wedge v + *_N v $ for some path $v = v(t) \in \Omega^1(N)$ of 1-forms. The condition $d_A W = 0$ implies $d_\Gamma *_N v = 0$ and $*_N d_\Gamma v = \partial_t v$. In particular, the above elliptic estimate implies
$$\Vert v(t) \Vert_{\mathcal{C}^0(N)} \leq C \big( \Vert v(t) \Vert_{L^2(N)}  + \Vert \partial_t^2 v(t) \Vert_{L^2(N)}\big).$$
It suffices to show that $f(t) \defeq \Vert v(t) \Vert_{L^2(N)}^2 + \Vert \partial_t^2 v(t) \Vert_{L^2(N)}^2$ decays exponentially in $t$ at a rate of $ 2\mu_\Gamma^+$. To see this, differentiate twice to get
$$\begin{array}{rcl}
 f''(t) & = & 2\Vert \partial_t v(t) \Vert_{L^2(N)}^2 + 2\Vert \partial_t^3 v(t) \Vert_{L^2(N)}^2 + 2 \big(\partial_t^2 v(t), v(t)\big) + 2\big(\partial_t^4 v(t), \partial_t^2 v(t)\big)\\
& = & 4\Vert d_\Gamma v(t) \Vert_{L^2(N)}^2 + 4\Vert d_\Gamma \partial_t^2 v(t) \Vert_{L^2(N)}^2
\end{array}$$
where we used $\partial_t v = *_N d_\Gamma  v $ and integration by parts. By definition of $H^+(X, \mathrm{ad}(A))$, $v(t)$ is orthogonal to the $\Gamma$-harmonic space and $d_\Gamma *_N v(t) = 0$. It follows that $v(t)$ lies in the image of $*_N d_\Gamma$. Moreover, the 2-form $W$, and hence $v$, is in $L^2$, by definition. This combines with the equation $*_N d_\Gamma v = \partial_t v$ to imply that $v(t)$ lies in the span of the negative eigenspaces of $*_N d_\Gamma $: express $v(t)$ as a $t$-dependent linear combination of an orthonormal basis of eigenvectors, use a separation of variables argument to show that each coefficient decays or grows down the end according to whether the eigenvalue is negative or positive, and then use the fact that $v$ is in $L^2$ and so must decay down the end.

Likewise, $\partial^2_t v(t)$ always lies in the negative eigenspace of $*_N d_\Gamma$. Recall that $\mu_\Gamma^{{+}}$ is the smallest positive eigenvalue of $-*_N d_\Gamma$. Thus, it is the absolute value of the largest negative eigenvalue of $*_N d_\Gamma$, and so
$$f''(t) = 4\Vert d_\Gamma v(t) \Vert_{L^2(N)}^2 + 4\Vert d_\Gamma \partial_t^2 v(t) \Vert_{L^2(N)}^2 \geq 4(\mu_\Gamma^{+})^2 f(t).$$
Since $f(t)$ is non-negative and converges to 0 as $t$ approaches $\infty$, it follows from this estimate that $f(t) \leq C e^{- 2 \mu_\Gamma^{+} t}$ (e.g., see \cite[p.~623]{DS}).
\end{proof}

\subsubsection{Non-degenerate $\Gamma$} \label{sec:Non-Deg}

 Here we assume that $\Gamma$ is non-degenerate in the sense that the harmonic space $H^1_\Gamma = \left\{0 \right\}$ is trivial. Then any center manifold necessarily consists of a single point and so there is a unique choice of cutoff function $\beta$. Then the $\mASD$ operator is the $\ASD$ operator. Moreover, non-degeneracy implies that any finite-energy $\ASD$ connection asymptotic to $\Gamma$ decays exponentially on the end at a rate of $e^{-\delta t}$ for any $ \delta/2 <  \mu_\Gamma^-$, where $\mu_\Gamma^-$ is as in Section \ref{sec:AuxiliaryChoices} (the proof of this assertion is similar to the proof of Lemma \ref{lem:deltadecay}, but $\mu_\Gamma^-$ appears in the present setting since the curvature is \emph{anti}-self dual). We also assumed that $\delta$ is greater than $\mu_\Gamma^-$, but that was only to construct the map $p_T$, whose existence is trivial in the non-degenerate setting since the center manifold consists of a single point. Thus, in the non-degenerate case this lower bound restriction on $\delta$ can be dropped, though we still need to retain the assumption that $\delta / 2$ is not an eigenvalue of $-*_N d_\Gamma$. In particular, it follows that whenever $-\mu_\Gamma^+ < \delta/2 <  \mu_\Gamma^- $, the resulting $\mASD/\ASD$ space is independent of the choice of this $\delta$, and the dimension of the space is given by (\ref{eq:indexcomp}) (apart from this paragraph, we always assume $\delta > 0$). 
 
 In summary, when $\Gamma$ is non-degenerate, there is an essentially canonical choice of thickening data $\mathcal{T}_{\Gamma, \mathrm{can}}$. Moreover, if $A'$ is any connection defining a slice, then the associated space of $\mASD$ connections
$$\rM(\mathcal{T}_{\Gamma, \mathrm{can}}, A') = \left\{ A \in \A(X) \; \Big| \; F_A^+ = 0, \;\;\; d^*_{A'} (A - A') = 0,  \;\;\;\lim_{t \rightarrow \infty} A \vert_{\left\{t \right\} \times N} = \Gamma \right\}$$
is the set of $\ASD$ connections in the $A'$-Coulomb slice that are asymptotic to $\Gamma$.

\subsubsection{Closed $X$} \label{sec:ClosedX}

Here we assume that $X$ is closed. View $X$ as a cylindrical end 4-manifold with an empty end. Then one can check that it makes sense to choose the empty set $\mathcal{T}_{\emptyset} = \emptyset$ of thickening data, and that, e.g., the $\mASD$ space $\rM(\mathcal{T}_{\emptyset}, A')$ is exactly the set of $\ASD$ connections on $E$ in the $A'$-Coulomb slice.

\section{Gluing two $\mASD$ connections}\label{sec:GenGluing}

Here we state and prove our first gluing result, which discusses gluing together $\mASD$ connections over the compact parts of two cylindrical end 4-manifolds. When the connections are not regular, the resulting connection may not be $\mASD$, and its failure to be $\mASD$ is captured by a suitable obstruction map. Our set-up is very similar to that of $\ASD$ gluing outlined in Donaldson--Kronheimer \cite[Section 7.2]{donaldson-kronheimer1}, to which we refer the reader for more details at various points. When introducing new terms for the analysis, we have tried to keep our notation as consistent with that of \cite{donaldson-kronheimer1} as possible. Our emphasis below will be on the new features that arise in the $\mASD$ setting. In the present section, the only serious new features arise from the fact that the $\mASD$ operator $s$ has a nonlinear term not present in the usual $\ASD$ setting; these features manifest themselves in the proofs of the claims appearing in the proof of Theorem \ref{thm:1}.

\subsection{Set-up for gluing}\label{sec:Setupforgluing}

Let $X_1$ and $X_2$ be oriented cylindrical end 4-manifolds equipped with asymptotically cylindrical metrics as in Section \ref{sec:TheChoiceOfMetric}. We will write $X_{k0}$ for compact part of $X_k$ and we set $N_k \defeq \partial X_{k0}$. We will need parameters $\lambda> 0$ and $L > 1$ so that $b \defeq  4 L \lambda^{1/2}  \ll 1$. The constant $L$ will be fixed later, but we will ultimately be interested in allowing $\lambda$ to be arbitrarily small. For each $k$, fix a point
$$x_k \in B_b(x_k) \subset \mathrm{int}(X_{k0})$$ 
in the interior of the compact part. To simplify the discussion, we assume that the metric on $X_{k}$ is flat over $B_b(x_k)$; see \cite[Section IV(vi)]{donaldsonb} for how to extend the discussion to handle more general metrics.

Following the approach in \cite[Section 7.2.1]{donaldson-kronheimer1}, we glue along the annuli
$$\Omega_k \defeq B_{L \lambda^{1/2}}(x_k)  \backslash B_{L^{-1} \lambda^{1/2}}(x_k)$$ 
using an ``inversion'' map $f_{\lambda}\colon  \Omega_1 \rightarrow \Omega_2$ to produce a connected sum
$$X = X(L, \lambda) \defeq \big(X_1 \backslash  B_{L^{-1} \lambda^{1/2}}(x_1)\big)  \cup_{f_{L,\lambda}} \big(X_2 \backslash  B_{L^{-1} \lambda^{1/2}}(x_2)\big).$$
Then $X$ is an oriented cylindrical end 4-manifold with asymptotic 3-manifold $N = N_1 \sqcup N_2$. We will write $X_0$ for the compact part of $X$; this is formed by analogously gluing the compact parts $X_{k0}$ of the $X_k$. The metrics on the $X_k$ can be glued to form a metric on $X$, and we assume this is done as outlined at the end of p.~293 in \cite{donaldson-kronheimer1}. We denote this metric by $g_{L, \lambda}$. Since we are interested in the limiting behavior of this for small $\lambda$, we will include the metric in the notation for our various norms and spaces of connections, forms, etc. whenever it is relevant.

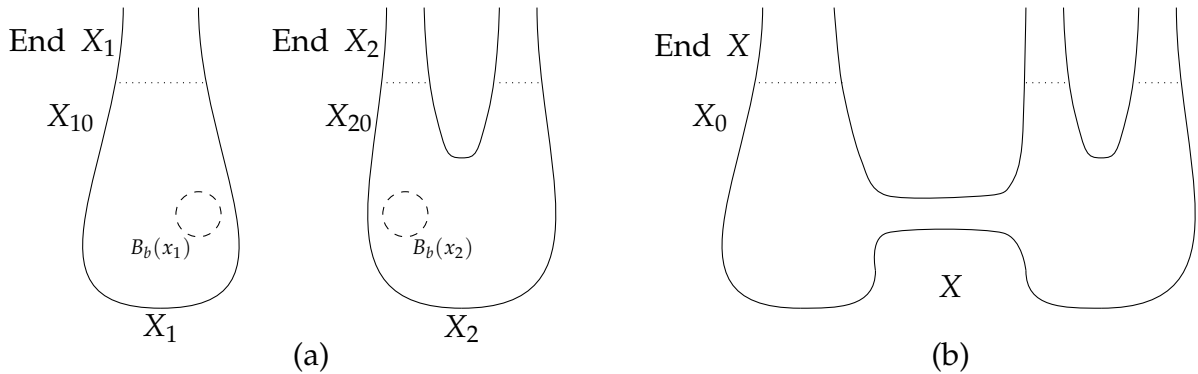
\begin{figure}[h]
\begin{tikzpicture}
\tikzstyle{every node} = [sloped,above] 
  \draw (0,-3).. controls +(left:2cm) and +(down:2cm)  ..(-.5,1);
   \draw (0,-3).. controls +(right:2cm) and +(down:2cm)  ..(.5,1);
   \draw[dashed] (.5,-1.75) circle (.3cm);
   \node[font=\tiny] at (0,-2.5) {$B_b(x_1)$};
   \draw [dotted] (-.55,0) -- (.6,0);
   \node at (-1.2,-.8) {$X_{10}$};
   \node at (-1.3,.2) {$\End  X_1$};
   \node at (0,-3.6) {$X_{1}$};
     \draw (4,-3).. controls +(left:2cm) and +(down:2cm)  ..(3,1);
     \draw plot [smooth, tension = 1] coordinates {(3.5,1) (3.65,-.4) (4,-1) (4.35,-.4) (4.5,1)};
   \draw (4,-3).. controls +(right:2cm) and +(down:2cm)  ..(5,1);
    \draw[dashed] (3.25,-1.75) circle (.3cm);
    \node[font = \tiny] at (3.75,-2.5) {$B_b(x_2)$};
    \draw [dotted] (3,0) -- (3.55,0);
    \draw [dotted] (4.5,0) -- (5.05,0);
       \node at (2.5,-.8) {$X_{20}$};
   \node at (2.2,.2) {$\End  X_2$};
    \node at (4,-3.6) {$X_{2}$};
    \node at (2,-4) {(a)};
\begin{scope}[shift={(8.5,0)}]
     \draw (0,-3).. controls +(left:2cm) and +(down:2cm)  ..(-.5,1);
     \draw (0,-3).. controls  +(right:.5cm) and +(down:.5cm) ..(1,-2.5);
     \draw [dotted] (-.55,0) -- (.6,0);
        \node at (-1.2,-.8) {$X_{0}$};
   \node at (-1.3,.2) {$\End  X$};
      \node at (2,-3) {$X$};
      \node at (2,-4) {(b)};
     \draw (4,-3).. controls  +(left:.5cm) and +(down:.5cm) ..(3,-2.5);
     \draw plot [smooth, tension = 1] coordinates {(3.5,1) (3.65,-.4) (4,-1) (4.35,-.4) (4.5,1)};
   \draw (4,-3).. controls +(right:2cm) and +(down:2cm)  ..(5,1);
       \draw [dotted] (3,0) -- (3.55,0);
    \draw [dotted] (4.5,0) -- (5.05,0);
        \draw plot [smooth] coordinates { (1,-2.5) (1.15,-2) (2.7,-2) (3,-2.5)};
         \draw plot [smooth] coordinates {(.45,1) (.55,0) (.8,-1) (1.15,-1.5) (2.5,-1.5) (2.8,-1.35)(2.95,-.8) (3,1)};
    \end{scope}
\end{tikzpicture}
    \caption{Illustrated above are the manifolds $X_1, X_2$ in (a), and their connected sum $X$ in (b). The 3-manifolds $N_1, N_2$, and $N$ are unlabeled, but are illustrated as dotted lines in the figure above.} 
    \label{fig:1}
    \end{figure}

Fix principal $G$-bundles $E_k \rightarrow X_k$ and flat connections $\Gamma_k \in \A(N_k)$ for $k = 1, 2$. These induce a bundle over $N$ as well as a flat connection $\Gamma$ on $N$. Fix $\delta > 0$ as in Section \ref{sec:ThickeningData} associated to this flat connection $\Gamma$. It follows that, for $k = 1, 2$, the quantity $\delta/ 2$ is not in the spectrum of $-*d_{\Gamma_k} $ on 1-forms. Let $\mathcal{T}_{k, \Gamma_k}$ be thickening data for $E_k$ with this $\delta$.

Fix an isomorphism $\rho\colon  (E_1)_{x_1} \rightarrow (E_2)_{x_2}$ of $G$-spaces, as well as flat connections $A_{\flat, k}$ for $E_k$ over $B_b(x_k)$. Using these flat connections and radial parallel transport, we can extend $\rho$ to a bundle isomorphism $E_1 \vert_{\Omega_1} \cong E_2 \vert_{\Omega_2}$ covering $f_{\lambda}$. It is with this bundle isomorphism that we glue the $E_k$ over the $\Omega_k$ to obtain a bundle 
$$E = E(\rho, L, \lambda)  \longrightarrow X(L, \lambda).$$
Since the gluing takes place away from the cylindrical end, the thickening data $\mathcal{T}_{k, \Gamma_k}$ for the $E_k$ induce thickening data $\mathcal{T}_\Gamma$ for $E$.

Fix $1 \leq  p < 4$ and suppose that, for each $k$, we have a smooth $\mASD$ connection 
$$A_k \in  \A^{1,p}(\mathcal{T}_{k, \Gamma_k})$$ 
on $X_k$. By performing the cutting off procedure described in Sections 7.2.1 and 4.4.5 of \cite{donaldson-kronheimer1}, we can form a connection $A_k'$ on $E_k$ that is equal to $A_k$ outside of the ball $B_{b}(x_k)$ and equal to the flat connection $A_{\flat, k}$ inside of the ball $B_{b/2}(x_k)$. Then the $A_k'$ patch together to determine a smooth connection $A' = A'(A_1, A_2)$ on $E$; this depends on $\rho, L, \lambda$ and the $A_k$. It follows that $A'$ is equal to $A_k$ in $X_k \backslash B_b(x_k) \subseteq X$ and that $A'$ is approximately $\mASD$:
\eqncount\begin{equation}\label{eq:our727}
\Vert s(A') \Vert_{L^p_\delta(X, g_{L, \lambda})}  \leq C_{(\ref{eq:our727})} b^{4/p}
\end{equation}
where $C_{(\ref{eq:our727})}$ is a constant independent of $L, \lambda$ (see (7.2.36) in \cite{donaldson-kronheimer1}). We will refer to $A'$ as the \emph{preglued connection}. We define the maps $i$ and $\iota$ of Section \ref{sec:TheSpaceOfConnections} by taking $A_{ref} \defeq A'$. 

\begin{remark}\label{rem:irreducible}
Assume $2< p < 4$ and set $p^* = 4p/(4-p)$. By \cite[Eq. (7.2.37)]{donaldson-kronheimer1}, as $b \rightarrow 0$, the connections $A'_k$ converge in $L^{p^*}_\delta$ to $A_k$. In particular, by Remark \ref{rem:irreduciblesmallp}, if $A_k$ is irreducible, then so too is $A'_k$ provided $b$ is sufficiently small. The stabilizer group of $A'$ is contained in that of $A'_k$ and so it follows that $A'$ is irreducible when either of $A_1$ or $A_2$ is irreducible and $b$ is sufficiently small.  
\end{remark}

On $X_k$, use the slice $\Sl(A_k)$ defined by $A_k$ and denote by $H_k^+ \defeq H^+_{A_k, \delta}$ the cokernel of the linearized operator $D_k \defeq (Ds\vert_{\Sl(A_k)})_{A_k}$, as in (\ref{eq:H+def}). As described in \cite[p.~290]{donaldson-kronheimer1}, we can choose lifts 
$$\sigma_k\colon  H^+_k \longrightarrow L^p_{\delta}(\Omega^+(X_k))$$
so that the operator $D_k \oplus \sigma_k$ is surjective. Moreover, we can do this in such a way that, for every $v \in H^+_k$, the form $\sigma_k(v)$ is supported in the complement of the ball $B_{2b}(x_k)$. Set 
$$H^+ \defeq H^+_1 \oplus H^+_2$$
and consider the linear map
$$\sigma \defeq \sigma_1 \oplus \sigma_2 \colon  H^+ \longrightarrow L^p_\delta(\Omega^+(X), g_{L, \lambda}).$$
Relative to the $L^p_\delta(X, g_{L, \lambda})$-norm on $H^+$, this map $\sigma$ is bounded with a bound independent of $L$ and $\lambda$.

\subsection{Gluing two connections}\label{sec:MainResultsForGluing}

The main result of this section is as follows.

\begin{theorem}\label{thm:1}
Assume $2 \leq  p < 4$ and set $p^* = 4p / (4 - p)$. Fix $\rho, \delta$, thickening data, and $\mASD$ connections $A_1$, $A_2$ as in Section \ref{sec:Setupforgluing}. Then there are constants $C, L, \lambda_0 > 0$ so that the following holds for each $0 < \lambda < \lambda_0$. 

Let $A' = A'(A_1, A_2)$ be the preglued connection constructed from $\rho, L, \lambda$, and the $A_k$.

\begin{itemize}
\item[(a)] There is a $\C^m$-map $J_{A_1, A_2} \colon  L^p_\delta(\Omega^+(X), g_{L, \lambda}) \rightarrow {\Sl(A')} \subseteq  \A^{1,p}(\mathcal{T}_\Gamma)$ that satisfies $J_{A_1, A_2}(0) = A'$. The first $m$ derivatives of $\xi \mapsto J_{A_1, A_2}(\xi)$ are bounded in operator norm by a bound that is independent of $\lambda$. 
\item[(b)] There is a linear map $\pi\colon  L^p_\delta(\Omega^+(X), g_{L, \lambda}) \rightarrow H^+$ satisfying $\sigma \circ \pi \circ \sigma  = \sigma$ and
$$\Vert \pi \xi \Vert_{H^+} \leq C \Vert \xi \Vert_{L^p_\delta(X, g_{L, \lambda})} \hspace{1cm} \forall \xi \in \Omega^+(X).$$
\item[(c)] There is a unique 2-form $\xi({A_1, A_2}) \in L^p_\delta(\Omega^+(X))$ so that
\eqncount\begin{equation}\label{eq:estypre}
 \Vert \xi(A_1, A_2) \Vert_{L^p_\delta(X, g_{L, \lambda})}  \leq  C b^{4/p} 
\end{equation}
and so that the connection $\mathcal{J}(A_1, A_2)  \defeq J_{A_1, A_2}(\xi(A_1, A_2))$ satisfies
\eqncount\begin{equation}\label{eq:estyp}
s(\mathcal{J}(A_1, A_2) )  =  -\sigma \pi \xi(A_1,A_2).
\end{equation}
\end{itemize}
In particular, for $k = 1, 2$ the connection $\mathcal{J}(A_1, A_2)$ is close to $A_k$ on $ X_k \backslash B_{L \lambda^{1/2}}(x_k) \subseteq X$ in the sense that
\eqncount\begin{equation}\label{eq:close}
\Vert \iota^{-1}(\mathcal{J}(A_1, A_2)) - \iota^{-1}(A_k) \Vert_{L^2_2(N_k) \times L^{p^*}_{\delta}(  X_k \backslash B_{L \lambda^{1/2}}(x_k))} \leq C b^{4/p}.
\end{equation}

If $A_1$ and $A_2$ are regular, then the connection $\mathcal{J}(A_1, A_2)$ is regular. In this case, $\mathcal{J}(A_1, A_2)$ is $\mASD$ and the maps $(A_1, A_2, \xi) \mapsto J_{A_1, A_2}(\xi)$ and $(A_1, A_2) \mapsto \xi(A_1, A_2)$ are both $\C^m$-smooth, relative to the specified topologies. 

If $p > 2 $ and either $A_1$ or $A_2$ is irreducible, then so is $\mathcal{J}(A_1, A_2)$.
\end{theorem}

Before getting to the proof, we briefly discuss the maps appearing in this theorem, and their analogues in the standard $\ASD$ theory; precise definitions of these maps are given in the proof, below. First, the most interesting part of the theorem is part (c), with parts (a) and (b) serving to set up (c). The map $\pi$ from (b) is a measure of the failure of $A'$ to be regular. It serves the same role here and enjoys the same properties as the map of the same name \cite[pp.~290---291]{donaldson-kronheimer1} in the $\ASD$ setting. As for $J$ in (a), this map is formed from a near-right inverse $P$ of $(Ds\vert_{\Sl(A')})_{A} \oplus \sigma$, pre-composed by an exponential map (see Claim 1 below for a precise statement). As an example, in the special case where $\Gamma$ is non-degenerate (so $\mASD = \ASD$) the relevant space of connections is an affine space, and this exponential map is simply given by the affine action. In this case, the map $J_{A_1, A_2}$ simplifies to $J_{A_1, A_2}(\xi) = A' + P\xi$, just as in the usual $\ASD$ setting \cite[p.~289]{donaldson-kronheimer1}. 

The object $\mathcal{J}(A_1, A_2)$ is the glued connection we are after. Equation (\ref{eq:estyp}) expresses the degree to which this connection is $\mASD$. In particular, the obstruction map mentioned in the introduction can be taken to be the map $(A_1, A_2) \mapsto \pi \circ \xi(A_1, A_2)$. Finally we mention that the diffeomorphism $\iota^{-1}$ appears in (\ref{eq:close}) only to make explicit the sense in which $\mathcal{J}(A_1, A_2)$ approximates the $A_k$ away from the gluing points.  

The proof of Theorem \ref{thm:1} that we adopt relies on the following two lemmas. 

\begin{lemma}\label{lem:dklem}
Let $\St\colon  B \rightarrow B$ be a $\C^{m}$-map on a Banach space $B$ with $\St(0) = 0$ and 
\eqncount\begin{equation}\label{eq:quad}
\Vert \St( \xi_1) - \St( \xi_2) \Vert \leq \kappa ( \Vert \xi_1 \Vert + \Vert \xi_2 \Vert) \Vert \xi_1 - \xi_2 \Vert,
\end{equation}
 for some $\kappa > 0$ and all $\xi_1, \xi_2 \in B_1(0) \subset B$ in the unit ball. Then for each $\eta \in B$ with $\Vert \eta \Vert < 1/(10\kappa)$, there is a unique $\xi \in B$ with $\Vert \xi \Vert \leq 1/ (5\kappa)$ such that
$$\xi + \St(\xi) = \eta.$$
Moreover, if $\eta = \eta(a)$ depends $\C^{m}$-smoothly on a parameter $a$, then the solution $\xi = \xi(a)$ depends $\C^{m}$-smoothly on this parameter as well. 
\end{lemma}

\begin{proof}
The existence and uniqueness claims follow from the contraction mapping principle and is carried out in \cite[Lemma 7.2.23]{donaldson-kronheimer1}. The $\C^{m}$-smooth dependence of $\xi$ on $a$ follows from, e.g., the discussion in \cite[Section I.5]{lang}. 
\end{proof}

The remaining lemma will be used to establish the nonlinear estimate (\ref{eq:quad}) in our $\mASD$ setting. 

\begin{lemma}\label{lem:quad1}
Assume $2 \leq p < 4$ and set $p^* = 4p / (4 - p)$. There is a constant $C_{(\ref{lem:quad1})}$ so that if $L, \lambda > 0 $ are any constants for which the connected sum $X$ is defined, then
$$\Vert fg \Vert_{L^p_\delta(X, g_{ L, \lambda})} \leq C_{(\ref{lem:quad1})} \Vert f \Vert_{L^{p^*}_\delta(X, g_{ L,\lambda})} \Vert g \Vert_{L^{p^*}_\delta(X, g_{L, \lambda})}$$
for all real-valued functions $f, g \in L^{p^*}_\delta(X)$. 
\end{lemma}

\begin{proof}
Writing $X = X_0 \cup \End  X$, it suffices to show that there is a uniform constant $C$ so that
$$\begin{array}{rcl}
\Vert fg \Vert_{L^p(X_0, g_{L, \lambda})} & \leq & C \Vert f \Vert_{L^{p^*}(X_0, g_{L,  \lambda)}} \Vert g \Vert_{L^{p^*}(X_0, g_{ L, \lambda})}\\
\Vert fg \Vert_{L^p_\delta(\End  X, g_{L,  \lambda})} & \leq & C \Vert f \Vert_{L^{p^*}_\delta(\End  X, g_{L, \lambda})} \Vert g \Vert_{L^{p^*}_\delta(\End  X, g_{L, \lambda})}.
\end{array}$$
We begin with the estimate over $\End  X$. Note that the metric $g_{L, \lambda}$ is independent of $L, \lambda$ over this region, so we do not need to worry about showing that any such constant $C$ is independent of $L, \lambda$. To obtain the estimate, use the definition of the $\delta$-dependent norms, together with H\"{o}lder's inequality to get
$$\begin{array}{rcl}
\Vert f g \Vert_{L^p_\delta(\End  X)}  &= & \Vert e^{t \delta/2 } fg \Vert_{L^p(\End  X)}\\
&= & \Vert (e^{-t \delta/2 }  e^{t \delta/2 }  f) (e^{t \delta/2 } g) \Vert_{L^p(\End  X)}\\
& \leq & \Vert e^{-t \delta/2 }  (e^{t \delta/2 }  f) \Vert_{L^4(\End  X)} \Vert e^{t \delta/2 } g \Vert_{L^{p^*}(\End  X)}
\end{array}$$
Since $2 \leq p < 4$, we have $4 \leq p^* < \infty$. Hence there is some $4 < r \leq \infty$ with $4^{-1} = r^{-1} + (p^*)^{-1}$. Then we can use H\"{o}lder's inequality again to continue the above:
$$\begin{array}{rcl}
\Vert f g \Vert_{L^p_\delta(\End  X)}  & \leq & \Vert e^{-t \delta/2 }  \Vert_{L^r(\End  X)} \Vert  f \Vert_{L^{p^*}_\delta(\End  X)} \Vert g \Vert_{L^{p^*}_\delta(\End  X)}.
\end{array}$$
Then the requisite estimate holds with $C = \Vert e^{-t \delta/2 }  \Vert_{L^r(\End  X)} $, which is plainly finite. 

As for the estimate over $X_0$, the same type of argument gives
$$\Vert fg \Vert_{L^p(X_0, g_{L, \lambda})}  \leq \mathrm{vol}(X_0, g_{L, \lambda})^{1/r} \Vert f \Vert_{L^{p^*}(X_0, g_{L,  \lambda)}} \Vert g \Vert_{L^{p^*}(X_0, g_{ L, \lambda})}.$$
As discussed on p.~293 of \cite{donaldson-kronheimer1}, the condition $p \geq 2$ implies that $\mathrm{vol}(X_0, g_{L, \lambda})$ can be taken to be independent of $L, \lambda$, provided $L \lambda^{1/2}$ is uniformly bounded from above (which is necessarily the case whenever the connected sum is defined).
\end{proof}

\begin{proof}[Proof of Theorem \ref{thm:1}]
Our intention is to apply Lemma \ref{lem:dklem}. To do this, we need to recast solving $s(A) = 0$ for $A$ into solving an equation for a self-mapping $\St$ of a Banach space. Ultimately, the Banach space will be the codomain of the $\mASD$ operator $s$, and $\St$ will essentially be the quadratic part of $s$. 

We begin this process by passing to a local chart on $\A^{1,p}(\mathcal{T}_\Gamma)$ (recall from Section \ref{sec:TheSpaceOfConnections} that this space of connections is generally \emph{not} an affine space). For this, write 
$$A' = \iota(h', V') = i(h') + V'$$ 
for $(h', V') \in \cH_{out} \times L^p_{1, \delta}(\Omega^1(X))$. Let $\exp_{h'}\colon  B_\eps(0) \subset T_{h'} \cH \rightarrow \cH$ be the exponential map associated to the $L^2_2(N)$-metric on $\cH \defeq \cH_\Gamma$; here $\eps> 0$ is small enough so that the exponential is well-defined. This is all taking place on the 3-manifold $N$, and so this exponential and this $\eps$ are manifestly independent of $L, \lambda$. Coupling this exponential on $\cH$ with the exponential on $\Omega^1(X)$ given by the affine action, we obtain a map
$$\begin{array}{rcl}
\exp_{(h', V')}\colon  B_\eps(0) \times L^p_{1, \delta}\big(\Omega^1(X)\big) & \longrightarrow & \cH \times L^p_{1, \delta}\big(\Omega^1(X)\big)\\
(\eta, V) & \longmapsto & \big(\exp_{h'}(\eta), V' + V\big).
\end{array}$$
The chart for $\A^{1, p}(\mathcal{T}_\Gamma)$ that we will use is $\iota \circ \exp_{(h', V')}$. {We note that $\iota \circ \exp_{(h', V')}$ identifies the ``slice'' $B_{\eps}(0) \times L^p_{1, \delta}(\ker(d_{A'}^{*, \delta}))$ with a neighborhood of $A'$ in $\Sl(A')$.} 

\begin{remark}
Throughout the proof that follows, we will work with the $L^2_2(N)$-norm on $T_{h'} \cH$; we will often not keep track of this in the notation. Note that this choice of norm is effectively immaterial since $\cH$ is finite-dimensional and so any two norms are equivalent, provided they are well-defined. 
\end{remark}

Consider the map
$$\begin{array}{rcl}
\widetilde{s} \colon  B_\eps(0) \times L^p_{1, \delta}\big(\ker(d_{A'}^{*, \delta})\big) & \longrightarrow & L^p_{\delta}\big(\Omega^+(X), g_{L, \lambda}\big)\\
(\eta,V) & \longmapsto & s\Big(\iota\big(\exp_{(h', V')}(\eta, V)\big)\Big)
\end{array}$$
which is the map $s\vert_{\Sl(A')}$ relative to the chart just described. This satisfies $\widetilde{s}(0, 0) = s(A')$ and so (\ref{eq:our727}) gives
\eqncount\begin{equation}\label{eq:bounds}
\Vert \widetilde{s}(0, 0) \Vert_{L^p(X, g_{L, \lambda})} \leq  C_{(\ref{eq:our727})} b^{4/p}.
\end{equation}
Write $(D \widetilde{s})_{(\eta, V)}$ for the linearization of $\widetilde{s}$ at $(\eta, V)$. As the following claim indicates, the definition of $\sigma$ implies that the operator $(D \widetilde{s})_{(0, 0)} \oplus \sigma$ is surjective.

\medskip

\noindent \emph{Claim 1: For $2 \leq p < 4$, there are constants $C_{(\ref{eq:estyp0})}, \lambda_0 > 0$, and $L > 1$, as well as linear maps 
$$\begin{array}{rcl}
P \colon  L^p_\delta(\Omega^+(X), g_{L, \lambda}) & \longrightarrow & T_{{h'}} \cH \times L^p_{1, \delta}\big({\ker(d_{A'}^{*, \delta})}, g_{L, \lambda}\big) \\
\pi\colon  L^p_\delta(\Omega^+(X), g_{L, \lambda}) & \longrightarrow &H^+
\end{array}$$ 
so that $P \oplus \pi$ is a right inverse to $(D \widetilde{s})_{(0, 0)} \oplus \sigma$ that, for all $0 < \lambda < \lambda_0$, satisfies 
\eqncount\begin{equation}\label{eq:estyp0}
\Vert (P \oplus \pi)\xi \Vert_{(L^2_2(N) \times L^{p^*}_\delta(X, g_{L, \lambda})) \oplus L^p_\delta(X, g_{L, \lambda})}  \leq C_{(\ref{eq:estyp0})}\Vert \xi \Vert_{L^p_\delta(X,g_{L, \lambda})} \hspace{1cm} \forall \xi \in \Omega^+(X) .
\end{equation}
}

\medskip

This claim also has an extension to some $p < 2$; see Corollary \ref{cor:smallp0} for details. We will prove Claim 1 shortly. At the moment, we will show how we use it to finish the proof of the theorem. To prove Theorem \ref{thm:1} (a), set
$$J(\xi) =J_{A_1, A_2}(\xi) \defeq \iota\big(\exp_{(h', V')}(P \xi)\big) { \in \Sl(A')}.$$ 
Clearly $J(0) = A'$ and $J$ is an immersion near $0$. That the derivatives of $J$ are bounded uniformly ($\lambda$-independent) follows from the fact that $P$ is uniformly bounded (by the claim) and the fact that the maps $\iota$ and $\exp_{(h', V')}$ are uniformly bounded. The map in Theorem \ref{thm:1} (b) is the map $\pi$ from Claim 1. 

We now prove Theorem \ref{thm:1} (c). Define $\St\colon  L^p_\delta(\Omega^+(X), g_{L, \lambda}) \rightarrow L^p_\delta(\Omega^+(X), g_{L, \lambda})$ by
$$\St(\xi) \defeq   \widetilde{s}(P\xi) - (D \widetilde{s})_{(0, 0)} P \xi  - \widetilde{s}(0, 0).$$
This is $\C^{m}$ and is the nonlinear part of the map $\widetilde{s} \circ P$. The following claim says that this map satisfies the requisite nonlinear estimates.

\medskip

\noindent \emph{Claim 2: The map $\St$ satisfies the hypotheses of Lemma \ref{lem:dklem} with a constant $\kappa$ that is independent of $0< \lambda < \lambda_0$.}
 
\medskip

Once again, we defer the proof of this claim until after we have finished the argument for Theorem \ref{thm:1}. It follows from Claim 2, Lemma \ref{lem:dklem}, and the estimate (\ref{eq:bounds}) that, provided $\lambda$ is sufficiently small, there is a unique $\xi = \xi({A_1, A_2})  \in  L^p_{\delta}(\Omega^+)$ satisfying
\eqncount\begin{equation}\label{eq:xisttilde}
\xi + \St(\xi) = - \widetilde{s}(0, 0) \indent \indent \textrm{and} \indent \indent \Vert \xi \Vert_{L^p_\delta(X)} \leq 1/(5\kappa).
\end{equation}
Setting $\Vert \cdot \Vert_{L^p_\delta} \defeq \Vert \cdot \Vert_{L^p_\delta(X, g_{L, \lambda})}$ and using (\ref{eq:quad}), we get
$$\Vert \xi \Vert_{L^p_\delta} \leq \Vert \widetilde{s}(0, 0) \Vert_{L^p_\delta} + \Vert \St(\xi) \Vert_{L^p_\delta} \leq C_{(\ref{eq:our727})} b^{4/p} + \frac{1}{5}\Vert \xi \Vert_{L^p_\delta}. $$
This implies the requisite estimate on $\xi$. Unraveling the definitions, we also have 
$$s(\mathcal{J}(A_1, A_2)) = \widetilde{s}(P \xi) = \widetilde{s}(0, 0) + (D \widetilde{s})_{(0, 0)} P \xi + \St(\xi) =  - \xi + (D \widetilde{s})_{(0, 0)} P \xi = - \sigma \pi \xi,$$
where $\mathcal{J}(A_1, A_2) \defeq J(\xi)$, by definition. This finishes the proof of (c). 

To prove (\ref{eq:close}), note that
$$\iota^{-1}(\mathcal{J}(A_1, A_2))- \iota^{-1}(A')   = \exp_{(h', V')}(P \xi) - \exp_{(h', V')}(0).$$
From our definition of the exponential appearing on the right, we see that the difference $\exp_{(h', V')}(P \xi) - \exp_{(h', V')}(0)$ equals $P \xi$ plus some higher order terms supported only on the center manifold component, and thus in a component with a norm independent of $\lambda$. In particular, since $\xi$ is small, these higher order terms can be uniformly controlled to yield a uniform first order estimate of the form
$$\Vert \iota^{-1}(\mathcal{J}(A_1, A_2)) - \iota^{-1}(A') \Vert_{L^2_2(N) \times L^{p^*}_{\delta}(X, g_{L, \lambda})} \leq C_1 \Vert P \xi \Vert_{L^2_2(N) \times L^{p^*}_{\delta}(X, g_{L, \lambda})}$$
We can then combine this with (\ref{eq:estyp0}) to get
\eqncount\begin{equation}\label{eq:closeinA}
\Vert \iota^{-1}(\mathcal{J}(A_1, A_2)) - \iota^{-1}(A') \Vert_{L^2_2(N) \times L^{p^*}_{\delta}(X, g_{L, \lambda})} \leq  C_1 C_{(\ref{eq:estyp0})} \Vert \xi \Vert_{L^p_\delta(X, g_{L, \lambda})} \leq C_2 b^{4/p}
\end{equation}
where the second inequality comes from the estimates of the previous paragraph. The estimate (\ref{eq:close}) follows from this and the fact that $A_k$ agrees with $A' = A'(A_1, A_2)$ on $X_k \backslash B_{L \lambda^{1/2}}(x_k)$.

When the $A_k$ are regular, the map $\pi$ is the zero map so $\mathcal{J}(A_1, A_2)$ is automatically $\mASD$ by (\ref{eq:estyp}). In this case, the operator $(D\iota)_{(h', V')} \circ P$ is a right inverse to $(Ds \vert_{\Sl(A')})_{A'}$, essentially by definition; in particular, $(Ds \vert_{\Sl(A')})_{A'}$ is surjective. {Thus $(Ds \vert_{\Sl(A)})_{A}$ is surjective whenever $A$ is close to $A'$. In particular, by (\ref{eq:closeinA}), this is the case with $A =  \mathcal{J}(A_1, A_2)$ and so $\mathcal{J}(A_1, A_2)$ is regular.} The $\C^m$-smooth dependence of $J$ on the $A_k$ follows from Remark \ref{rem:GaugeSlice} (a) below, and the $\C^m$-smoothness of $\xi$ follows from the $\C^m$-smoothness assertion of Lemma \ref{lem:dklem}. 

Finally, assume $A_1$ or $A_2$ is irreducible (assume $p > 2$ so this term is defined). It follows from Remark \ref{rem:irreducible} that $A'$ is irreducible as well. The irreducibility of $\mathcal{J}(A_1, A_2)$ then follows from (\ref{eq:closeinA}) and Lemma \ref{lem:irreducible}.

\medskip

To finish the proof of Theorem \ref{thm:1}, it therefore suffices to verify the claims; we begin with Claim 1. Let $\cH_k$ be the space $\cH_{out}$ for the connection $\Gamma_k$, and set
$$h_k = p_T(A_k) \in \cH_k$$
where $p_T$ is the map from Section \ref{sec:TheSpaceOfConnections}. Similar to what we did above over $X$, for each $k$, we can form a map 
$$\widetilde{s}_k\colon  T_{h_k} \cH_{k} \times L^p_{1, \delta} (\ker(d^{*, \delta}_{A_k})) \longrightarrow L^p_\delta(\Omega^+(X_k))$$ 
by precomposing $s$ with $\iota$ and the exponential $\exp_{h_k}$ for $\cH_{k}$ based at $h_k$. Linearizing at $(0, 0)$, and coupling with $\sigma_k$, we obtain a bounded linear map
$$\widetilde{D}_k  \defeq (D\widetilde{s}_k)_{(0, 0)} \oplus \sigma_k\colon  \Big(T_{h_k} \cH_{k} \times L^p_{1, \delta}\big(\ker(d^{*, \delta}_{A_k})\big)\Big) \oplus H^+_k \longrightarrow L^p_\delta(\Omega^+(X_k)).$$ 
This is surjective by the construction of $\sigma_k$. Standard elliptic theory for $\delta$-decaying spaces \cite{lockhart-mcowen} and the finite-dimensionality of $\cH_{k}$ imply that $\widetilde{D}_k$ restricts to a bounded map of the form 
$$\widetilde{D}_k   :\Big( T_{h_k} \cH_k \times L^p_{\ell+1,\delta}\big(\ker(d^{*, \delta}_{A_k})\big)\Big) \oplus H^+_k \longrightarrow L^p_{\ell,\delta}(\Omega^+)$$ 
for each $\ell \geq 0$, and this restriction remains surjective. In particular, the ``Laplacian'' 
$$ \widetilde{D}_k \widetilde{D}_k^{*, \delta}: L^p_{2, \delta}(\Omega^+) \longrightarrow L^p_{\delta}(\Omega^+)$$
is a Banach space isomorphism, where $\widetilde{D}_k^{*, \delta}$ is the adjoint of $\widetilde{D}_k$ relative to the $\delta$-decaying $L^2$-inner products on the domain and codomain. It follows from these observations that the formula
$$P_k \defeq \widetilde{D}_k^{*, \delta} \big( \widetilde{D}_k \widetilde{D}_k^{*, \delta}\big)^{-1}\colon  L^p_{\delta}(\Omega^+) \longrightarrow \Big( T_{h_k} \cH_k \times L^p_{1,\delta}\big(\ker(d^{*, \delta}_{A_k})\big)\Big) \oplus H^+_k$$
defines a bounded right inverse for $\widetilde{D}_k$. Coupling this with the embedding $L^{p}_{1, \delta} \hookrightarrow L^{p^*}_\delta$ it follows that there is a constant $c_k$ with
\eqncount\begin{equation}\label{eq:rightinverse}
\Vert P_k\xi \Vert_{(L^2_2(N_k) \times L^{p^*}_\delta(X_k) ) \oplus L^p_\delta(X_k)} \leq c_k \Vert \xi \Vert_{L^p_\delta(X_k)}, \hspace{.3cm} \forall \xi \in \Omega^+(X_k).
\end{equation}
The argument at this stage is almost identical to that given in \cite[Section 7.2.3]{donaldson-kronheimer1} (see also \cite[Prop.~7.2.18]{donaldson-kronheimer1}); however, we supply some of the details since we will refer to them below. Following \cite[p.~288]{donaldson-kronheimer1}, the operators $P_1, P_2$ can be glued together to produce an operator 
$$\Qt \colon  L^p_\delta\big(\Omega^+(X), g_{L, \lambda}\big)  \longrightarrow  \Big(T_{{h'}} \cH \times L^p_{1, \delta}\big(\Omega^1(X), g_{L, \lambda}\big) \Big) \oplus H^+$$
that satisfies
$$\Vert \Qt \xi \Vert_{(L^2_2(N) \times L^{p^*}_\delta(X, g_{L, \lambda}) )\oplus L^p_\delta(X, g_{L, \lambda})}  \leq (c_1 + c_2) \Vert \xi \Vert_{L^p_\delta(X,g_{L, \lambda})} \hspace{1cm} \forall \xi \in \Omega^+(X).$$
Somewhat more explicitly, we have
$$\Qt = \beta_1 P_1 \gamma_1 + \beta_2 P_2 \gamma_2$$
where the $\left\{\beta_1, \beta_2 \right\}$ and $\left\{\gamma_1, \gamma_2 \right\}$ are partitions of unity on $X$. The only property we need about these partitions is that the derivatives $\nabla \beta_k$ are supported in the gluing region and satisfy $\Vert \nabla \beta_k \Vert_{L^4(X)} \rightarrow 0$ as $L \rightarrow \infty$ and $b \rightarrow 0$. For future reference we note that since $A'$ equals the $A_k$ away from the gluing region and the $P_k$ take values in the $A_k$-slice, there is a uniform constant $C_{3}$ so that
\eqncount\begin{equation}\label{eq:betaestimate}
\Vert d_{A'}^{*, \delta} \Qt \xi \Vert_{L^p_\delta} \leq C_{3} ( \Vert \nabla \beta_1 \Vert_{L^4} + \Vert \nabla \beta_2 \Vert_{L^4} ) \Vert \Qt \xi \Vert_{L^{p^*}_\delta}.
\end{equation}
It also follows that $\Qt$ is an approximate right inverse to $(D \widetilde{s})_{(0, 0)} \oplus \sigma$ in the sense that 
$$\big((D \widetilde{s})_{(0, 0)} \oplus \sigma\big) \circ \Qt = I + \Rt$$
 for some $\Rt$ satisfying
$$\Vert \Rt(\xi) \Vert_{L^p_\delta} \leq \eps(L, b,p) \Vert \xi \Vert_{L^p_\delta}$$
where $\eps(L, b, p) \rightarrow 0$ as $L \rightarrow \infty$ and $b  \rightarrow 0$ (the assumption $p \geq 2$ is used here to establish this decay property for $\eps(L, b, p)$, see \cite[pp.~293,294]{donaldson-kronheimer1}). 

{At this stage, Donaldson and Kronheimer \cite{donaldson-kronheimer1} take their right inverse to be $\Qt(I + \Rt)^{-1}$. However, this is not sufficient for us since we want our right inverse to take values in the slice, and the operator $\Qt(I + \Rt)^{-1}$ generally does not. We thus want to modify the construction above, and we do this by simply projecting to the slice. Specifically, let $\Pi$ be the $L^2_\delta$-orthogonal projection of $\Omega^1(X)$ to the kernel of $d_{A'}^{*, \delta}$ (note that $d_{A'}^{*, \delta}$ is injective on the image of $I - \Pi$). This projection $\Pi$ extends to the codomain of $\Qt$ by acting on the $\Omega^1(X)$-factor only. Thus, the map
$$\widetilde{Q} \defeq \Pi \circ \Qt$$
takes values in the slice. We want to use this $\widetilde{Q}$ in place of $\Qt$, but for this we need to port the estimates on $\Qt$ over to $\widetilde{Q}$. To achieve this goal, note that the difference $\Qt - \widetilde{Q}$ takes values in the image of $I - \Pi$, on which the operator $d_{A'}^{*, \delta}$ is injective. From this we have the following Sobolev-type estimate
$$\Vert \Qt \xi -  \widetilde{Q} \xi \Vert_{(L^2_2(N) \times L^{p^*}_\delta(X, g_{L, \lambda}) )\oplus L^p_\delta(X, g_{L, \lambda})} \leq C_{4} \Vert d^{*, \delta}_{A'}(\Qt \xi -  \widetilde{Q} \xi ) \Vert_{L^p_\delta}$$
for all $\xi \in L^p_\delta(\Omega^+(X))$. As in \cite[Section 7.2.3]{donaldson-kronheimer1}, our range restriction on $p$ implies that this constant $C_{4}$ can be taken to be uniform in the neck-scaling parameter $b$. By construction, $\widetilde{Q}$ takes values in the kernel of $d^{*, \delta}_{A'}$ and so we can combine the above with (\ref{eq:betaestimate}) to get
\eqncount\begin{equation}\label{eq:restimatepre}
\Vert Q \xi - \widetilde{Q} \xi \Vert_{(L^2_2(N) \times L^{p^*}_\delta(X, g_{L, \lambda}) )\oplus L^p_\delta(X, g_{L, \lambda})} \leq C_{3}C_{4} ( \Vert \nabla \beta_1 \Vert_{L^4} + \Vert \nabla \beta_2 \Vert_{L^4}) \Vert Q \xi \Vert_{L^{p^*}_\delta}
\end{equation}
By taking $L$ large and $b$ sufficiently small, we can now transfer our estimate from $\Qt$ to $\widetilde{Q}$ to conclude that $\widetilde{Q}$ is uniformly bounded and is an approximate right inverse to $(D \widetilde{s})_{(0, 0)} \oplus \sigma$ in the sense that 
$$((D \widetilde{s})_{(0, 0)} \oplus \sigma) \circ \widetilde{Q} = I + \widetilde{R}$$
 for some $\widetilde{R}$ satisfying
 \eqncount\begin{equation}\label{eq:restimate}
\Vert \widetilde{R}(\xi) \Vert_{L^p_\delta} \leq \widetilde{\eps}(L, b,p) \Vert \xi \Vert_{L^p_\delta}
\end{equation}
where $\widetilde{\eps}(L, b, p) \rightarrow 0$ as $L \rightarrow \infty$ and $b  \rightarrow 0$. The whole point, of course, is that 
$$\widetilde{Q} \colon  L^p_\delta\big(\Omega^+(X), g_{L, \lambda}\big)  \longrightarrow  \Big(T_{{h'}} \cH \times L^p_{1, \delta}\big(\ker(d_{A'}^{*, \delta})\big) \Big) \oplus H^+$$
takes values in the slice where $\Qt$ may not have.}

Choose $L > 1, \lambda_0>0$ so that $\widetilde{\eps}(L, 4L \lambda_0^{1/2}, p) < 1/3$ and {$C_{3}C_{4} ( \Vert \nabla \beta_1 \Vert_{L^4} + \Vert \nabla \beta_2 \Vert_{L^4}) < 1$}. Then $\widetilde{Q}(I + \widetilde{R})^{-1}$ is a right inverse to $\widetilde{D} \oplus \sigma $ and has operator norm at most $3(c_1 + c_2+1)$. Then we can write this right inverse as $\widetilde{Q}(I + \widetilde{R})^{-1} = P \oplus \pi$, where the splitting corresponds to the direct sum decomposition of the codomain of $\widetilde{Q}$. The estimate (\ref{eq:estyp0}) immediately follows, so this finishes the proof of Claim 1. 

\begin{remark}\label{rem:GaugeSlice}\mbox{}
(a) It is not hard to show from the construction outlined in \cite{donaldson-kronheimer1} that, when the $A_k$ are regular, then the right inverse $P$ depends $\C^m$-smoothly on the $A_k$. The key observation here is that, though many choices have been made in this construction (e.g., cutoff functions), the only ones that depend on the $A_k$ are the choices of $\sigma_k$, but these can be taken to be zero when the $A_k$ are regular. 

\medskip

(b) {By construction, the connection $\mathcal{J}(A_1, A_2)$ of Theorem \ref{thm:1} naturally belongs to two slices: The one centered at itself, and the one centered at $A'$.}
\end{remark}

Now we move on to prove Claim 2. Fix $L, \lambda_0$ as in Claim 1 and we assume $\lambda \in (0, \lambda_0)$. We clearly have $\St(0) = 0$, so it suffices to show that $\St$ satisfies the quadratic estimate (\ref{eq:quad}) for a uniform constant $\kappa$. For this, note that by Lemma \ref{lem:regularityofalpha} and Taylor's Theorem, we can write
$$i(\exp_{h'}(\eta)) = i({h'}) + (D i)_{h'} \eta + q_{h'}(\eta)$$
where $q_{h'}\colon  T_{h'} \cH \rightarrow L^p_{1, loc}(X) \cap \C^0(X)$ vanishes to first order. Since $\cH_{out}$ is finite-dimensional, we can quantify this relative to any metric with respect to which the terms are well-defined. In particular, there is a constant $C_{(\ref{eq:qest})}$ so that 
\eqncount\begin{equation}\label{eq:qest}
\Vert q_{h'}(\eta_1) - q_{h'}(\eta_2)  \Vert_{\C^0(X)} \leq C_{(\ref{eq:qest})} ( \Vert \eta_1 \Vert_{L^2_{2}(N)} + \Vert \eta_2 \Vert_{L^2_2(N)})\Vert \eta_1 - \eta_2 \Vert_{L^2_2(N)}
\end{equation}
for all $\eta_1, \eta_2$ in the unit ball in $T_{h'} \cH$. Note that $q_{h'}(\eta)$ need not decay to zero down the ends of $X$ since $i({h'})$ and $i(\exp_{h'}(\eta))$ generally do not converge to the same connection at infinity. However, on the compact part we have
\eqncount\begin{equation}\label{eq:qvan}
q_{h'}(\eta) \vert_{X_0} = 0.
\end{equation}
Indeed, on $X_0$ the connection $i(h)$ equals the reference connection for all $h \in \cH_{out}$, and $i$ vanishes to all but the zeroth order on $X_0$. 

To verify (\ref{eq:quad}), fix $\xi_1, \xi_2 \in L^p_\delta(\Omega^+(X), g_{ \lambda})$ with $\Vert \xi_j \Vert_{L^p_\delta} \leq 1$ and set
$$(\eta_j, V_j) \defeq P \xi_j \in T_{{h'}} \cH \times L^p_{1, \delta}(\Omega^1(X)).$$
Then using the definition of $\St$ and the formula (\ref{eq:formfors}), we can write
\eqncount\begin{equation}\label{eq:expansionofQ}
\begin{array}{rcl}
\St(\xi_j) & = & \frac{1}{2} \left[ V_j \wedge V_j \right]^+  +  \frac{1- \beta'}{2} \left[\big( (D i)_{h'} \eta_j + q_{h'}(\eta_j)\big) \wedge \big((D i)_{h'} \eta_j + q_{h'}(\eta_j)\big) \right]^+  \\
&&+ (1- \beta')  d^+_{i({h'})} q_{h'}(\eta_j)  + \left[ V'  \wedge q_{h'}(\eta_j) \right]^+ \\
&&  + \left[  V_j \wedge q_{h'}(\eta_j)  \right]^+ +\left[  V_j \wedge  (D i)_{h'} \eta_j\right]^+. 
\end{array}
\end{equation}
(These are the higher order terms in the $\mASD$ operator $s$, expressed in terms of $V_j$ and $\eta_j$.) It suffices to show that each term on the right satisfies an estimate of the form (\ref{eq:quad}). Below we set $\Vert \cdot \Vert_{L^p_\delta} \defeq \Vert \cdot \Vert_{L^p_{\delta}(X, g_\lambda)}$. 

We begin with the first term on the right of (\ref{eq:expansionofQ}). This shows up in the $\ASD$ setting as well (see \cite[p.~289]{donaldson-kronheimer1}), but our argument is a bit more involved due to the non-compactness of $X$. For this, we use Lemma \ref{lem:quad1} to get
$$\begin{array}{rcl}
\frac{1}{2} \big\|  \left[ V_1 \wedge V_1 \right]^+ - \left[ V_2 \wedge V_2 \right]^+  \big\|_{L^p_\delta} & = & \frac{1}{2} \big\|  \left[ \big(V_1 + V_2\big) \wedge \big(V_1 - V_2\big) \right]^+ \big\|_{L^p_\delta}\\
& \leq & {\mathfrak{c}_{\mathfrak{g}}} \big\|  \vert V_1 + V_2\vert \vert V_1 - V_2\vert  \big\|_{L^p_\delta}\\
& \leq & {\mathfrak{c}_{\mathfrak{g}}}C_{(\ref{lem:quad1})} \Big( \Vert   V_1  \Vert_{L^{p^*}_\delta} + \Vert V_2 \Vert_{L^{p^*}_\delta} \Big)\Vert V_1 - V_2 \Vert_{L^{p^*}_\delta}
\end{array}$$
where $\mathfrak{c}_{\mathfrak{g}}$ is defined by 
\eqncount\begin{equation}\label{eq:defoffrakg}
\mathfrak{c}_{\mathfrak{g}}  \defeq \sup_{(\nu_1, \nu_2, \nu_3)} \vert \langle \nu_1, \left[ \nu_2, \nu_3 \right] \rangle \vert
\end{equation}
with the supremum running over all $\nu_j \in \mathfrak{g}$ with $\vert \nu_j \vert = 1$. Since $V_j$ is a component of $P\xi_j$, we can then use the estimate of Claim 1 to continue the above and get
$$\begin{array}{rcl}
\frac{1}{2} \big\|  \left[ V_1 \wedge V_1 \right]^+ - \left[ V_2 \wedge V_2 \right]^+  \big\|_{L^p_\delta} & \leq & {\mathfrak{c}_{\mathfrak{g}}} C_{(\ref{lem:quad1})} C_{(\ref{eq:estyp0})}^2 \Big( \Vert  \xi_1 \Vert_{L^{p}_\delta} + \Vert  \xi_2\Vert_{L^{p}_\delta} \Big)\Vert \xi_1 - \xi_2 \Vert_{L^{p}_\delta} 
\end{array}$$
which is the desired estimate.

Now we move on to the second term in (\ref{eq:expansionofQ}). Set $\rt(\eta) \defeq (D i)_{h'} \eta + q_{h'}(\eta) $, so we want to bound the $L^p_\delta$-norm of 
$$\begin{array}{l}
\frac{1- \beta'}{2} \Big(\left[ \rt(\eta_1) \wedge \rt(\eta_1) \right]^+  -\left[ \rt(\eta_2) \wedge \rt(\eta_2) \right]^+ \Big) = \frac{1- \beta'}{2} \left[ \big( \rt(\eta_1) + \rt(\eta_2)\big) \wedge \big(\rt(\eta_1) - \rt(\eta_2) \big) \right]^+ 
 \end{array}$$ 
in terms of the right-hand side of (\ref{eq:quad}). Note that this is supported on the compact cylinder $\mathrm{Cyl}_0 \defeq \left[T- 1/2, T + 1/2 \right] \times N$, and so its $L^p_\delta$-norm is bounded by a constant times
$$\begin{array}{l}
 \Big(\sum_{j = 1}^2 \Vert \rt(\eta_j) \Vert_{L^{p^*}_\delta( \mathrm{Cyl}_0)}  \Big) \Vert \rt(\eta_1) - \rt(\eta_2) \Vert_{L^{p^*}_\delta(\mathrm{Cyl}_0)} \\
\hspace{1cm}  \leq \Vert e^{t \delta/2} \Vert_{L^{p^*}(\mathrm{Cyl}_0)}^2\Big(\sum_{j = 1}^2  \Vert \rt(\eta_j) \Vert_{\C^0(\mathrm{Cyl}_0)}  \Big)\Vert \rt(\eta_1) - \rt(\eta_2) \Vert_{\C^0(\mathrm{Cyl}_0)} 
\end{array}$$
By Lemma \ref{lem:regularityofalpha}, this is bounded by a constant times
$$\begin{array}{l}
\Big(\Vert \eta_1 \Vert_{L^2_2(N)}   + \Vert \eta_2 \Vert_{L^2_2(N)}\Big)  \Vert \eta_1 - \eta_2 \Vert_{L^2_2(N)} \leq C_{(\ref{eq:estyp0})}^2 \Big( \Vert  \xi_1 \Vert_{L^{p}_\delta} + \Vert  \xi_2\Vert_{L^{p}_\delta} \Big)\Vert \xi_1 - \xi_2 \Vert_{L^{p}_\delta}
\end{array}$$
as desired.

The estimate for the third term $(1- \beta')  d^+_{i({h'})} q_{h'}(\eta_j) $ is similar and we leave it to the reader. Moving on to the fourth term in (\ref{eq:expansionofQ}), recall from (\ref{eq:qvan}) that $q_{h'}(\eta_j)$ vanishes on $X_0$. This observation combines with (\ref{eq:qest}) and then Claim 1 to give
$$\begin{array}{l}
\big\| \left[ V' \wedge q_{h'}(\eta_1) \right]^+   - \left[ V' \wedge q_{h'}(\eta_2) \right]^+   \big\|_{L^p_\delta} \\
\hspace{1cm} =  \big\| \left[ V' \wedge \big(q_{h'}(\eta_1) - q_{h'}(\eta_2)\big) \right]^+   \big\|_{L^p_\delta(\End  X)} \\
\hspace{1cm}  \leq   {\mathfrak{c}_{\mathfrak{g}}}\Vert V' \Vert_{L^p_\delta(\End  X)}  \Vert q_{h'}(\eta_1) - q_{h'}(\eta_2)  \Vert_{\C^0(\End  X)} \\
\hspace{1cm}  \leq  {\mathfrak{c}_{\mathfrak{g}}} C_{(\ref{eq:qest})} \Vert V' \Vert_{L^p_\delta(\End  X)} \Big( \Vert \eta_1 \Vert_{L^2_2(N)} + \Vert \eta_2 \Vert_{L^2_2(N)} \Big) \Vert \eta_1 - \eta_2 \Vert_{L^2_2(N)}\\
\hspace{1cm}  \leq  {\mathfrak{c}_{\mathfrak{g}}} C_{(\ref{eq:qest})} C_{(\ref{eq:estyp0})}^2 \Vert V' \Vert_{L^p_\delta(\End  X)} \Big( \Vert \xi_1 \Vert_{L^p_\delta} + \Vert \xi_2 \Vert_{L^p_\delta} \Big) \Vert \xi_1 - \xi_2 \Vert_{L^p_\delta}
\end{array}$$
This is the desired estimate for this term because $\Vert V' \Vert_{L^p_\delta(\End  X)}$ is plainly independent of $\lambda$ and the $\xi_j$. 

The remaining two terms are the most difficult to bound. This is because (i) these terms involve both the infinite-dimensional terms $V_i$ as well as the finite-dimensional terms $q_{h'}(\eta_j)$ and $(D i)_{h'} \eta_j$, and (ii) neither of these finite-dimensional terms generally decays to zero at infinity (nor do the differences $q_{h'}(\eta_1)- q_{h'}(\eta_2)$ and $(Di)_{h'} \eta_1 - (D i)_{h'} \eta_j$). The main estimate we need is the following, which we will see is equivalent to the fact that the operator $d_{A'}^+$ is Fredholm (on the appropriate spaces) with our choice of $\delta$.

\medskip

\noindent \emph{Claim 3: There is some $T_1 \gg T + 3/2 $ and a constant $C_{(\ref{eq:Lemma3})}$ so that 
\eqncount\begin{equation}\label{eq:Lemma3}
\Vert V \Vert_{L^p_\delta(\left[ T_1 , \infty \right) \times N)} \leq C_{(\ref{eq:Lemma3})} \Vert d_{A'}^+ V \Vert_{L^p_\delta(\left[ T_1-1 , \infty \right) \times N)}
\end{equation}
for all $V \in L^p_{1, \delta}(\Omega^1(X))$ with $d_{A'}^{*, \delta} V \vert_{\left[ T_1-1 , \infty \right) \times N} = 0$. }

\medskip

We prove Claim 3 after we finish our estimates for the last two terms in (\ref{eq:expansionofQ}). The argument we give applies to both of these last two terms, so we focus on establishing the estimate for the second-to-last term:
$$\begin{array}{l}
\left[  V_1 \wedge q_{h'}(\eta_1) \right]^+ - \left[  V_2 \wedge q_{h'}(\eta_2) \right]^+ \\
\hspace{1cm}= \frac{1}{2} \Big( \left[ \big(V_1 - V_2\big)  \wedge \big(q_{h'}(\eta_1) + q_{h'}(\eta_2)\big) \right]^+  +\left[ \big(V_1 + V_2 \big) \wedge \big(q_{h'}(\eta_1) - q_{h'}(\eta_2)\big) \right]^+    \Big) .
\end{array}$$
It suffices to bound the $L^p_\delta$-norm of each term on the right by the right-hand side of (\ref{eq:quad}); we will carry this out for $\left[ V_1 + V_2  \wedge q_{h'}(\eta_1) - q_{h'}(\eta_2) \right]^+  $, the other term is similar. Since the $q_{h'}(\eta_j)$ are supported on $\End  X$, we do not need to worry whether our constants are $\lambda$-dependent. With $T_1$ as in Claim 3, write
\eqncount\begin{equation}\label{eq:last}
\begin{array}{l}
 \big\| \left[ \big(V_1 + V_2\big)  \wedge \big(q_{h'}(\eta_1) - q_{h'}(\eta_2)\big) \right]^+ \big\|_{L^p_\delta} \\
 \hspace{1cm}  \leq  \Big( \big\| \left[ \big(V_1 + V_2 \big) \wedge \big(q_{h'}(\eta_1) - q_{h'}(\eta_2)\big) \right]^+ \big\|_{L^p_\delta(\left[0, T_1 \right] \times N)} \\
\hspace{4cm} + \big\| \left[ \big(V_1 + V_2\big)  \wedge \big(q_{h'}(\eta_1) - q_{h'}(\eta_2)\big) \right]^+ \big\|_{L^p_\delta(\left[T_1, \infty \right) \times N)} \Big).
\end{array}
\end{equation}
Set $\mathrm{Cyl}_1 \defeq \left[0, T_1 \right] \times N$ and estimate the first term on the right as follows:
$$\begin{array}{l}
\big\| \left[ \big(V_1 + V_2\big)  \wedge \big(q_{h'}(\eta_1) - q_{h'}(\eta_2)\big) \right]^+ \big\|_{L^p_\delta(\mathrm{Cyl}_1 )}\\
\hspace{1cm}  \leq  {\mathfrak{c}_{\mathfrak{g}}}  \Big(\Vert  V_1  \Vert_{L^{p^*}_\delta(\mathrm{Cyl}_1 )} +\Vert  V_2 \Vert_{L^{p^*}_\delta(\mathrm{Cyl}_1 )}\Big) \Vert  q_{h'}(\eta_1) - q_{h'}(\eta_2) \Vert_{L^{4}_\delta(\mathrm{Cyl}_1 )}\\
\hspace{1cm} \leq {\mathfrak{c}_{\mathfrak{g}}}  \Vert e^{t \delta/2} \Vert_{L^{4}(\mathrm{Cyl}_1)} \Big(\Vert  V_1  \Vert_{L^{p^*}_\delta} +\Vert  V_2 \Vert_{L^{p^*}_\delta}\Big) \Vert  q_{h'}(\eta_1) - q_{h'}(\eta_2) \Vert_{\C^0(\mathrm{Cyl}_1 )}\\
\hspace{1cm}  \leq {\mathfrak{c}_{\mathfrak{g}}} C_{(\ref{eq:qest})} C_{(\ref{eq:estyp0})}^3 \Vert e^{t \delta/2} \Vert_{L^{4}(\mathrm{Cyl}_1)} \Big(\Vert  \xi_1  \Vert_{L^{p}_\delta} +\Vert  \xi_2 \Vert_{L^{p}_\delta}\Big)^2 \Vert  \xi_1 - \xi_2  \Vert_{L^{p}_\delta}
 \end{array}$$
which is the desired estimate for this terms since we have assumed $\Vert  \xi_j \Vert_{L^{p}_\delta} \leq 1$, and so
$$(\Vert  \xi_1  \Vert_{L^{p}_\delta} +\Vert  \xi_2 \Vert_{L^{p}_\delta})^2 \leq 2(\Vert  \xi_1  \Vert_{L^{p}_\delta} +\Vert  \xi_2 \Vert_{L^{p}_\delta}).$$ 
As for the remaining term on the right of (\ref{eq:last}), note that it is bounded by a constant times
$$\begin{array}{l}
\Big( \Vert V_1 \Vert_{L^p_\delta(\left[T_1, \infty \right) \times N)}+ \Vert V_2 \Vert_{L^p_\delta(\left[T_1, \infty \right) \times N)}\Big) \Vert q_{h'}(\eta_1) - q_{h'}(\eta_2) \Vert_{\C^0(\left[T_1, \infty \right) \times N)}\\
\hspace{.75cm}  \leq C_{(\ref{eq:qest})} C_{(\ref{eq:estyp0})}^2 \Big( \Vert V_1 \Vert_{L^p_\delta(\left[T_1, \infty \right) \times N)}+ \Vert V_2 \Vert_{L^p_\delta(\left[T_1, \infty \right) \times N)}\Big) (\Vert  \xi_1  \Vert_{L^{p}_\delta} +\Vert  \xi_2 \Vert_{L^{p}_\delta}) \Vert  \xi_1 - \xi_2  \Vert_{L^{p}_\delta}
\end{array}$$
We will therefore be done with the proof of Claim 2 if we can show that the terms $\Vert V_j \Vert_{L^p_\delta(\left[T_1, \infty \right) \times N)}$ are bounded. For this, note that the linearization of $\widetilde{s}$ can be written as
$$\begin{array}{rcl}
(D\widetilde{s})_{(0, 0)} (\eta, V) & = & d^+_{A'}\big((D i)_{h'} \eta + V\big) - \beta' d^+_{i({h'})} (D i)_{h'} \eta\\
& = & d_{A'}^+  V + (1- \beta') d^+_{i({h'})} (D i)_{h'} \eta + \left[V' \wedge (Di)_{h'} \eta \right]^+.
\end{array}$$
Note also that, by the definition of $P$, the 1-form $V_j$ lies in the kernel of $d^{*, \delta}_{A'}$. We can therefore use Claim 3 and the above formula for $(D\widetilde{s})_{(0, 0)} $ to write
$$\begin{array}{rcl}
\Vert V_j \Vert_{L^p_\delta(\left[T_1, \infty \right) \times N)} & \leq & C_{(\ref{eq:Lemma3})} \Vert d^+_{A'} V_j \Vert_{L^p_\delta( \left[ T_1 - 1 , \infty \right) \times N)}\\
& \leq &  C_{(\ref{eq:Lemma3})} \Big( \Vert (D \widetilde{s})_{(0, 0)} (\eta_j, V_j) \Vert_{L^p_\delta} + \Vert \left[ V' \wedge (D i)_{{h'}} \eta_j \right]^+ \Vert_{L^p_\delta } \Big)
\end{array}$$
where we also used the fact that $\beta' = 1$ on $\left[ T_1 - 1 , \infty \right) \times N$. Since $(\eta_j, V_j) = P\xi_j$ and $P \oplus \pi$ is a right-inverse for $(D \widetilde{s})_{(0, 0)} \oplus \sigma$, we can continue this as
$$\begin{array}{rcl}
& \leq & C_{(\ref{eq:Lemma3})} \Big( \Vert \xi_j \Vert_{L^p_\delta} + \Vert \sigma \pi \xi_j \Vert_{L^p_\delta} +  {\mathfrak{c}_{\mathfrak{g}}} \Vert V' \Vert_{L^p_\delta} \Vert (D i)_{h'} \eta_j \Vert_{\C^0} \Big)\\
& \leq & C_{(\ref{eq:Lemma3})} \Big( 1 + C_{\sigma \pi} +  {\mathfrak{c}_{\mathfrak{g}}} C_{(Di)_{h'}} \Vert V' \Vert_{L^p_\delta} \Vert \eta_j \Vert_{L^2_2(N)} \Big)\\
& \leq & C_{(\ref{eq:Lemma3})} \Big( 1 + C_{\sigma \pi}+  {\mathfrak{c}_{\mathfrak{g}}}C_{(Di)_{h'}}C_{(\ref{eq:estyp0})}  \Vert V' \Vert_{L^p_\delta} \Big)
\end{array}$$
where $C_{\sigma \pi} $ and $C_{(Di)_{h'}}$ are the operator norms of $\sigma \pi$ and $(Di)_{h'}$, respectively (we are viewing the latter as a map $L^2_2(N) \rightarrow \C^0(X)$; see Lemma \ref{lem:regularityofalpha} and the definition of $i$). This is the uniform bound we are after, and thus finishes the proof of Claim 2. 

\medskip

Finally, we prove Claim 3. Let $h'_T \colon  \left[T, \infty \right) \rightarrow \cH_{out}$ denote the flow of the trimmed vector field $\Xi^{tr}$ with $h'_T(T) =  h'$. Let 
$$h'_\infty \defeq \lim_{t \rightarrow \infty} h'_T(t) \in \cH_{out}$$ 
be the limiting connection of this flow. This is a connection on $N$, but we will view it as a connection on $\End  X = \left[ 0, \infty \right) \times N$ that is constant in the $t$-direction. Let $\mathcal{X}$ be the $L^p_{1, \delta}$-completion of the space of 1-forms on $X$ supported on $\End  X$, and let $\mathcal{Y}$ be the $L^p_\delta$-completion of the elements of $\Omega^+ \oplus \Omega^0$ supported on $\End  X$ (so the elements of $\mathcal{X}$ and $\mathcal{Y}$ vanish on the compact part). Then the map 
$$d^+_{h'_\infty} \oplus d^{*, \delta}_{h'_\infty}\colon  \mathcal{X} \longrightarrow \mathcal{Y}$$
 is bounded and elliptic. We have assumed that $\delta/ 2$ is not in the spectrum of $-* d_h$, so it follows that the above operator has trivial kernel (it also has trivial cokernel, though we do not need this). In particular, there is a constant $C_5$ so that
$$\Vert V \Vert_{L^p_{\delta}} \leq \Vert V \Vert_{L^p_{1, \delta}}  \leq C_5 \Big( \Vert d^+_{h'_\infty} V \Vert_{L^p_\delta} + \Vert d^{*, \delta}_{h'_\infty} V \Vert_{L^p_\delta}\Big)$$
for all $V \in \mathcal{X}$. The connection $A'$ is $\C^0$-asymptotic to $h'_\infty$. In particular, we can choose $T_1$ large enough so that 
$$\Vert A' - h'_\infty \Vert_{\C^0(\left[T_1 - 1 , \infty \right) \times N)}  < 1/(6 {\mathfrak{c}_{\mathfrak{g}}} C_4).$$
Then if $V \in \mathcal{X}$ is supported on $\left[ T_1 - 1, \infty \right) \times N$ and in the kernel of $d_{A'}^{*, \delta}$ we have
$$\begin{array}{rcl}
\Vert V \Vert_{L^p_\delta( \left[ T_1 - 1, \infty \right) \times N)} & \leq & C_5\Big( \Vert d^+_{h'_\infty} V \Vert_{L^p_\delta( \left[ T_1 - 1, \infty \right) \times N)} + \Vert d^{*, \delta}_{h'_\infty} V \Vert_{L^p_\delta( \left[ T_1 - 1, \infty \right) \times N)}\Big)\\
& = & C_5\Big( \big\| d^+_{A'} V + \left[ \big(h'_\infty - A'\big) \wedge V \right]^+ \big\|_{L^p_\delta( \left[ T_1 - 1, \infty \right) \times N)}\\
&& \indent \indent \indent \indent + \big\| \left[ \big(h'_\infty - A'\big) \wedge * V \right] \big\|_{L^p_\delta( \left[ T_1 - 1, \infty \right) \times N)} \Big)\\
& \leq & C_5 \Vert d^+_{A'} V  \Vert_{L^p_\delta( \left[ T_1 - 1, \infty \right) \times N)} + \frac{1}{3} \Vert V  \Vert_{L^p_\delta( \left[ T_1 - 1, \infty \right) \times N)}.
\end{array}$$
Then Claim 3 follows (with $C_{(\ref{eq:Lemma3})} \defeq 3C_5/2$) from this estimate and a cutoff function.
\end{proof}

\subsection{Extensions to $p < 2$}\label{sec:smallp}

 In our existence result of Section \ref{sec:b^+=1}, we will need extensions to $p < 2$ of the estimates (\ref{eq:estypre}) and (\ref{eq:estyp0}); we state and prove the relevant extensions here. In fact, all we will need is an extension to $p = 4/3$ (so $p^* = 2$); we leave any more general extensions to the interested reader. Throughout this section, we fix data as in the statement of Theorem \ref{thm:1}.

For the first result, let $L > 1, \lambda_0> 0$, and $P \oplus \pi$ be as in the statement of Claim 1 appearing in the proof of Theorem \ref{thm:1}.

\begin{corollary}\label{cor:smallp0}
There is a constant $C$ so the following holds for all $0 < \lambda < \lambda_0$ and $\xi \in \Omega^+(X)$:
$$\Vert (P \oplus \pi)\xi \Vert_{(L^2_2(N) \times L^{2}_\delta(X, g_{L, \lambda}) ) \oplus L^{{4/3}}_\delta(X, g_{L, \lambda})}  \leq C\Vert \xi \Vert_{L^{4/3}_\delta(X,g_{L, \lambda})} .$$
\end{corollary}

\begin{proof}
We refer to the notation established in the proof of Claim 1. Momentarily suppressing Sobolev completions, let 
$$(P \oplus \pi)^{*,\delta} : \Big(T_{h'} \cH \times \ker(d^{*, \delta}_{A'})\Big) \oplus H^+ \longrightarrow \Omega^+(X)$$ 
be the formal adjoint of $P \oplus \pi$, relative to the $L^2_\delta$-inner product on $X$. By the duality isometries $(L_\delta^2(X))^* \cong L^2_\delta(X)$, $(L_\delta^4(X))^* \cong L^{4/3}_\delta(X)$, and $(L_2^2(N))^* \cong L^2_{-2}(N)$, we will be done if we can establish a uniform bound of the form
$$\Vert (P \oplus \pi)^{*,\delta}(\eta, V, \mu) \Vert_{L^4_\delta(X)} \leq C_1\left( \Vert \eta \Vert_{ L^2_{-2}(N)} + \Vert V \Vert_{L^{2}_\delta(X)} + \Vert \mu \Vert_{L^{4}_\delta(X)}\right) .$$
Since $T_{h'} \cH$ is finite-dimensional, there is a bound of the form $\Vert \eta \Vert_{L^2_{2}(N)} \leq C_2 \Vert \eta \Vert_{L^2_{-2}(N)}$ for all $\eta \in T_{h'} \cH$. It therefore suffices to show
\eqncount\begin{equation}\label{eq:popluspi}
\Vert (P \oplus \pi)^{*,\delta}(\eta, V, \mu) \Vert_{L^4_\delta(X)} \leq C_3\left( \Vert \eta \Vert_{L^2_{2}(N)} + \Vert V \Vert_{L^{2}_\delta(X)} + \Vert \mu \Vert_{L^{4}_\delta(X)}\right) 
\end{equation}
for a uniform constant $C_3$. This is precisely the estimate of Claim 1, except with the adjoint operator $(P \oplus \pi)^{*, \delta}$ in place of $P \oplus \pi$. We will show that the proof of Claim 1 can be sufficiently modified to hold for this adjoint. 

Towards this end, note that the adjoint of $P_k = \widetilde{D}_k^{*, \delta}(\widetilde{D}_k\widetilde{D}_k^{*,\delta})^{-1}$ is given by $P_k^{*,\delta} = (\widetilde{D}_k\widetilde{D}_k^{*,\delta})^{-1} \widetilde{D}_k$ and so satisfies 
$$\Vert P_k^{*, \delta} (\eta_k, V_k, \mu_k) \Vert_{L^4_\delta(X_k)} \leq c_k \left( \Vert \eta_k \Vert_{L^2_2(N_k)} + \Vert V_k \Vert_{L^2_\delta(X_k)} + \Vert \mu_k \Vert_{L^4_\delta(X_k)}\right).$$
Just as before, these can be glued together to form an operator $\Qt^{*, \delta}$ that satisfies
\eqncount\begin{equation}\label{eq:tqestta}
\Vert \Qt^{*, \delta} (\eta, V, \mu) \Vert_{L^4_\delta(X)} \leq (c_1 + c_2) \left( \Vert \eta \Vert_{L^2_2(N)} + \Vert V \Vert_{L^2_\delta(X)} + \Vert \mu \Vert_{L^4_\delta(X)}\right).
\end{equation}
Moreover, it is not hard to see that this gluing can be done so that $\Qt^{*, \delta}$ is exactly the formal $L^2_\delta$-adjoint of the operator $\Qt$ appearing in the proof of Claim 1. It follows that the formal $L^2_\delta$-adjoint of $\widetilde{Q}$ is the restriction of $\Qt^{*, \delta}$ to the slice (this just restricts $V$ to line in the kernel of $d^{*, \delta}_{A'}$). 

Then the defining formula $P \circ \pi = \widetilde{Q}(I + \widetilde{R})^{-1}$ implies 
\eqncount\begin{equation}\label{eq:formulaforpcircpiadjoint}
(P \circ \pi)^* = (I + \widetilde{R}^{*, \delta})^{-1}\widetilde{Q}^{*,\delta}
\end{equation}
where $\widetilde{R}^{*, \delta}$ is the formal adjoint of $\widetilde{R}$ and so satisfies $\Vert \widetilde{R}^{*, \delta} \xi \Vert_{L^2_\delta(X)} = \Vert \widetilde{R} \xi \Vert_{L^2_\delta(X)}$. Then the estimate (\ref{eq:popluspi}) follows from (\ref{eq:formulaforpcircpiadjoint}), (\ref{eq:tqestta}), and (\ref{eq:restimate}).
\end{proof}

For the second and last of the extensions we need, let $\xi({A_1, A_2}) \in L^2_\delta(\Omega^+(X))$ be as in the conclusion of Theorem \ref{thm:1} (c).

\begin{corollary}\label{cor:smallp1}
There are $C, \lambda_0' > 0$, so that if $0 < \lambda < \lambda_0'$ then
$$\Vert \xi({A_1, A_2})  \Vert_{L^{4/3}_\delta(X, g_{L, \lambda})}  \leq  C \lambda^{3/2}.$$
\end{corollary}

\begin{proof}
Setting $\xi \defeq \xi(A_1, A_2)$, the identity in (\ref{eq:xisttilde}) gives
$$\Vert \xi \Vert_{L^{4/3}_\delta(X)} \leq \Vert \widetilde{s}(0, 0) \Vert_{L^{4/3}_\delta(X)} + \Vert \St(\xi) \Vert_{L^{4/3}_\delta(X)}.$$
The estimate (\ref{eq:our727}) holds with $p = 4/3$, so the same is true of (\ref{eq:bounds}); that is,
$$\Vert \widetilde{s}(0, 0) \Vert_{L^{4/3}_\delta(X)}  \leq C_1 b^3$$
for a uniform constant $C_1$, where $b = 4 L \lambda^{1/2}$. To estimate $\St(\xi)$, note that the formula (\ref{eq:expansionofQ}) implies that $\St(\xi)$ is quadratically bounded in $P \xi$. Then we can argue as we did in the proof of Claim 2, but use H\"{o}lder's inequality $\Vert fg \Vert_{L^{4/3}_\delta} \leq \Vert f \Vert_{L^4_\delta} \Vert g \Vert_{L^2}$, to get a uniform estimate of the form
$$ \Vert \St(\xi) \Vert_{L^{4/3}_\delta(X)} \leq C_2 \Vert P \xi \Vert_{L^4_\delta(X)} \Vert P \xi \Vert_{L^2_\delta(X)}.$$
By (\ref{eq:estyp0}) and Corollary \ref{cor:smallp0}, this implies
$$ \Vert \St(\xi) \Vert_{L^{4/3}_\delta(X)} \leq C_3 \Vert \xi \Vert_{L^2_\delta(X)} \Vert  \xi \Vert_{L^{4/3}_\delta(X)}.$$
It follows from (\ref{eq:estypre}) that we can assume $\Vert  \xi \Vert_{L^2_\delta(X)} < (2 C_3)^{-1}$, provided $\lambda > 0$ is sufficiently small. In summary, this implies
$$\Vert \xi \Vert_{L^{4/3}_\delta(X)} \leq C_1 b^3 + \frac{1}{2}  \Vert \xi \Vert_{L^{4/3}_\delta(X)} $$
from which the corollary follows with $C = 128 L^3 C_1$.
\end{proof}

\section{Gauge fixing and the $\mASD$ condition}\label{sec:gaugefixing}

In the next section, we will find ourselves in the situation where we have an $\mASD$ connection $A$ and a nearby connection $A_{ref}$. We will want to find a gauge transformation $u$ so that $u^*A$ is in the Coulomb slice of $A_{ref}$. The issue is that, due to the failure of the $\mASD$ equation to be gauge invariant, the connection $u^* A$ will no longer be $\mASD$. Nevertheless, we will show in this section that, when $A$ is regular, the connection $u^*A$ is close to a unique $\mASD$ connection that lies in the $A_{ref}$-Coulomb slice. This is made precise in Theorem \ref{thm:3}, which extends the discussion to handle connections $A$ that are not regular by means of an obstruction map. To accomplish this, we first prove a general gauge fixing result that is tailored to our setting; this is stated in Proposition \ref{prop:1}.

\subsection{Gauge fixing}\label{subsec:GaugeFixing}

We begin by refining our choices of $\delta$ and the cut-off function $\beta$ used to create $\mathcal{H}_{out}$. For the former, we assume $\delta^2/4$ is not in the spectrum of the Laplacian $\Delta$ on real-valued functions. It then follows from Sobolev embedding that, for each $1 < q < 4$, there is a constant $\mathfrak{c}_q$ so that
\eqncount\begin{equation}\label{eq:24}
\Vert f \Vert_{L^q_\delta(X)} + \Vert f \Vert_{L^{q^*}_\delta(X)} \leq \mathfrak{c}_q \Vert d f \Vert_{L^q_\delta(X)}
\end{equation}
for all compactly supported real-valued smooth functions $f$, where $q^* = 4q/(4-q)$ is the Sobolev conjugate. 

As for the cutoff function $\beta\colon  \cH \rightarrow \left[0, 1 \right]$, we assume this is chosen so that it has small support in the sense that
\eqncount\begin{equation}\label{eq:supdis}
\sup_{h, h_0 \in \mathrm{supp}(\beta)} \Vert h - h_0 \Vert_{\C^0(N)} + \Vert \Theta(h) - \Theta(h_0) \Vert_{\C^0(N)}  < \frac{1}{2\mathfrak{c}_2 \mathfrak{c}_{\mathfrak{g}}}
\end{equation}
where $\mathfrak{c}_{\mathfrak{g}}$ is the constant from (\ref{eq:defoffrakg}) and $\mathfrak{c}_2$ is the constant from (\ref{eq:24}) with $q = 2$.

The main gauge fixing result we will need is as follows.

\begin{proposition}\label{prop:1}
Fix $2 < p < 4$, set $p^* = 4p/ (4 - p)$, and assume $\delta, \beta$ are as above. There are constants $C, \epsilon > 0$ so that if $A = \iota(h, V)$ and $A_{ref} = \iota(h_{ref}, V_{ref})$ are in $ \A^{1, p}(\mathcal{T}_\Gamma)$ and satisfy
$$\Vert V - V_{ref} \Vert_{L^{p^*}_{\delta}(X)} + \Vert d^{*, \delta}_{A_{ref}}(V - V_{ref}) \Vert_{L^p_\delta(X)}  < \epsilon$$
then there is a unique $\mu = \mu({A, A_{ref}}) \in L^p_{2, \delta}(\Omega^0(X))$ so that 
$$\exp(\mu)^* A \in \Sl(A_{ref}), \hspace{.5cm} \mathrm{and} \hspace{.5cm}  \Vert d^{*, \delta}_{A_{ref}} d_{A}  \mu \Vert_{L^p_{\delta}(X)} \leq C \Vert d_{A_{ref}}^{*, \delta}(V - V_{ref}) \Vert_{L^p_{\delta}(X)} . $$
Moreover, this 0-form $\mu({A, A_{ref}})$ depends $\C^m$-smoothly on the pair $(A, A_{ref})$.
\end{proposition}

\begin{proof}
We will show below that 
$$\Vert \mu \Vert_{\mathcal{X}} \defeq \Vert d^{*, \delta}_{A_{ref}} d_{A} \mu \Vert_{L^p_\delta(X)}$$
 defines a norm on the space $\Omega^0(X)$ of smooth rapidly decaying adjoint bundle-valued 0-forms. Assuming this for now, we denote by $\mathcal{X}$ the completion of $ \Omega^0(X)$ relative to $\Vert \cdot \Vert_{\mathcal{X}}$. Let $\mathcal{Y}$ be the completion of $\Omega^0(X)$, but relative to the norm $\Vert \cdot  \Vert_{\mathcal{Y}} \defeq \Vert \cdot \Vert_{L^p_\delta(X)}$. Since $p > 2$, the map
$$\begin{array}{rcl}
\mathcal{F}\colon  \A^{1, p}(\mathcal{T}_\Gamma)\times \A^{1, p}(\mathcal{T}_\Gamma)  \times \mathcal{X} &\longrightarrow & \mathcal{Y}\\
(A, A_{ref}, \mu) & \longmapsto &  d^{*, \delta}_{A_{ref}}(u^* A - i(p(u^*A)) - V_{ref})
\end{array}$$
is $\C^{m}$-smooth, where we have set $u = \exp(\mu) \in \G^{2, p}_\delta$. Note that, relative to the product structure given by $\iota$ via (\ref{eq:iotamapdiff}), the quantity $u^*A - i(p(u^*A))$ is the $L^p_{1, \delta}(\Omega^1(X))$-component (i.e., non-center manifold-component) of $u^* A$, and so $u^* A \in \Sl(A_{ref})$ for $u = \exp(\mu)$ if and only if $\mathcal{F}(A, A_{ref}, \mu) = 0$. It therefore suffices to solve $\mathcal{F}(A, A_{ref}, \mu) = 0$ for $\mu$. For this, we have that $\mu = 0$ is an approximate solution since 
$$\mathcal{F}(A, A_{ref}, 0) = d^{*, \delta}_{A_{ref}}(V - V_{ref})$$ 
which we have assumed is bounded by $\eps$. The linearization in the third component of $\mathcal{F}$ at $(A, A_{ref}, 0)$ is the operator
$$\mu \longmapsto d^{*, \delta}_{A_{ref}} d_A \mu .$$
This has operator norm 1 relative to the norms on $\mathcal{X}$ and $\mathcal{Y}$. In particular, it is invertible and so the proposition follows from the inverse function theorem (e.g., precompose $\F$ in the third component with the inverse of $d^{*, \delta}_{A_{ref}} d_A$ and then use Lemma \ref{lem:dklem}). 

All that remains is to show that $\Vert \cdot \Vert_{\mathcal{X}}$ defines a norm; it suffices to show that the operator 
$$d^{*, \delta}_{A_{ref}} d_{A}\colon  L^p_{2, \delta} \longrightarrow L^p_\delta$$ 
is injective. For this, suppose $\mu $ lies in its kernel and let $(\cdot, \cdot)_\delta$ be the $\delta$-dependent $L^2$-inner product. Note that $\mu \in L^2_{1, \delta}$ by Sobolev's embedding theorem $L^p_2 \hookrightarrow L^2_1$ on cylindrical end 4-manifolds \cite[Prop.~3.20]{donaldson10}. This justifies the following computation:
$$0 = (d^{*, \delta}_{A_{ref}} d_A \mu, \mu)_\delta = (d_A \mu , d_{A_{ref}} \mu)_\delta = \Vert d_{A_{ref}} \mu \Vert^2_{L^2_\delta} + \Big(\big[ \big(A - A_{ref}\big), \mu \big], d_{A_{ref}} \mu\Big)_\delta.$$
Hence
$$\Vert d_{A_{ref}} \mu \Vert_{L^2_\delta} \leq \Big\| \big[ \big(A - A_{ref} \big), \mu \big] \Big\|_{L^2_\delta}.$$
The definition of $\iota$ gives $A - A_{ref} = \beta''(h- h_{ref} +(\Theta(h) - \Theta(h_{ref})) dt ) + V - V_{ref}$. Then H\"{o}lder's inequality and (\ref{eq:supdis}) allow us to continue the above inequality to get
$$\begin{array}{rcl}
\Vert d_{A_{ref}} \mu \Vert_{L^2_\delta} & \leq & \mathfrak{c}_{\mathfrak{g}} \big(\Vert V - V_{ref} \Vert_{L^4} \Vert \mu \Vert_{L^4_\delta} + \Vert h - h_{ref} + (\Theta(h) - \Theta(h_{ref})) dt \Vert_{\C^0} \Vert \mu \Vert_{L^2_\delta}\big)\\
&\leq & \mathfrak{c}_{\mathfrak{g}} \Vert e^{-\delta t/2} \Vert_{L^r} \Vert V - V_{ref} \Vert_{L^{p^*}_\delta} \Vert \mu \Vert_{L^4_\delta} + \frac{1}{2 \mathfrak{c}_{2}} \Vert \mu \Vert_{L^2_\delta}\\
\end{array}$$
where $r$ is defined by $r^{-1} + (p^*)^{-1} = 4^{-1}$. Using (\ref{eq:24}) with $f = \vert \mu \vert$, and then Kato's inequality $\vert d \vert \mu \vert \vert \leq \vert d_{A_{ref}} \mu \vert$ (which holds for arbitrary metric connections), we can use the above to get
$$\begin{array}{rcl}
\Vert \mu \Vert_{L^2_\delta} + \Vert \mu \Vert_{L^4_\delta} & \leq & \mathfrak{c}_2 \Vert d \vert \mu \vert \Vert_{L^2_\delta}\\
& \leq & \mathfrak{c}_2 \Vert d_{A_{ref}} \mu \Vert_{L^2_\delta}\\
& \leq &  \mathfrak{c}_2 \mathfrak{c}_{\mathfrak{g}} \Vert e^{-\delta t/2} \Vert_{L^r} \Vert V- V_{ref} \Vert_{L^{p^*}_\delta} \Vert \mu \Vert_{L^4_\delta}   + \frac{1}{2} \Vert \mu \Vert_{L^2_\delta}.
\end{array}$$
When $\eps < 1/(2 \mathfrak{c}_2 \mathfrak{c}_{\mathfrak{g}} \Vert e^{-\delta t/2} \Vert_{L^r} )$, this implies that $\mu = 0$. 
\end{proof}

\begin{remark}\label{rem:bounds2}
The operator $d^{*, \delta}_{A_{ref}} d_{A}\colon  L^p_{2, \delta} \rightarrow L^p_\delta$ is Fredholm, and we have just seen that it has trivial kernel under the hypotheses of the proposition. It then follows from the embedding $L^p_{2, \delta} \subseteq L^p_\delta \cap \C^0$ that there is a constant $C$ so that 
$$\Vert \mu \Vert_{L^p_\delta(X)} + \Vert \mu \Vert_{\C^0(X)} \leq C\Vert d^{*, \delta}_{A_{ref}} d_{A}  \mu \Vert_{L^p_{\delta}(X)}$$
for all $\mu \in L^p_{2, \delta}(\Omega^0(X))$. It follows from arguments similar to those just used that this constant can be chosen to be independent of $A$ and $A_{ref}$ provided these connections satisfy the hypotheses of Proposition \ref{prop:1}.
\end{remark}

\subsection{Recovering the $\mASD$ condition within a slice}\label{sec:gluinginagaugeslice}

Throughout this section, we assume $2 < p < 4$, and $\delta, \beta$ are chosen as in Section \ref{subsec:GaugeFixing}.

As suggested in the introduction to this section, we will use Proposition \ref{prop:1} to put $\mASD$ connections into a fixed nearby slice, but this process will generally not preserve the $\mASD$ condition. The following theorem is our main readjustment tool that will recover the $\mASD$ condition, while simultaneously preserving the slice condition. To state it, use the $L^2_\delta$-inner product to identify the cokernel $H^+_{A, \delta} = \mathrm{coker}(Ds\vert_{\Sl(A)})_A$ with the subset of $A$-harmonic self-dual forms in $L^p_\delta( \Omega^+(X))$. We denote by 
$$\sigma_A\colon  H^+_{A, \delta} \longrightarrow L^p_\delta(\Omega^+(X)) \hspace{2cm} \pi_{A}\colon  L^p_\delta( \Omega^+(X)) \longrightarrow H^+_{A, \delta}$$ 
the inclusion and $L^2_\delta$-orthogonal projection, respectively. (These maps will play a role analogous to the one played by $\sigma$ and $\pi$ in Section \ref{sec:GenGluing}.) It follows that $(Ds\vert_{\Sl(A)})_A \oplus \sigma_A$ maps surjectively onto $L^p_\delta(\Omega^+(X))$.

\begin{theorem}\label{thm:3}
Fix $A_{ref} =\iota(h_{ref}, V_{ref}) \in \A^{1,p}(\mathcal{T}_\Gamma)$. Then there are constants $C, \eps > 0$ so that the following holds for all $\mASD$ connections $A = \iota(h, V) $ satisfying 
\eqncount\begin{equation}\label{eq:epsaa}
 \Vert h - h_{ref} \Vert_{L^2_2(N)} + \Vert V - V_{ref} \Vert_{L^{p^*}_{\delta}(X)} + \Vert d^{*, \delta}_{A_{ref}}(V - V_{ref}) \Vert_{L^p_\delta(X)}  < \epsilon.
 \end{equation}
\begin{itemize}
\item[(a)] There is a $\C^m$-map $K_{A}\colon  L^p_\delta(\Omega^+(X)) \rightarrow \Sl(A_{ref})$ that restricts to an embedding on a neighborhood $U$ of $0$.
\item[(b)] There is a unique 2-form $\zeta(A) \in U \subseteq L^p_\delta(\Omega^+(X))$ so that 
$$\Vert \zeta(A) \Vert_{L^p_\delta(X)}  \leq  C \Vert d^{*, \delta}_{A_{ref}}(V - V_{ref}) \Vert_{L^p_\delta(X)}$$
and so that the connection $\mathcal{K}(A) \defeq K_A(\zeta(A))$ satisfies
$$s(\mathcal{K}(A) )  =  -\sigma_A \pi_{A} \zeta(A).$$ 
\end{itemize}
In particular, the connection $\mathcal{K}(A)$ is close to $A$ in the sense that there is a constant $C'$ so that
$$\Vert \iota^{-1}(\mathcal{K}(A) ) - \iota^{-1}(A) \Vert_{L^2_2(N) \times L^{p^*}_{\delta}(  X)} \leq C' \Vert d^{*, \delta}_{A_{ref}}(V - V_{ref}) \Vert_{L^p_\delta(X)}.$$
If either $A$ or $A_{ref}$ is regular, then {all three connections $A$, $A_{ref}$ and $\mathcal{K}(A)$ are regular and $\mathcal{K}(A)$ is $A_{ref}$-regular}. In this case, the connection $\mathcal{K}(A) $ is $\mASD$ and the maps $(A, \zeta) \mapsto K_A(\zeta)$ and $A \mapsto \zeta(A)$ are both $\C^m$-smooth, relative to the specified topologies. If $A$ is irreducible, then so is $\mathcal{K}(A)$. The constants $\epsilon, C, C'$ can be chosen to vary continuously in $A_{ref}$.
\end{theorem}

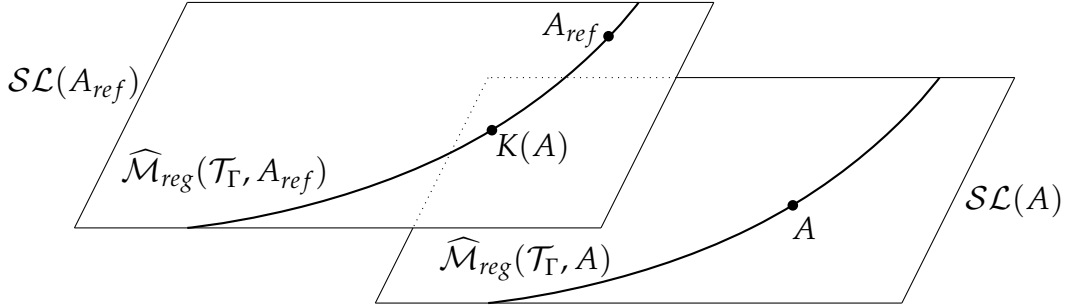
\begin{figure}[h]
\begin{tikzpicture}
\tikzstyle{every node} = [sloped,above] 
\coordinate (a) at (-2,0);
\coordinate (b) at (5,0);
\coordinate (c) at (6.5,3);
\coordinate (d) at (-.5,3);
 \draw  (a) -- (b);
  \draw (c) -- (b);
   \draw (c) -- (d);
    \draw (a) -- (d);
    \node at (-2,1.5) {$\Sl(A_{ref})$};
    	   \draw [thick] plot [smooth, tension = 1] coordinates {(-.5,0) (3,1) (5.5,3)};
    	      \node at (4.1,.7) {$\mathcal{K}(A)$};
    	      \fill [black] (3.55,1.3) circle (2pt);
    	         \node at (4.6,2.25) {$A_{ref}$};
\fill [black] (5.1,2.55) circle (2pt);
 \node at (0,.3) {$\rM_{reg}(\mathcal{T}_\Gamma, A_{ref})$};
    \coordinate (a') at (2,-1);
\coordinate (b') at (9,-1);
\coordinate (c') at (10.5,2);
\coordinate (d') at (3.5,2);
\coordinate (m) at (6,2);
\coordinate (n) at (2.5,0);
 \draw  (a') -- (b');
  \draw (c') -- (b');
  \draw (c') -- (m);
   \draw [dotted] (m) -- (d');
   \draw [dotted] (d') -- (n);
    \draw  (a') -- (n);
    \node at (10.5,0) {$\Sl(A)$};
        \draw [thick] plot [smooth, tension = 1] coordinates {(3.5,-1) (7,0) (9.5,2)};
            	      \fill [black] (7.55,.3) circle (2pt);
    	         \node at (7.7,-.3) {$A$};
  \node at (4,-.85) {$\rM_{reg}(\mathcal{T}_\Gamma, A)$};
\end{tikzpicture}
    \caption{Pictured above is the special case of Theorem \ref{thm:3} where $A$ is regular. The curved lines represent the spaces of regular $\mASD$ connections in the slices $\Sl(A_{ref})$ and $\Sl(A)$, respectively.} \label{fig:2}
    \end{figure}

\begin{proof}
Take $\eps> 0$ to be no larger than the epsilon from the statement of Proposition \ref{prop:1}. Then it follows from that proposition and Remark \ref{rem:bounds2} that, given $A = \iota(h, V)$ with
$$s(A) = 0, \hspace{.5cm} \mathrm{and} \hspace{.5cm}  \Vert V - V_{ref} \Vert_{L^{p^*}_{\delta}(X)} + \big\| d^{*, \delta}_{A_{ref}}(V - V_{ref}) \big\|_{L^p_\delta(X)}  < \epsilon$$
there is a unique $\mu \in L^p_{2, \delta}(\Omega^0)$ so that $\exp(\mu)^* A \in \Sl(A_{ref})$ and
$$\Vert \mu \Vert_{L^p_\delta} + \Vert \mu \Vert_{\C^0} \leq C_1 \big\| d_{A_{ref}}^{*, \delta}(V - V_{ref})\big\|_{L^p_\delta}.$$  
Set $u = \exp(\mu)$ and write $u^* A = \iota(h_A, V_A)$ for $h_A \in \cH_{out}$ and $V_A \in L^p_{1, \delta}(\Omega^1)$. Let $\exp_{h_A}\colon  B_{\eps}(0) \subseteq T_{h_A} \cH \rightarrow \cH$ be the exponential map for the center manifold based at $h_A$, and extend this to a map 
$$\begin{array}{rcl}
\exp_{(h_A, V_A)}\colon  B_{\eps}(0) \times L^p_{1, \delta}\big(\ker(d_{A_{ref}}^{*, \delta})\big) & \longrightarrow & \cH \times L^p_{1, \delta}(\ker(d_{A_{ref}}^{*, \delta}))\\
(\eta, V) & \longmapsto & \big(\exp_{h_A}(\eta), V_A + V\big)
\end{array}$$
which is a $\C^m$-diffeomorphism in a neighborhood of $(0, 0)$. Using this, define
$$\widetilde{s}\colon  T_{h_A} \cH \times L^p_{1, \delta}\big(\ker(d_{A_{ref}}^{*, \delta})\big) \longrightarrow L^p_\delta(\Omega^+), \hspace{.5cm} (\eta, V) \longmapsto s\Big(\iota\big(\exp_{(h_A, V_A)}(\eta, V)\big)\Big).$$
By definition of $\sigma_A$, the operator $(Ds\vert_{\Sl(A)})_{A} \oplus \sigma_{A}$ is surjective. The operators $(Ds)_A$ and $(Ds)_{u^*A}$ are approximately equal when $u$ is $\C^0$-close to the identity (i.e., when $\Vert \mu \Vert_{\C^0}$ is small). Thus, when $\eps$ is sufficiently small, the operator $(Ds\vert_{\Sl(A_{ref})})_{u^*A} \oplus \sigma_A$ is also surjective, as is $(D\widetilde{s})_{(0, 0)} \oplus \sigma_A$. Then we can choose a right inverse to $(D\widetilde{s})_{(0, 0)} \oplus \sigma_A$ of the form $P \oplus \pi_A$, where $\pi_A$ is the projection to $H^+_{A, \delta}$. For $\zeta \in L^p_\delta(\Omega^+(X))$, define
$$K_A(\zeta) \defeq \iota\big(\exp_{(h_A, V_A)} (P \zeta)\big).$$
This proves (a) in the statement of the theorem, by taking $U \subseteq L^p_\delta(\Omega^+(X))$ to be small enough so that $P(U) \subseteq B_{\eps}(0) \times \ker(d^{*, \delta}_{A_{ref}})$. 

To prove (b), we use the same implicit function theorem argument as in Theorem \ref{thm:1}. Namely, set
$$\St(\zeta) \defeq  \widetilde{s}(P \zeta) - (D \widetilde{s})_{(0, 0)} P \zeta   - \widetilde{s}(0, 0).$$
The argument of Claim 2 in the proof of Theorem \ref{thm:1} carries over to show that $\St$ satisfies the quadratic estimate of Lemma \ref{lem:dklem}. We will show in a moment that there is a uniform constant $C_2$ so that
\eqncount\begin{equation}\label{eq:tilde}
\Vert \widetilde{s}(0, 0) \Vert_{L^p_\delta} \leq C_2 \big\| d_{A_{ref}}^{* , \delta}(V - V_{ref}) \big\|_{L^p_\delta}.
\end{equation}
From this and Lemma \ref{lem:dklem} it follows that, by assuming $\epsilon$ is sufficiently small, there is a unique $\zeta(A)$ so that 
$$\zeta(A) + \St(\zeta(A)) = - \widetilde{s}(0, 0).$$
As we argued in the proof of Theorem \ref{thm:1}, this $\zeta(A)$ satisfies the assertions of (b). The regularity and irreducibility assertions also follow as in Theorem \ref{thm:1}.

It therefore suffices to verify (\ref{eq:tilde}). By definition of $\widetilde{s}$, we have $\widetilde{s}(0, 0) = s(u^*A)$. Recall that the projection $p_T$ to the center manifold is gauge invariant, so $p_T(u^* A) = p_T(A) = h$. This implies 
$$\begin{array}{rcl}
\widetilde{s}(0, 0) = s(u^*A) & =& F_{u^* A}^+ - \beta F_{i(h)}^+\\
& = & \mathrm{Ad}(u^{-1}) F_A^+ - \beta F_{i(h)}^+\\
& = & \mathrm{Ad}(u^{-1}) \beta F^+_{i(h)} - \beta F_{i(h)}^+\\
& = & \beta (\mathrm{Ad}(u^{-1}) - I) F^+_{i(h)}
\end{array}$$
where, in the penultimate equality, we used the assumption that $s(A) = F^+_A - \beta F^+_{i(h)} =  0$ vanishes. By shrinking $\eps$ further still, we may suppose $\Vert \mu \Vert_{\C^0} \leq 1$. Then the Taylor expansion for the exponential $u = \exp(\mu)$ gives 
$$\begin{array}{rcl}
\Vert \widetilde{s}(0, 0) \Vert_{L^p_\delta} &\leq & C_3  \big\| F^+_{i(h)} \Vert_{\C^0} \Vert \mu \big\|_{L^p_\delta}\\
& \leq & C_1 C_3  \big\| F^+_{i(h)} \big\|_{\C^0} \big\| d_{A_{ref}}^{* , \delta}(V - V_{ref}) \big\|_{L^p_\delta}.
\end{array}$$
The quantity $\Vert F^+_{i(h)} \Vert_{\C^0} $ is bounded independent of $h$ since $\cH_{out}$ is compact. 
\end{proof}

\section{Gluing regular families}\label{sec:GluingRegFam}

Throughout this section we work with the space $\A^{1,p}(\mathcal{T}_\Gamma)$ for fixed $2 < p < 4$. We assume that $\delta$ and the cutoff function $\beta$ are chosen as in Section \ref{subsec:GaugeFixing}. We also assume $\delta/2 < \mu_\Gamma^-$, so the index formula discussed in Section \ref{sec:mmrerror} applies. 

We freely refer to the notation of Section \ref{sec:GenGluing}. For $k = 1, 2$, fix a precompact open set
$$G_k \subseteq \rM_{reg}(\mathcal{T}_{k, \Gamma_k}, A_{ref, k})$$ 
of $A_{ref,k}$-regular $\mASD$ connections on $X_k$ relative to some reference connection $A_{ref, k}$. Since the $G_k$ are precompact, we can fix $L, \lambda > 0$ so that conclusions of Theorem \ref{thm:1} hold for all $(A_1, A_2) \in G_1 \times G_2$. (In our applications of the material of this section, the values of $L$ and $\lambda$ will be fixed, so we do not keep track of them in the notation.) Then Theorem \ref{thm:1} produces a regular $\mASD$ connection $\mathcal{J}(A_1, A_2) \in \A^{1,p}(\mathcal{T}_\Gamma)$.

Ideally, we would want to view the mapping $(A_1, A_2) \mapsto \mathcal{J}(A_1, A_2)$ as a function from $G_1 \times G_2$ into a fixed $\mASD$ space. However, since the Coulomb slice to which $\mathcal{J}(A_1, A_2)$ belongs depends on $(A_1, A_2)$ (cf. Remark \ref{rem:GaugeSlice} (b)), it is more natural to realize this mapping as a section of a bundle. Towards this end, set
$$\E \defeq \left\{(A_1, A_2, A)\; \Big| \; A_k \in G_k, \hspace{.25cm} A \in \rM_{reg}\big(\mathcal{T}_\Gamma, \mathcal{J}(A_1, A_2)\big)\right\}.$$
Let $\Pi \colon  \E \rightarrow G_1 \times G_2$ be the projection to the first two factors. Then the map 
$$\Psi(A_1, A_2) \defeq \big(A_1, A_2,  \mathcal{J}(A_1, A_2)\big)$$ 
is clearly a section of the map $\Pi$. 

\begin{theorem}\label{thm:2} \mbox{}

(a) For all sufficiently small $\lambda>0$, there is a neighborhood $\mathcal{U} \subseteq \E$ of the image of $\Psi$ so that the restriction $\Pi \vert_{\mathcal{U} }\colon  \mathcal{U}  \rightarrow G_1 \times G_2$ is a locally trivial fiber bundle. The fibers of $\Pi \vert_{\mathcal{U}}$ can be identified with open subsets of $\rM_{reg}(\mathcal{T}_\Gamma, A_{ref})$ for some $A_{ref}$.

 More specifically, every $(A_{10}, A_{20}) \in G_1 \times G_2$ is contained in an open neighborhood $\mathcal{V} \subseteq G_1 \times G_2$ so that the following holds. Let $A_{ref} = A'(A_{10}, A_{20})$ be the preglued connection, and consider the map $\mathcal{K}$ from Theorem \ref{thm:3}, defined relative to this reference connection $A_{ref}$. Then the map
\eqncount\begin{equation}\label{eq:trivie}
\Pi^{-1}(\mathcal{V}) \cap \mathcal{U} \longrightarrow \mathcal{V} \times \rM_{reg}(\mathcal{T}_\Gamma, A_{ref}) \hspace{1cm} (A_1, A_2, A) \longmapsto (A_1, A_2, \mathcal{K}(A))
\end{equation}
is a well-defined $\C^m$-diffeomorphism onto an open subset of the codomain, and this map produces a local trivialization of $\Pi \vert_{\mathcal{U} }$ over $\mathcal{V}$. 

\medskip

(b) The map 
$$\Phi \defeq \mathcal{K} \circ \mathcal{J} : \mathcal{V} \longmapsto \rM_{reg}(\mathcal{T}_\Gamma, A_{ref})$$
is a $\C^m$-embedding. If $A_1$ or $A_2$ is irreducible, then the connection $ \Phi(A_1, A_2)$ is also irreducible. 
\end{theorem}

Part (b) can be restated by saying that, relative to the local trivilization of (a), locally the section $\Psi$ becomes a $\C^m$-embedding onto an embedded $\C^m$-submanifold of the fiber. This is an $\mASD$ version of the familiar result for $\ASD$ connections that gluing produces a parametrized family of connections in the $\ASD$ moduli space for a connected sum. See Figure \ref{fig:3} for an illustration of the fiber bundle in (a), and Figure \ref{fig:4} for an illustration of the specified trivialization, as well as the map $\Phi$.

\begin{figure}[h]
\begin{tikzpicture}
\tikzstyle{every node} = [sloped,above] 
\coordinate (R1) at (5,2);
\coordinate (R2) at (5,1);
\coordinate (R3) at (5,0);
\coordinate (R4) at (5,-1);
\coordinate (R5) at (5,-2);
\coordinate (L1) at (-5,2);
\coordinate (L2) at (-5,1);
\coordinate (L3) at (-5,0);
\coordinate (L4) at (-5,-1);
\coordinate (L5) at (-5,-2);
\coordinate (A1) at (.2,-2.2);
\coordinate (B1) at (.2,-3.6);
\coordinate (A2) at (-.2,-2.2);
\coordinate (B2) at (-.2,-3.6);
\coordinate (RB) at (-5,-3.8);
\coordinate (LB) at (5,-3.8);
\draw  (L1) \irregularline{1mm}{10} ;
\draw  (L5) \irregularline{1mm}{10} ;
    \draw [dashed] plot [smooth, tension = 1] coordinates {(R2) (0,1.2) (L2)};
  \draw [dashed] plot [smooth, tension = 1] coordinates {(R4) (0,-.8) (L4)};
     \draw [ line width = .4mm]  plot [smooth, tension = 1] coordinates {(R3) (0,.2) (L3)};
   \draw (L1) -- (L5);
   \draw (R1) -- (R5);
   \draw [ line width = .4mm]  (RB) -- (LB);
         \draw[->, line width = .2mm] (A2) -- (B2);
          \draw[->, line width = .2mm] (B1) -- (A1);
   \node at (-5.85,-.25) {$\mathcal{U}$};
    \node at (-7.15, -.25) {$\E$};
     \node at (0,.20) {$\mathrm{im}(\Psi)$};
       \node at (-.5,-3.2) {$\Pi$};
        \node at (.5,-3.2) {$\Psi$};
      \node at (-6,-4.15) {$G_1 \times G_2$};
      \draw [decorate,decoration={brace,amplitude=20pt},xshift=-4pt,yshift=0pt]
(-6,-2) -- (-6,2) node [black,midway,xshift=-0.6cm] {};
 \draw [decorate,decoration={brace,amplitude=10pt},xshift=0pt,yshift=0pt]
(-5.25,-1) -- (-5.25,1) node [black,midway,xshift=-0.3cm] {};
\end{tikzpicture}
    \caption{The above picture illustrates the fiber bundle $\Pi \vert_{\mathcal{U}}: \mathcal{U} \rightarrow G_1 \times G_2$ obtained by restricting the projection $\Pi: \E \rightarrow G_1 \times G_2$ to the open submanifold $\mathcal{U} \subseteq \E$. The fibers are $\C^m$-diffeomorphic to open subsets of $\rM_{reg}(\mathcal{T}_\Gamma, A_{ref})$.} \label{fig:3}
    \end{figure}
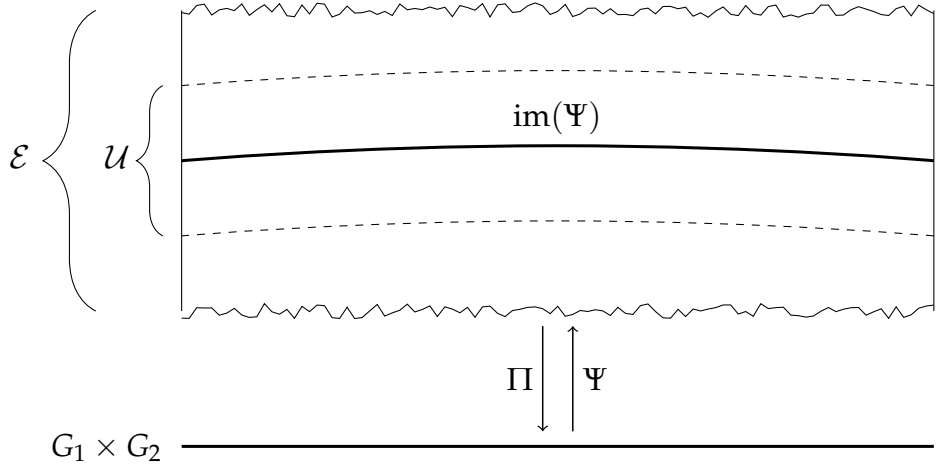

\begin{figure}[h]
\begin{tikzpicture}
\coordinate (L1) at (0,2);
\coordinate (L2) at (3,2);
\coordinate (L4) at (0,0);
\coordinate (L3) at (3,0);
\coordinate (LM1) at (0,1);
\coordinate (LM2) at (3,1);
 \draw[dashed]  plot [smooth, tension = 1] coordinates {(L1) (1.5,2.1) (L2)};
  \draw[dashed] (L2) -- (L3);
   \draw[dashed]  plot [smooth, tension = 1] coordinates {(L3) (1.5,.1) (L4)};
    \draw[dashed] (L1) -- (L4);
      \draw[ line width = .4mm]  plot [smooth, tension = 1] coordinates {(LM1) (1.5,1.1) (LM2)};
         \node at (1.5,1.35) {$\mathrm{im}(\Psi)$};
          \node at (1.5,2.6) {$\Pi^{-1}(\mathcal{V}) \cap \mathcal{U}$};
\coordinate (M1) at (6,2);
\coordinate (M2) at (9,2);
\coordinate (M4) at (6,0);
\coordinate (M3) at (9,0);
\coordinate (M1') at (6,2.5);
\coordinate (M2') at (9,2.5);
\coordinate (M4') at (6,-.5);
\coordinate (M3') at (9,-.5);
\coordinate (M2'') at (9,1.5);
\coordinate (M4'') at (6,.5);
 \draw[dashed] (M1) -- (M2);
  \draw[dashed] (M3) -- (M2);
   \draw[dashed] (M4) -- (M3);
    \draw[dashed] (M1) -- (M4);
     \draw[dotted] (M1') -- (M2');
  \draw[dotted] (M1') -- (M1);
   \draw[dotted] (M2') -- (M2);
   \draw[dotted] (M4') -- (M3');
      \draw[dotted] (M3') -- (M3);
   \draw[dotted] (M4') -- (M4);
     \draw [ line width = .4mm]  (M4'') -- (M2'');
     \node at (7.5,3) {$\mathcal{V} \times \rM_{reg}(\mathcal{T}_\Gamma, A_{ref})$};
\coordinate (R1) at (12,2);
\coordinate (R2) at (12,0);
\coordinate (R1') at (12,2.5);
\coordinate (R2') at (12,-.5);
\coordinate (R1'') at (12,1.5);
\coordinate (R2'') at (12,.5);
\draw [ line width = .4mm]  (R1'') -- (R2'');
\draw [dotted] (R1') -- (R1);
\draw [dotted] (R2') -- (R2);
\draw[dashed] (R1'') -- (R1);
\draw[dashed] (R2'') -- (R2);
\node at (12,3) {$\rM_{reg}(\mathcal{T}_\Gamma, A_{ref})$};
\coordinate (B1) at (6,-3);
\coordinate (B2) at (9,-3);
\draw [ line width = .4mm]  (B1) -- (B2);
\node at (8.5,-3.4) {$\mathcal{V} \subseteq G_1 \times G_2$};
 \coordinate (A1) at (3.5,1);
 \coordinate (A2) at (5.5,1);
    \draw[->, line width = .2mm] (A1) -- (A2);  
     \node at (4.5,1.5) {(\ref{eq:trivie})  };
     \coordinate (B1) at (9.5,1);
 \coordinate (B2) at (11.5,1);
    \draw[->, line width = .2mm] (B1) -- (B2);   
         \coordinate (D1) at (1.4,-.3);
 \coordinate (D2) at (6.7,-2.7);
    \draw[->, line width = .2mm] (D1) -- (D2);   
    \node at (3.8,-1.7) {$\Pi$};
             \coordinate (DU1) at (1.7,-.3);
 \coordinate (DU2) at (7,-2.7);
    \draw[->, line width = .2mm] (DU2) -- (DU1);   
    \node at (5,-1.5) {$\Psi$};
           \coordinate (E1) at (7.5,-1);
 \coordinate (E2) at (7.5,-2.7);
    \draw[->, line width = .2mm] (E1) -- (E2);   
             \coordinate (F1) at (11.75,.75);
 \coordinate (F2) at (8,-2.7);
    \draw[->, line width = .2mm] (F2) -- (F1);  
     \node at (10.1,-1.3) {$\Phi$}; 
\end{tikzpicture}
    \caption{The horizontal arrow on the left is the $\C^m$-diffeomorphism (\ref{eq:trivie}) trivializing $\Pi \vert_{\mathcal{U}}$ over $\mathcal{V}$. This is a $\C^m$-diffeomorphism onto the region in $\mathcal{V} \times \rM_{reg}(\mathcal{T}_\Gamma, A_{ref})$ represented by the dashed lines. It takes $\mathrm{im}(\Psi)$, represented by the solid arc, $\C^m$-diffeomorphically onto the region in $\mathcal{V} \times \rM_{reg}(\mathcal{T}_\Gamma, A_{ref})$ represented by the solid diagonal line. The horizontal arrow on the right is the projection of $\mathcal{V} \times \rM_{reg}(\mathcal{T}_\Gamma, A_{ref})$ to the second factor. This takes the region represented by the dotted (resp., dashed) lines in the domain onto the region represented by the dotted (resp., dashed) line in the codomain. It restricts to a $\C^m$-diffeomorphism from the region represented by the solid diagonal line onto the region in the codomain represented by the solid line. The vertical arrow is the projection of $\mathcal{V} \times \rM_{reg}(\mathcal{T}_\Gamma, A_{ref})$ to the first factor. This restricts to a $\C^m$-diffeomorphism from the region represented by the solid diagonal line onto $\mathcal{V}$. The map $\Phi$ is a $\C^m$-embedding and its image is represented by the solid line in the picture on the right (despite what the picture suggests, the dimension of the image of $\Phi$ may not equal the dimension of $\rM_{reg}(\mathcal{T}_\Gamma, A_{ref})$; {a higher dimensional illustration would be needed to accurately represent this phenomenon}).} \label{fig:4}
    \end{figure}
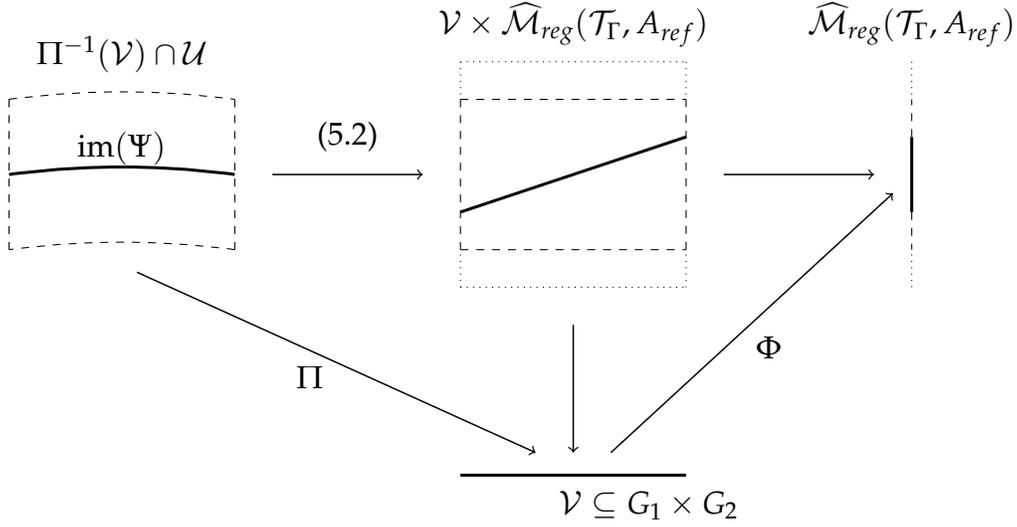

\begin{remark}\label{rem:G2point} \mbox{}
(a) We will also be interested in the case where $G_2$ consists of a single point (and so not necessarily an \emph{open} set in the $\mASD$ space). In this case, Theorem \ref{thm:2} continues to hold verbatim; there is no significant change in the proofs to account for this extension.

\medskip

(b) Recall the fiber isomorphism $\rho$ from the beginning of Section \ref{sec:GenGluing}. The usual $\ASD$ gluing results (e.g., those of \cite{taubes3,taubes4,donaldson-kronheimer1}) allows $\rho$ to vary as a ``gluing parameter''. Presumably a similar construction could be carried out in the $\mASD$ setting, with the expectation that a parametrized version of the map $\Phi$ from (a) would be a $\C^m$-diffeomorphism onto an open subset of $\rM_{reg}(\mathcal{T}_\Gamma, A_{ref})$. Though we do not pursue the details of this parametrized gluing construction here, the isomorphism $\rho$ will play an active role in our existence result of Section \ref{sec:b^+=1}.

\medskip

(c) We have chosen to phrase Theorem \ref{thm:2} in terms of the reference connection $A_{ref} = A'(A_{10}, A_{20})$ given by the preglued connection. This is only in preparation for our applications below, and this specific choice is by no means necessary. Indeed, the proof will show the connection $A_{ref}$ can be replaced by any connection that is sufficiently close to $\mathcal{J}(A_{10}, A_{20})$ in the sense that the coordinates of $A_{ref}$ and $\mathcal{J}(A_{10}, A_{20})$ satisfy the estimate (\ref{eq:hh}).
\end{remark}

We begin by giving several technical lemmas in Section \ref{sec:ImmersionLemmas}, which are used to prove that $\Phi$ is an immersion. The proof of Theorem \ref{thm:2} is given in Section \ref{sec:ProofOfThm:2}.

\subsection{Immersion lemmas}\label{sec:ImmersionLemmas}

Our ultimate goal is to show that the map $\Phi$ is an immersion. Recall this is made up of the maps $\mathcal{J}$ and $\mathcal{K}$, and hence of the maps $J, \xi, K, \zeta$ of Theorems \ref{thm:1} and \ref{thm:3}. Each of the four lemmas below establishes an estimate on the derivative of one of these latter four maps. To state the lemmas, we introduce the following seminorms on the tangent space $T_A \A^{1, p}(\mathcal{T}_\Gamma)$:   Fix an open subset $U \subseteq X$ containing $\End  X$ and let $W \in T_A \A^{1, p}(\mathcal{T}_\Gamma)$. Using the isomorphism (\ref{eq:Dnorma}), we can identify $W$ with a pair
$$(\eta, V) \in T_{p(A)} \cH_{ \Gamma} \times L^p_{1, \delta}(\Omega^1(X)).$$
Then set
$$\Vert W \Vert_{\L(U); A} \defeq \Vert \eta \Vert_{L^2_2(N)} + \Vert V \Vert_{L^p_\delta(U)} + \Vert d_A^+ V \Vert_{L^p_\delta(U)} + \Vert d_A^{*,\delta} V \Vert_{L^p_\delta(U)}$$
where the derivatives defining the norm $\Vert \cdot \Vert_{L^2_2(N)}$ on $T_{p(A)} \cH_{ \Gamma}$ are defined using the connection $\Gamma$. Then $\Vert \cdot \Vert_{\L(U); A}$ is a continuous seminorm.  These seminorms are gauge-invariant in the sense that
\eqncount\begin{equation}\label{eq:gaugeinvariantnorm}
\Vert W \Vert_{ \L(U); A} = \Vert \mathrm{Ad}(u^{-1}) W \Vert_{\L(U); u^* A}
\end{equation}
for all gauge transformations $u \in \G^{2,p}_\delta(\Gamma)$; here, via a slight abuse of notation, we are writing $\mathrm{Ad}(u^{-1}) W$ for the linearization in the direction of $W$ of the map $A \mapsto u^*A$. We use similar notation on the $X_k$. We note that if $U = X$, then $\Vert \cdot \Vert_{ \L(X); A}$ is a norm that induces the topology on $T_A \A^{1, p}(\mathcal{T}_\Gamma)$. When the metric $g_{L, \lambda}$ is relevant, we will include it in the notation by writing $\Vert \cdot \Vert_{\L(X, g_{L, \lambda}); A}$.

Our first lemma deals with the map $(A_1, A_2, \xi) \mapsto J_{A_1, A_2}(\xi)$. To first order, this map is the sum of the pregluing map $(A_1, A_2) \mapsto A'(A_1, A_2)$ together with a map that is bounded in $\xi$. We now quantify this to an extent that is sufficient for our purposes. 

\begin{lemma}\label{lem:1a}
Fix connections $A_k \in  G_k$ for $k = 1, 2$. There are constants $C, L, \lambda_0 , \eps > 0$ so that the following holds for all $0 < \lambda < \lambda_0$ and all $\xi \in  L^p_\delta(\Omega^+(X), g_{L, \lambda})$ with $\Vert \xi \Vert_{L^p_\delta(X, g_{L, \lambda})} < \eps$. Let $D J_{(A_1, A_2, \xi)}(W_1, W_2, x)$ denote the linearization at $(A_1, A_2, \xi)$ in the direction $(W_1, W_2, x)$ of the map
$$(A_1, A_2, \xi) \longmapsto J_{A_1, A_2} (\xi)$$
from Theorem \ref{thm:1} (a). Then
$$\sumdd{k=1}{2} \Vert W_k \Vert_{\L(X_k ); A_k}  \leq C\Big( \Vert D J_{(A_1, A_2, \xi)} (W_1, W_2, x) \Vert_{\L(X,g_{L, \lambda}), A'} +  \Vert x \Vert_{L^p_\delta(X,g_{L, \lambda})} \Big)$$
for all $W_k \in T_{A_k} G_k$ and all $x \in L^p_\delta(\Omega^+(X), g_{L, \lambda})$, where $A' \defeq A'(A_1, A_2)$ is the preglued connection. The constants $C, L, \lambda_0, \eps$ can be chosen to depend continuously on the $A_k \in G_k$. 
\end{lemma}

The next lemma shows that the map $(A_1, A_2) \mapsto \xi(A_1, A_2)$ depends minimally on the connections $A_1, A_2$. In the next section, this will combine with the previous lemma to show that the map $\mathcal{J}(A_1, A_2) = J_{A_1, A_2}(\xi(A_1, A_2))$ is approximately the pregluing map $(A_1,A_2) \mapsto A'(A_1, A_2)$ for $\lambda$ small; at this point it will follow that $\mathcal{J}$ is an immersion.

\begin{lemma}\label{lem:1b}
Fix $A_k \in  G_k$ for $k = 1, 2$. Then there are constants $C, L, \lambda_0 > 0$ so that the following holds for all $0 < \lambda < \lambda_0$. Let $D\xi_{(A_1, A_2)} (W_1, W_2)$ denote the linearization at $(A_1, A_2)$ in the direction $(W_1, W_2)$ of the map
$$(A_1, A_2) \longmapsto \xi(A_1 , A_2)$$
from Theorem \ref{thm:1} (c). Then this satisfies
$$\Vert D \xi_{(A_1, A_2)}(W_1, W_2) \Vert_{L^p_\delta(X, g_{L, \lambda})} \leq b^{4/p} C\sumdd{k=1}{2} \Vert W_k \Vert_{\L(X_k); A_k} $$
for all $W_k \in T_{A_k} G_k$, where $b = 4 L \lambda^{1/2}$. The constants $C, L, \lambda_0$ can be chosen to depend continuously on the $A_k \in G_k$. 
\end{lemma}

These next two lemmas are analogues of the previous two, but for the map $\mathcal{K}(A) = K_A(\zeta(A))$ in place of $\mathcal{J}(A_1, A_2) = J_{A_1, A_2}(\xi(A_1, A_2))$.

\begin{lemma}\label{lem:2a}
Fix a regular connection $A_{ref}  \in  \A^{1,p}(\mathcal{T}_{\Gamma})$. Then there are constants $C, \epsilon' > 0$ so that the following holds for all connections $A \in \rM_{reg}(\mathcal{T}_\Gamma, A_{ref})$ satisfying (\ref{eq:epsaa}) with respect to any $0 < \eps < \eps'$. Let $DK_{(A, \zeta)} (W, z)$ denote the linearization at $(A, \zeta)$ in the direction $(W, z)$ of the map
$$(A, \zeta) \longmapsto K_A (\zeta)$$
from Theorem \ref{thm:3} (a). Then this satisfies
$$\Vert W \Vert_{\L(X), A }  \leq C\Big( \Vert D K_{(A, \zeta)}(W, z) \Vert_{\L(X), A'} +  \Vert z \Vert_{L^p_\delta(X)}\Big)$$
for all $W \in T_{A} \rM_{reg}(\mathcal{T}_\Gamma, A_{ref})$ and all $z \in L^p_\delta(\Omega^+(X))$. The constants $C, \eps$ can be chosen to depend continuously on $A$ and $A_{ref}$. 
\end{lemma}

\begin{lemma}\label{lem:2b}
Fix a regular connection $A_{ref}  \in  \A^{1,p}(\mathcal{T}_{\Gamma})$. Then there are constants $C, \epsilon' > 0$ so that the following holds for all connections $A \in \rM_{reg}(\mathcal{T}_\Gamma, A_{ref})$ satisfying (\ref{eq:epsaa}) with respect to any $0 < \eps < \eps'$. Let $D\zeta_A W$ denote the linearization at $A$ in the direction $W$ of the map
$$A \longmapsto \zeta(A)$$
from Theorem \ref{thm:3} (b). Then this satisfies
$$\Vert D\zeta_A W \Vert_{L^p_\delta(X)} \leq  C\Vert d^{*, \delta}_{A_{ref}}(V - V_{ref}) \Vert_{L^p_\delta(X)}  \Vert W \Vert_{\L(X), A } $$
for all $W \in T_{A} \rM_{reg}(\mathcal{T}_\Gamma, A_{ref})$. The constants $C, \eps$ can be chosen to depend continuously on $A_{ref}$.
\end{lemma}

Now we give the proofs of Lemmas \ref{lem:1a}, \ref{lem:2a}, \ref{lem:1b}, and \ref{lem:2b}, in that order.

\begin{proof}[Proof of Lemma \ref{lem:1a}]
The tangent space $T_{A_k} G_k$ is cut out by linear elliptic equations. In particular, unique continuation holds for the elements of this tangent space, and so the assignment
$$W_k \longmapsto \Vert W_k \Vert_{\L(X_k \backslash B_{L\lambda^{1/2}(x_k)}); A_k} $$
defines a norm on $T_{A_k} G_k$. Since $T_{A_k} G_k$ is a finite-dimensional vector space, any two norms are equivalent and so there is a constant $C_1$ so that
$$\Vert W_k \Vert_{\L(X_k ); A_k}   \leq C_1 \Vert W_k \Vert_{\L(X_k \backslash B_{L\lambda^{1/2}(x_k)}); A_k}  $$
for all $W_k \in T_{A_k } G_k$. A simple contradiction argument shows that this constant can be taken to be independent of $L, \lambda$, provided $L \geq 1$ and $\lambda$ is sufficiently small. 

Now fix tangent vectors $W_k \in T_{A_k} G_k$ for $k = 1, 2$. Since $A_k$ is regular, we can find a $\C^m$-smooth path $A_k(\tau)$ of $A_k$-regular $\mASD$ connections with $A_k(0) = A_k$ and $\frac{d}{d\tau} \vert_{\tau = 0} A_k(\tau) = W_k$. Let $W' = \frac{d}{d\tau} \vert_{\tau = 0} A'(A_1(\tau), A_2(\tau))$. Note that the construction of the preglued connection $A'(A_1, A_2)$ implies there is a uniform constant $C_2$ so that
$$\sumd{k} \Vert W_k \Vert_{\L(X_k \backslash B_{L\lambda^{1/2}}(x_k)); A_k}  \leq C_2  \left\| W' \right\|_{\L(X); A'(A_1, A_2)} $$
provided $\lambda> 0$ is sufficiently small. Thus we have
\eqncount\begin{equation}\label{eq:starthere}
\sumd{k} \Vert W_k \Vert_{\L(X_k ); A_k}  \leq C_1 C_2  \left\| W' \right\|_{\L(X); A'}.
\end{equation}

The next claim ties this in with the linearization of the map $J$ at $(A_1, A_2, \xi)$ when $\xi = 0$.

\medskip

\noindent \emph{Claim 1: $DJ_{(A_1, A_2, 0)}(W_1, W_2, x) = W' + (D \iota)_{\iota^{-1}(A')} P x \hspace{1cm} \forall x \in L^p_\delta(\Omega^+(X))$.}

\medskip

Here $P$ is the right-inverse from the proof Theorem \ref{thm:1}. This depends on $A_1, A_2$, so to emphasize this, we will temporarily write $P_{A_1, A_2} \defeq P$. Consider the map
\eqncount\begin{equation}\label{eq:psiphi}
(A, \xi) \longmapsto  \iota \circ \exp_{\iota^{-1}(A)} \big(P_{A_1, A_2} \xi\big)
\end{equation}
where $A$ ranges over all connections near $A' = A'(A_1, A_2)$ and $\xi$ ranges over all self-dual 2-forms near 0. The linearization of (\ref{eq:psiphi}) at $(A', 0)$ is the operator
$$(W, x) \longmapsto W + (D \iota)_{\iota^{-1}(A')} \big(P_{A_1, A_2} x\big).$$
Recall from the proof of Theorem \ref{thm:1} that
$$J_{A_1, A_2} (\xi) = \iota \left( \exp_{\iota^{-1}(A'(A_1, A_2))} \big( P_{A_1, A_2} \xi \big) \right).$$
That is, $(A_1, A_2, \xi) \mapsto J_{A_1, A_2}(\xi)$ is the map (\ref{eq:psiphi}) precomposed with $A'(A_1, A_2)$ in the $A$-component. Then Claim 1 follows from the chain rule and the fact that we are differentiating at $\xi  = 0$, which kills off all terms involving the $A_k$-derivatives of $P_{A_1, A_2}$. 

In summary, we have
$$\begin{array}{rcl}
\sumd{k} \Vert W_k \Vert_{\L(X_k ); A_k} & \leq & C_1 C_2 \left\| W' \right\|_{\L(X); A'(A_1, A_2)}\\
& \leq & C_1 C_2 \Big( \Vert  W' + (D \iota)_{\iota^{-1}(A')} P x \Vert_{\L(X); A'} + \Vert (D \iota)_{\iota^{-1}(A')} P x \Vert_{\L(X); A'} \Big)\\
  &= &C_1 C_2  \Big( \Vert  DJ_{(A_1, A_2, 0)}(W_1, W_2, x) \Vert_{\L(X); A'}   + \Vert (D \iota)_{\iota^{-1}(A')} P x \Vert_{\L(X); A'} \Big).
  \end{array}$$
  We will discuss each term on the right individually.
  
The first term on the right is almost satisfactory, except we linearized at $(A_1, A_2, 0)$ instead of $(A_1, A_2, \xi)$. To account for this, note that it follows from our regularity assumptions and Theorem \ref{thm:1} that $J$ is $\C^m$-smooth. In particular, Taylor's theorem gives
$$\begin{array}{l}
\Vert (D J_{(A_1, A_2, \xi)} - DJ_{(A_1, A_2, 0)}) (W_1, W_2, x) \Vert_{\L(X); A'} \\
\hspace{4cm} \leq  C_3  \Vert \xi \Vert_{L^p_\delta(X)} \Big( \Vert x \Vert_{L^p_\delta(X)} + \sumd{k} \Vert W_k \Vert_{\L(X_k); A_k} \Big)
\end{array}$$
for some constant $C_3$ that depends continuously on the $A_k$ and $\lambda$. 

\medskip

\noindent \emph{Claim 2: The constant $C_3$ can be taken to be independent of $\lambda$, provided $\lambda$ is sufficiently small.} 

\medskip

To see this, recall that the proof of Taylor's theorem shows that $C_3$ can be taken to be a constant multiple of the supremum of the operator norm of the second derivative of $J$ at $(A_1, A_2, 0)$. By the chain rule, it therefore suffices to uniformly estimate the first two derivatives of $\iota,  \exp_{\iota^{-1}(A'(A_1, A_2))}$ and $P = P_{A_1, A_2}$. Obtaining such estimates for $\iota$ and the exponential map follow readily because the gluing region is in the complement of the cylindrical end (e.g., $\iota$ is affine-linear over this gluing region). That the derivatives of $P$ are uniformly bounded is addressed in Remark \ref{rem:GaugeSlice} (a), above.

With this claim in hand, we have
$$\begin{array}{rcl}
\sumd{k} \Vert W_k \Vert_{\L(X_k); A_k}   &\leq  &C_1 C_2  \Big( \Vert  DJ_{(A_1, A_2, \xi)}(W_1, W_2, x) \Vert_{\L(X); A'}   + \Vert (D \iota)_{\iota^{-1}(A')} P x \Vert_{\L(X); A'} \Big)\\
 && +   C_1 C_2 C_3  \Vert \xi \Vert_{L^p_\delta(X)} \Big( \Vert x \Vert_{L^p_\delta(X)} + \sumd{k} \Vert W_k \Vert_{\L(X_k); A_k} \Big)
\end{array}$$
When $\Vert \xi \Vert_{L^p_\delta(X)} < \eps \defeq 1/ 2 C_1 C_2C_3$ this implies that $\sumd{k} \Vert W_k \Vert_{\L(X_k); A_k} $ is bounded by
$$\begin{array}{c}
2C_1 C_2  \Big( \Vert  DJ_{(A_1, A_2, \xi)}(W_1, W_2, x) \Vert_{\L(X); A'}   + \Vert (D \iota)_{\iota^{-1}(A')} P x \Vert_{\L(X); A'} \Big)+ \Vert x \Vert_{L^p_\delta(X)}
\end{array}$$
The lemma now follows from the next claim.

\medskip

\noindent \emph{Claim 3: There are constants $C_4, L, \lambda_0 >0$ so that}
$$\Vert (D \iota)_{\iota^{-1}(A')} P x \Vert_{\L(X, g_{L, \lambda}); A'} \leq C_4 \Vert x \Vert_{L^p_\delta(X, g_{L, \lambda})}$$
\emph{for all $x$ and all $0 < \lambda < \lambda_0$. These constants can be chosen to depend continuously on the $A_k \in \A^{1,p}(\mathcal{T}_{k, \Gamma_k})$.}

\medskip

We briefly sketch the proof, leaving the details to the reader. Use the fact that $P$ is uniformly bounded to control the zeroth order terms appearing in the definition of $\Vert \cdot \Vert_{\L(X); A'}$. {The term involving $d^{*, \delta}_{A'}$ vanishes because $P$ takes values in the slice}. To control the $d_{A'}^+$ term use the fact that $P$ is a right inverse to an operator that is essentially $d^+_{A'}$ plus lower order terms.
\end{proof}

\begin{proof}[Proof of Lemma \ref{lem:2a}]
Fix $A$ and $W$ as in the lemma. Let $A(\tau)$ be a path in $\rM_{reg}(\mathcal{T}_\Gamma, A_{ref})$ that is $\C^m$-smooth and satisfies $A(0) = A$ and $\frac{d}{d\tau} \vert_{\tau = 0} A(\tau) = W$. Let $\mu_\tau = \mu(A(\tau), A_{ref})$ be the 0-form from Proposition \ref{prop:1} associated to $A(\tau)$ and $A_{ref}$. Set 
$$u_\tau \defeq \exp(\mu_\tau), \indent A' \defeq u_0^* A , \indent W' \defeq \frac{d}{d\tau} \Big|_{\tau = 0} u_\tau^*A(\tau).$$
By the product rule, we have
$$W' = \mathrm{Ad}(u_0^{-1}) W + d_{A'} \Big(\frac{d}{d\tau} \Big|_{\tau = 0} \mu_\tau \Big).$$
Now the gauge invariance (\ref{eq:gaugeinvariantnorm}) and the definition of our norms give
$$\begin{array}{rcl}
\Vert W \Vert_{\L(X); A} & = & \Vert \mathrm{Ad}(u_0^{-1}) W \Vert_{\L(X); A'} \\
& \leq & \Vert W' \Vert_{\L(X); A'}  + \Big\| d_{A'} \Big(\frac{d}{d\tau} \Big|_{\tau = 0} \mu_\tau \Big)\Big\|_{\L(X); A'}\\
& = & \Vert W' \Vert_{\L(X); A'} + \Big\| d_{A'} \Big(\frac{d}{d\tau} \Big|_{\tau = 0} \mu_\tau \Big)\Big\|_{L^p_\delta(X) }\\
&&+ \Big\| \left[ F_{A'}^+,  \Big(\frac{d}{d\tau} \Big|_{\tau = 0} \mu_\tau \Big) \right] \Big\|_{L^p_\delta(X)} + \Big\| d_{A'}^{*, \delta} d_{A'} \Big(\frac{d}{d\tau} \Big|_{\tau = 0} \mu_\tau \Big)\Big\|_{L^p_\delta(X)}.
\end{array}$$
Focusing on the second term on the right, we note that the operator $d_{A'}^{*, \delta}$ is injective on $\mathrm{im}(d_{A'})$ so there is a bound of the form
$$ \Big\| d_{A'} \Big(\frac{d}{d\tau} \Big|_{\tau = 0} \mu_\tau \Big)\Big\|_{L^p_\delta (X)}\leq  C_1\Big\| d_{A'}^{*, \delta} d_{A'} \Big(\frac{d}{d\tau} \Big|_{\tau = 0} \mu_\tau \Big)\Big\|_{L^p_\delta(X)}.$$
As for the third term on the right, the fact that $A$ is $\mASD$ implies that $F_A^+$ is uniformly bounded in $\C^0$; the same is therefore true of $F_{A'}^+ = \mathrm{Ad}(u_0^{-1}) F_A^+$. Combining this with the fact that the operator $d_{A'}^{*, \delta} d_{A'}$ is injective on 0-forms, we obtain
$$\begin{array}{rcl}
 \Big\| \left[ F_{A'}^+,  \Big(\frac{d}{d\tau} \Big|_{\tau = 0} \mu_\tau \Big) \right] \Big\|_{L^p_\delta(X)} & \leq &  C_2 \Big\| \frac{d}{d\tau} \Big|_{\tau = 0} \mu_\tau \Big\|_{L^p_\delta(X)}\\
& \leq & C_3   \Big\| d_{A'}^{*, \delta} d_{A'} \Big(\frac{d}{d\tau} \Big|_{\tau = 0} \mu_\tau \Big)\Big\|_{L^p_\delta(X)}.
\end{array}$$
In summary, we have
$$\begin{array}{rcl}
\Vert W \Vert_{\L(X); A} & \leq & \Vert W' \Vert_{\L(X); A'} + (1+C_1 + C_3) \Big\| d_{A'}^{*, \delta} d_{A'} \Big(\frac{d}{d\tau} \Big|_{\tau = 0} \mu_\tau \Big)\Big\|_{L^p_\delta(X)}
\end{array}$$
Our hypotheses imply that $A'$ and $A_{ref}$ differ by a term that is controlled by the $\C^0$-norm of $\mu$. This implies we have an estimate of the form
$$\Big\| d_{A'}^{*, \delta} d_{A'} \Big(\frac{d}{d\tau} \Big|_{\tau = 0} \mu_\tau \Big)\Big\|_{L^p_\delta} \leq C_4 \Big\| d_{A_{ref}}^{*, \delta} d_{A'} \Big(\frac{d}{d\tau} \Big|_{\tau = 0} \mu_\tau \Big)\Big\|_{L^p_\delta}.$$
To estimate this further, differentiate the defining identity $0 = d_{A_{ref}}^{*, \delta}(u_\tau^* A_\tau - A_{ref})$ at $\tau = 0$ to get
$$d_{A_{ref}}^{*, \delta} d_{A'} \Big(\frac{d}{d\tau} \Big|_{\tau = 0} \mu_\tau \Big)= - d_{A_{ref}}^{*, \delta} W.$$
Thus
$$\begin{array}{rcl}
\Big\| d_{A_{ref}}^{*, \delta} d_{A'} \Big(\frac{d}{d\tau} \Big|_{\tau = 0} \mu_\tau \Big)\Big\|_{L^p_\delta} & = & \Big\| d_{A_{ref}}^{*,\delta} W \Big\|_{L^p_\delta}\\
& \leq & \Big\| d_{A}^{*,\delta} W \Big\|_{L^p_\delta} + C_5 \Vert W \Vert_{L^p_\delta}\\
& = & \Big\| d_{A'}^{*,\delta}  W' \Big\|_{L^p_\delta} + C_5 \Vert W' \Vert_{L^p_\delta}\\
& \leq & \max(1, C_5)\Vert W' \Vert_{\L(X); A'}.
\end{array}$$
Hence
$$\Vert W \Vert_{\L(X); A} \leq C_6 \Vert W' \Vert_{\L(X); A'}$$
where $C_6 =  1 + (1+ C_1 + C_3)C_4 \max(1, C_5)$. To finish the proof of the lemma, argue exactly as we did in the proof of Lemma \ref{lem:1a}, starting after the estimate (\ref{eq:starthere}).
\end{proof}

\begin{proof}[Proof of Lemma \ref{lem:1b}]
Let $A_k(\tau)$ be a $\C^m$-smooth path in $G_k$ satisfying $A_k(0) = A_k$ and $\frac{d}{d\tau}\vert_{\tau = 0} A_k(\tau) = W_k$. Set $\xi_\tau\defeq \xi(A_1(\tau), A_2(\tau))$. Note that the $\tau$-derivative 
$$\frac{d}{d\tau} \Big|_{\tau = 0} \xi_\tau = D \xi_{(A_1, A_2)}(W_1, W_2)$$
is the term that we are looking to bound. 

The regularity hypotheses and Theorem \ref{thm:1} (c) imply that $\xi_\tau$ satisfies $ 0  = s(J_\tau(\xi_\tau)) $ for all $\tau$, where $J_\tau \defeq J_{A_1(\tau), A_2(\tau)}$ is the map from Theorem \ref{thm:1} (a). Continuing to use a subscript $\tau$ for any term defined in terms of the $A_k(\tau)$ (and hence dependent on $\tau$), we recall the definition of $\widetilde{s}$ from the proof of Theorem \ref{thm:1}; in particular, this satisfies $s(J_\tau( \cdot)) = \widetilde{s}(P_\tau ( \cdot))$. The Taylor expansion of $\widetilde{s}$ therefore gives
$$0 = s(J_\tau(\xi_\tau)) = \widetilde{ s}(P_\tau(\xi_\tau)) = \widetilde{s}_\tau(0, 0) + \xi_\tau + \St_\tau(\xi_\tau).$$
Differentiate the right-hand side at $\tau = 0$ and rearrange to get
\eqncount\begin{equation}\label{eq:xider}
\frac{d}{d\tau} \Big|_{\tau = 0} \xi_\tau = - \frac{d}{d\tau} \Big|_{\tau = 0} \widetilde{s}_\tau(0, 0) - \frac{d}{d\tau} \Big|_{\tau = 0} \St_\tau (\xi_0) - (D\St_0)_{\xi_0} \Big(\frac{d}{d\tau} \Big|_{\tau = 0} \xi_\tau\Big)
\end{equation}
where $(D\St_0)_{\xi_0}$ is the linearization at $\xi_0$ of $\St_0$. We will return to this after we estimate each term on the right individually. 

For the first term on the right of (\ref{eq:xider}), note that $\widetilde{s}_\tau(0, 0) = s(A'_\tau)$ depends on $\tau$ only through the preglued connection $A'_\tau \defeq A'(A_1(\tau), A_2(\tau))$. Moreover, the proof of (\ref{eq:our727}) shows that $s(A'_\tau)$ is equal to a product of a cutoff function supported in the gluing region, times the connection form for $A_\tau'$ in this region. In particular, differentiating this in $\tau$, the same argument used for (\ref{eq:our727}) allows us to conclude a uniform bound of the form
$$\begin{array}{rcl}
\Big\| \frac{d}{d\tau} \Big|_{\tau = 0} \widetilde{s}_\tau(0, 0) \Big\|_{L^p_\delta(X, g_{L, \lambda})} & \leq & C_1 b^{4/p} \Big\| \frac{d}{d\tau} \Big|_{\tau = 0} A_\tau' \Big\|_{\L(X,g_{L, \lambda}), A'_0} \\
& \leq & C_2 b^{4/p}  \sumd{k} \Big\| W_k \Big\|_{\L(X); A_k}
\end{array}$$
where the second inequality follows by differentiating the defining formula for the preglued connection $A'(A_1, A_2)$. This is the desired bound on the first term. 

The second term on the right of (\ref{eq:xider}) is similar, albeit a little more involved. The point here is that the quadratic estimates on $\St_\tau$ give a uniform bound of the form
$$\Big\| \frac{d}{d\tau} \Big|_{\tau = 0} \St_\tau (\xi_0) \Big\|_{L^p_\delta(X, g_{L, \lambda})} \leq C_3 \Vert \xi_0 \Vert_{L^p_\delta(X, g_{L, \lambda})} \sumd{k} \Big\|  W_k\Big\|_{T_{A_k} \A} .$$
Theorem \ref{thm:1} (c) gives $\Vert \xi_0 \Vert_{L^p_\delta(X, g_{L, \lambda})} \leq C_4 b^{4/p}$, so the desired estimate for this term follows.

Turn now to the last term on the right of (\ref{eq:xider}). By the estimate (\ref{eq:quad}), the linearization $(D \St_0)_{\xi_0}$ satisfies
$$\Vert (D \St_0)_{\xi_0} \xi' \Vert_{L^p_\delta(X, g_{L, \lambda})} \leq 2 \kappa \Vert \xi_0 \Vert_{L^p_\delta(X, g_{L, \lambda})} \Vert \xi' \Vert_{L^p_\delta(X, g_{L, \lambda})}$$
for all $\xi'$. Since $\Vert \xi_0 \Vert_{L^p_\delta(X, g_{L, \lambda})} \leq C_4 b^{4/p}$, we may assume that $\Vert \xi_0 \Vert_{L^p_\delta(X, g_{L, \lambda})} < 1/4\kappa$, which gives
$$\Vert (D \St_0)_{\xi_0} \xi' \Vert_{L^p_\delta(X, g_{L, \lambda})}  \leq \frac{1}{2} \Vert \xi' \Vert_{L^p_\delta(X, g_{L, \lambda})}.$$

To see that the above estimates imply the lemma, take the norm of each side of (\ref{eq:xider}) and use the estimates just established to obtain
$$ \Big\| \frac{d}{d\tau} \Big|_{\tau = 0} \xi_\tau \Big\|_{L^p_\delta(X, g_{L, \lambda})}\\
 \leq  (C_2+ C_3 C_4) b^{4/p} \sumd{k} \Big\| W_k \Big\|_{T_{A_k} \A} + \frac{1}{2} \Big\| \frac{d}{d\tau} \Big|_{\tau = 0} \xi_\tau \Big\|_{L^p_\delta(X, g_{L, \lambda})}.$$
The corollary follows by subtracting the last term from both sides, and using the identity $D\xi_{(A_1, A_2)}(W_1, W_2) =\frac{d}{d\tau} \Big|_{\tau = 0} \xi_\tau$.
\end{proof}

\begin{proof}[Proof of Lemma \ref{lem:2b}]
This follows from the same type of argument given for Lemma \ref{lem:1b}.
\end{proof}

\subsection{Proof of Theorem \ref{thm:2}}\label{sec:ProofOfThm:2}
Let $\epsilon > 0$ be small enough so that Theorem \ref{thm:3} holds with this value of $\epsilon$. Define  $\mathcal{U}$ to be the set of triples $(A_1, A_2, A)$ with $A_k \in G_k$ and $A \in \rM(\mathcal{T}_\Gamma, \mathcal{J}(A_1, A_2))$, and so that
\eqncount\begin{equation}\label{eq:hh}
 \Vert h - h_{0} \Vert_{L^2_2(N)} + \Vert V - V_{0} \Vert_{L^{p^*}_{\delta}(X)} + \Vert d^{*, \delta}_{A_{0}}(V - V_{0}) \Vert_{L^p_\delta(X)}  < \epsilon/3
 \end{equation}
where $A = \iota(h, V)$ and $A_{0} = \iota(h_{0}, V_{0}) \defeq \mathcal{J}(A_1, A_2)$. Since all elements of the $G_k$ are regular, it follows from Theorem \ref{thm:1} that $\mathcal{J}(A_1, A_2)$ is regular, so any connection $A$ satisfying (\ref{eq:hh}) is automatically $A_0$-regular. Thus $\mathcal{U} \subseteq \mathcal{E}$. 

To show that $\Pi\vert_{\mathcal{U}}\colon \mathcal{U} \rightarrow G_1 \times G_2$ is locally trivial, fix $(A_{10}, A_{20}) \in G_1 \times G_2$, and set $A_{ref} \defeq A'(A_{10}, A_{20})$. By (\ref{eq:closeinA}) and the fact that $\mathcal{J}(A_{10}, A_{20})$ takes values in the $A_{ref}$-slice, by choosing $\lambda$ sufficiently small, it follows that the coordinates of $A_{ref}$ and $\mathcal{J}(A_{10}, A_{20})$ satisfy the estimate (\ref{eq:hh}); in fact, this estimate is uniform in $\lambda$, in the sense that it holds for all sufficiently small $\lambda$.  Fix any such $\lambda$; we will refine this choice in the next paragraph. Take $\mathcal{V} \subseteq G_1 \times G_2$ to be a neighborhood of $(A_{10}, A_{20})$ that is small enough so that if $(A_1, A_2) \in \mathcal{V}$, then the components of $A_{ref}$ and $A'(A_1, A_2)$ satisfy (\ref{eq:hh}). Though we do not use this observation presently, we note that the set $\mathcal{V}$ can also be chosen to be uniform in $\lambda$, provided $\lambda > 0$ is sufficiently small; this is due to the scaling properties the $L^{p^*}$-norm of 1-forms \cite[p.293]{donaldson-kronheimer1}. Use two applications of the triangle inequality to conclude that, for any triple $(A_1, A_2, A) \in \Pi^{-1}(\mathcal{V}) \cap \mathcal{U}$, the pair $A, A_{ref}$ satisfies the hypotheses of Theorem \ref{thm:3}. Thus the map (\ref{eq:trivie}) is well-defined and indeed provides a local trivialization of $\Pi\vert_{\mathcal{U}}$.

To finish the proof, it suffices to show that the map $\Phi  = \mathcal{K} \circ \mathcal{J} \colon  \mathcal{V} \rightarrow \rM_{reg}(\mathcal{T}_\Gamma, A_{ref})$ is a $\C^m$-immersion; it can then be made into an embedding by further shrinking $\mathcal{V}$, if necessary. Since $\Phi = \mathcal{K} \circ \mathcal{J}$, it suffices to show that $\mathcal{J}$ and $\mathcal{K}$ are immersions for all $\lambda > 0$ sufficiently small. For $\mathcal{J}$, note that the linearization at $(A_1, A_2)$ is the map
$$(W_1, W_2) \longmapsto D J_{(A_1, A_2, \xi(A_1, A_2))}(W_1, W_2, D\xi_{(A_1, A_2)}(W_1, W_2)).$$
Suppose this vanishes at some $(W_1, W_2)$. Then by Lemmas \ref{lem:1a} and \ref{lem:1b}, we would have
$$\sumd{k} \Vert  W_k \Vert_{\L(X_k); A_k} \leq C b^{4/p} \sumd{k} \Vert  W_k \Vert_{\L(X_k); A_k} $$
By taking $\lambda > 0$ sufficiently small, we may assume $Cb^{4/p} < 1$ and so $W_k = 0$. Thus $\mathcal{J}$ is an immersion. A similar argument, but using Lemmas \ref{lem:2a} and \ref{lem:2b}, shows that $\mathcal{K}$ is an immersion for small $\lambda$. 

The irreducibility claims follow from the analogous claims appearing in Theorems \ref{thm:1} and \ref{thm:3}. \qed

\section{Existence results}

Let $X$ be an oriented cylindrical end 4-manifold with $b^+(X) = 0$ or $1$. In this section, we will show how to use the above framework to prove the existence of families of $\mASD$ connections on $X$; the cases $b^+(X) = 0$ and $b^+(X) = 1$ are treated in Sections \ref{sec:b^+=0} and \ref{sec:b^+=1}, respectively. The $\ASD$ existence result Theorem \ref{thm:A} is proved in Section \ref{sec:ProofOfThmA}. 

Part of our existence results state that the connections we construct are topologically non-trivial in a certain sense. In the case of closed 4-manifolds, this non-triviality is captured by the non-vanishing of a characteristic class of the bundle supporting the connections. In the present cylindrical end setting, we will use a certain relative characteristic class to measure this non-triviality. The details of this are carried out in Section \ref{sec:AdaptedBundles}.

\subsection{Relative characteristic classes and adapted bundles}\label{sec:AdaptedBundles}

This section reviews topological quantities associated to 4-manifolds with cylindrical ends. We begin with a review of characteristic classes in the closed (compact with no boundary) setting.

Suppose $Z$ is a closed, oriented 4-manifold and $P \rightarrow Z$ is a principal $G$-bundle. We define
$$\kappa(P) \defeq  -\frac{1}{8 \pi^2} \intd{Z} \langle F_A \wedge F_A\rangle $$
where $A \in \A(P)$ is any connection and $\langle F_A \wedge F_A\rangle$ is obtained by combining the wedge and the inner product on $\mathfrak{g}$ defined via the immersion (\ref{eq:rep}). Then $\kappa(P)$ is independent of the choice of $A$ by the Bianchi identity. Topologically, $\kappa(P) = c_2(P \times_G \bb{C}^r)\left[Z \right]$ is the second Chern number of the $\bb{C}^r$-bundle associated to $P$ via the map (\ref{eq:rep}) and the standard action of $\SU(r)$ on $\bb{C}^r$. In particular, $\kappa(P) \in \bb{Z}$ is an integer representing an obstruction to $P$ being trivializable.

Now consider the bundle $Q \rightarrow N$ over the 3-manifold $N$, and fix a gauge transformation $u \in \G(Q)$. We can form the mapping torus $Q_u = \left[0, 1 \right] \times Q  / (0, u(q)) \sim (1 , q)$ which is a principal $G$-bundle over $S^1 \times N$. Then we define the \emph{degree} of $u$ to be the integer
$$\deg(u) \defeq \kappa(Q_u).$$
This depends only on the homotopy type of $u$ and so descends to a group homomorphism
$$\deg\colon  \pi_0(\G(Q)) \longrightarrow \bb{Z}$$
from the group of components of $\G(Q)$. The degree is an obstruction to extending $u$ to a gauge transformation on $X_0$ (or equivalently $X$). We denote by $\G_0(Q)$ the subgroup of degree-zero gauge transformations. When $G$ is simply-connected, the degree $\deg\colon  \pi_0(\G(Q)) \rightarrow \bb{Z}$ is injective, and so $\G_0(Q)$ is exactly the identity component of $\G(Q)$. 

Since the cylinder $\End X$ deformation retracts to the 3-manifold $N$, we have a natural isomorphism $\pi_0(\G(\End  X)) \cong \pi_0(\G(N))$ and so the degree provides a homomorphism $\deg\colon  \pi_0(\G(\End  X)) \rightarrow \bb{Z}$. We denote by $\G_0(\End  X)$ the degree-zero elements of $\G(\End  X)$.

We will be working with principal $G$-bundles on the cylindrical end 4-manifold $X$. Bundle isomorphism is too course of an equivalence relation to be useful in the cylindrical-end setting (e.g., when $G$ is simply-connected, all principal $G$-bundles are trivializable since $H^4(X) = 0$). A more useful relation for our purposes deals with \emph{adapted bundles}, which are pairs $(E, A_{\End})$, where $E \rightarrow X$ is a principal $G$-bundle, and $A_{\End}$ is a connection on the cylindrical end $\End  X$. Then we say that $(E, A_{\End})$ is \emph{equivalent} to $(E', A_{\End}')$ if there is a bundle isomorphism from $E$ to $E'$ that carries $A_{\End}$ to $A_{\End}'$. See Donaldson's book \cite[Section 3.2]{donaldson10} for more details; note that Donaldson only treats flat connections $A_{\End}$, but our applications require that we extend the discussion.

By the above discussion, it follows that $\G_0(\End  X)$ consists of gauge transformations on $\End  X$ that have extensions to $E \rightarrow X$. Thus, any adapted bundle $(E, A_{\End})$ depends on $A_{\End}$ only through its $\G_0(\End  X)$-equivalence class. The next example illustrates an interplay between the degree and the equivalence classes of adapted bundles; it will be relevant to our gluing discussion below. 

\begin{example}\label{ex:degree}
Fix an adapted bundle $(E, A_{\End})$ and a point $x \in X$. Suppose $E_\ell \rightarrow S^4$ is a principal $G$-bundle with $\kappa(E_\ell) = \ell \in \bb{Z}$. Taking the connected sum of $X$ and $S^4$ at $x$, we recover the same manifold back $X \cong X \# S^4$, up to diffeomorphism. At the bundle level, we can carry out a similar connected sum procedure to obtain a bundle $E' = E \# E_\ell$ over $X$. Provided $x \in X_0$ is not on the end, the connection $A_{\End}$ can be viewed as a connection on $E'$. Then the adapted bundles $(E, A_{\End})$ and $(E' , A_{\End})$ are equivalent if and only if $\ell = 0$. More generally, there is a gauge transformation $u$ on $\End  X$ with $\deg(u) = \ell$, and so that the adapted bundle $(E, u^* A_{\End})$ is equivalent to $(E', A_{\End})$. 
\end{example}

Assume that the connection $A_{\End}$ converges on the end in the sense that
$$\lim_{t \rightarrow \infty } A_{\End} \vert_{\left\{t \right\} \times N}  = \Gamma$$
for some connection $\Gamma$ on $N$, where the limit is in $L^2_1(N)$, say. Let $A$ be any connection on $E$ that restricts on $\End  X$ to $A_{\End}$. Then the quantity
$$\kappa(E, A_{\End})   \defeq  \limd{T \rightarrow \infty} -\frac{1}{8 \pi^2} \int_{X_0 \cup \left[0, T \right] \times N} \langle F_A \wedge F_A\rangle  $$
is well-defined and independent of the choice of $A$. We will call $\kappa(E, A_{\End}) $ the \emph{relative characteristic number} of the adapted bundle $(E, A_{\End})$. It depends on $A_{\End}$ only through the value of $\Gamma$ and the topological type of $E$. Indeed, if $E' = E\# E_\ell$ is as in Example \ref{ex:degree}, then
$$\kappa(E', A_{\End})  = \kappa(E, A_{\End}) + \ell.$$
Moreover, if $A_{\End}$ is asymptotic to $\Gamma$, then working modulo $\bb{Z}$, we recover 
$$ 8 \pi^2 \kappa(E, A_{\End}) = \CS(\Gamma) \indent \mod \bb{Z}$$
the Chern--Simons value of $\Gamma$ as defined in \cite[Section 2.1]{\mmr} (here one should interpret the trace in \cite{\mmr} as the one induced from (\ref{eq:rep})).

\subsection{Existence when $b^+(X) = 0$}\label{sec:b^+=0}

Let $E_{\mathrm{triv}} \rightarrow X$ be the trivial bundle, $A_{\mathrm{triv}}$ the trivial connection on $E_{\mathrm{triv}}$, and $\Gamma_{\mathrm{triv}}$ the trivial connection on the end. Fix thickening data $\mathcal{T}_{\Gamma_{\mathrm{triv}}}$. Here we assume $\delta$ and $\beta$ are chosen as in the beginning of Section \ref{sec:GluingRegFam}. We recall from Section \ref{sec:GaugeTheory} that the thickening data also includes the choice of $\eps_0 > 0$ so that any two points in the center manifold have Chern--Simons values differing by $\eps_0/2$. For each $0 < \eps < \eps_0$, we will write $\mathcal{T}(\eps)$ for the same set of thickening data as $\mathcal{T}_{\Gamma_{\mathrm{triv}}}$, but with $\eps$ in place of $\eps_0$. 

Let $E_\ell \rightarrow S^4$ be a principal $G$-bundle with $\kappa(E_\ell) = \ell \in \bb{Z}$, where $\kappa$ is the characteristic number of Section \ref{sec:AdaptedBundles}. We will write 
$$\M_\ell(S^4, G) \defeq  \left\{ A \in \A(E_\ell) \; \vert \; F^+_A = 0 \right\} \big/ \:\G(E_\ell)$$
 for the moduli space of $\ASD$ connections on $E_\ell$; here we are working relative to the standard metric on $S^4$. Let $\M_\ell^*(S^4 , G) \subseteq \M_\ell(S^4, G)$ denote the subset of irreducible $\ASD$ connections. The existence of irreducible $\ASD$ connections on $S^4$ was studied extensively in \cite[Section 8]{AHS}. For example, when $G = \SU(r)$ and the embedding (\ref{eq:rep}) is the identity, then the space $\M_\ell^*(S^4 , \SU(r))$ is nonempty if and only if $\ell \geq r/2 $. The most famous situation is when $G = \SU(2)$ and $\ell = 1$, in which case $\M_1(S^4, \SU(2))=\M^*_1(S^4, \SU(2))$ and this is diffeomorphic to the open unit ball in $\bb{R}^5$. The dimension of $\M^*_\ell(S^4, G)$ for general simple, simply-connected $G$ is given in \cite[Table 8.1]{AHS}. 

The following is the first of our main existence results for $\mASD$ connections; Theorem \ref{thm:C} with $b^+(X) = 0$ from the introduction is an immediate consequence.

\begin{theorem}[Existence of $\mASD$-connections when $b^+ = 0$]\label{thm:exist1}
Assume $b^+(X) = 0$ and $A_\ell$ is an irreducible $\ASD$ connection on $E_\ell \rightarrow S^4$ for some $\ell \in \bb{Z}$. Then for every $0 < \eps < \eps_0$, there is

\begin{itemize}
\item[(a)] a neighborhood $\mathcal{V} \subseteq \M_\ell(S^4, G)$ of $\left[A_\ell \right] \in \M_\ell(S^4, G)$;
\item[(b)] a trivializable principal $G$-bundle $E \rightarrow X$ that is canonically trivial on the end;
\item[(c)] a connection $A'$ on $E$ that is flat in the complement of a compact set, asymptotic to $\Gamma$ on the end, and satisfies $\kappa(E, A' \vert_{\End X})  = \ell$; and
\item[(d)] a $\C^m$-embedding 
$$\Phi: \mathcal{V}  \longrightarrow \rM_{reg}(\mathcal{T}(\eps), A').$$
\end{itemize}
The image of $\Phi$ consists of irreducible connections. In particular, there exist an irreducible, regular $\mASD$ connection $A$ on $E$ with $\vert \kappa(E, A \vert_{\End X}) -  \ell \vert < \eps/2$. 
\end{theorem}

{For $\ell \neq 0$ and $0< \epsilon < 2 \vert \ell \vert$}, the condition $\vert \kappa(E, A \vert_{\End X}) -  \ell \vert < \eps/2$ implies that $A$ is not flat. (The analogous statement in the case where $X$ is closed and $G = \SU(r)$ is that $A$ is supported on a bundle $P$ with $c_2(P)\left[ X \right] = \ell$.) 

\begin{proof}[Proof of Theorem \ref{thm:exist1}] 
View $S^4$ as a cylindrical end 4-manifold with no ends, and let $\mathcal{T}_\emptyset$ be the empty set of thickening data as in Section \ref{sec:ClosedX}. Form the connected sum of $S^4$ and $X$ at any point in $S^4$ and any point in $X$ lying in the interior of the compact part. Note that all $\ASD$ connections on $S^4$ are regular (e.g., use the Weitzenb\"{o}ck formula \cite[(7.1.23)]{donaldson-kronheimer1}). In particular, since $A_\ell$ is irreducible and regular, there is a neighborhood $\mathcal{V} \subseteq \M^*_\ell(S^4, G)$ of $\left[A_\ell \right]$ that is diffeomorphic to a precompact open set $G_1 \subseteq \rM_{reg}(\mathcal{T}_\emptyset, A_\ell)$ containing $A_\ell$; that is, $\mathcal{V}$ consists of gauge equivalence classes of regular, irreducible $\ASD$ connections on $S^4$, and $G_1$ consists of their lifts to the Coulomb slice through $A_\ell$. Using this diffeomorphism, we identify $\mathcal{V}$ and $G_1$.
 
By Proposition \ref{prop:index}, the assumption $b^+ (X) = 0$ implies that $A_{\mathrm{triv}}$ is regular; see also Remark \ref{rem:RegularityRemark}. Then the singleton set $G_2 \defeq \left\{A_{\mathrm{triv}} \right\}$ plainly consists of regular connections. By Remark \ref{rem:G2point} (a), we can apply Theorem \ref{thm:2} with this $G_2$. Define $E \defeq E_{\mathrm{triv}} \# E_\ell$ to be the connected sum bundle as in Section \ref{sec:AdaptedBundles}, equipped with thickening data $\mathcal{T}(\eps)$. Take $A'$ to be the preglued connection, which is plainly asymptotic to $\Gamma_{\mathrm{triv}}$. In particular, $\kappa(E; A' \vert_{\End X})  = \ell$ due to the discussion of Section \ref{sec:AdaptedBundles}. By possibly shrinking $\mathcal{V}$, if necessary, we define $\Phi$ to be the $\C^m$-embedding of the same name from Theorem \ref{thm:2} (b) (here we are using the identifications $\mathcal{V} \cong G_1 \cong G_1 \times G_2$). If $A$ is any connection in the image of $\Phi$, then $\vert \kappa(E, A \vert_{\End X}) -  \ell \vert < \eps/2$ follows from the definition of $\eps$ as a parameter in the set of thickening data and the fact that $\kappa(E; A \vert_{\End X})$ recovers the Chern--Simons value of the asymptotic limit of $A$. 
\end{proof}

\subsection{Existence when $b^+(X) = 1$}\label{sec:b^+=1}

Here we consider the case where $b^+(X) = 1$ and $G = \SU(2)$. We assume (\ref{eq:rep}) is the identity (so $r = 2$); then the characteristic number $\kappa$ from Section \ref{sec:AdaptedBundles} is the second Chern number. Fix thickening data $\mathcal{T}_{\Gamma_{\mathrm{triv}}}$ and assume $\delta$ and $\beta$ are chosen as in the beginning of Section \ref{sec:GluingRegFam}. Define $\mathcal{T}(\eps)$ as in Section \ref{sec:b^+=0}. The following is the second of our main existence results for $\mASD$ connections. Theorem \ref{thm:C} with $b^+(X) =1$ is an immediate consequence.

\begin{theorem}[Existence of $\mASD$-connections when $b^+ = 1$]\label{thm:exist2}
Assume $b^+(X) = 1$ and fix an integer $\ell \geq 2$. Then for every $0 < \eps < \eps_0$, there is

\begin{itemize}
\item[(a)] a trivializable principal $\SU(2)$-bundle $E \rightarrow X$ that is canonically trivial on the end;
\item[(b)] a connection $A'$ on $E$ that is flat in the complement of a compact set, asymptotic to $\Gamma$ on the end, and satisfies $\kappa(E; A' \vert_{\End X})  = \ell$; and
\item[(c)] an irreducible $\mASD$ connection $A \in \rM(\mathcal{T}(\eps), A')$ on $E$ satisfying 
$$\vert \kappa(E, A \vert_{\End X}) -\ell \vert < \eps/2.$$
\end{itemize}
\end{theorem}

\begin{proof}
Our proof follows that of \cite[Section 7]{taubes3} and \cite[pp.~327---334]{donaldsonb}. The assumption that $b^+ = 1$ implies that the cokernel of the operator $d^+: L^p_{1, \delta}(\Omega^1(X, \bb{R})) \rightarrow L^p_\delta(\Omega^+(X, \bb{R}))$ is one-dimensional. As in Section \ref{sec:FlatConnections}, this cokernel can be realized as the space $H^+(X, \bb{R})$ of closed self-dual 2-forms in $L^2(\Omega^+(X, \bb{R}))$ that restrict on each slice $\left\{t \right\} \times N$ to be orthogonal to the space of harmonic forms on $N$. Fix a non-zero element $\omega_0 \in H^+(X, \bb{R})$; this is unique up to scaling. By unique continuation for solutions of elliptic equations, it follows that the set of points in $X$ where $\omega_0$ does not vanish is open and dense. In particular, we can find two distinct points $x_1, x_2 \in \mathrm{int}(X_0)$ with $\omega_0(x_1) \neq 0$ and $\omega_0(x_2) \neq 0$. When $\ell > 2$, choose additional points $x_3, \ldots, x_{\ell - 2} \in \mathrm{int}(X_0)$; these can be arbitrarily chosen, provided the $x_i$ are all distinct. The gluing Theorem \ref{thm:1} has a straightforward extension to handle gluing for multiple connected sums that we briefly describe now. 

Fix scaling parameters $\lambda_1, \ldots, \lambda_\ell > 0$, and set 
$$\lambda \defeq \max(\lambda_1, \ldots,  \lambda_\ell).$$ 
Here we will consider $\ell$ copies of $S^4$; denote these copies by $S_1^4, \ldots, S_\ell^4$, and fix points $x_i' \in S^4_i$. Then as we did in Section \ref{sec:Setupforgluing}, glue $x_i \in X_0$ to $x_i' \in S^4_i$ over balls with radii controlled by $\lambda_i$.  

At the bundle level, let $E_{\mathrm{triv}} \rightarrow X$ be the trivial $\SU(2)$-bundle, and let $E_1 \rightarrow S^4$ be the $\SU(2)$-bundle with $\kappa(E_1) = c_2(E_1)\left[S^4 \right] = 1$. More concretely, we can take $E_1$ to be the frame bundle of $ \Lambda^+ T^* S^4$ (then $\Lambda^+ T^* S^4$ is the adjoint bundle of $E_1$). In Section \ref{sec:GenGluing}, gluing the bundles depended on the choice of fiber isomorphism $\rho$ identifying the fibers of the principal bundles at the gluing points. In the present setting with $\ell$ gluing points, this corresponds to the choice of a fiber isomorphism 
$$\rho_i \in \mathrm{Gl}_i = \Hom_{\SU(2)}((E_{\mathrm{triv}})_{x_i}, (E_1)_{x_i'})$$
for each $1 \leq  i \leq \ell$.

Let $A_{\mathrm{triv}}$ be the trivial connection on $E_{\mathrm{triv}}$. Let $A'$ be the preglued connection on $E$ obtained from $A_{\mathrm{triv}}$ and the standard ``one-instanton'' $A_{\mathrm{st}}$ on each of the bundles $E_1 \rightarrow S^4_i$ for $1 \leq i \leq \ell$. Then $\kappa(E, A' \vert_{\End  X}) = \ell$; note that $A'$ depends on the $\lambda_i$ and $\rho_i$. The proof of Theorem \ref{thm:1} extends to produce $C, L, \lambda_0, J,\pi$, and $\xi  \in L^p_\delta(\Omega^+(X))$, satisfying the conditions of Theorem \ref{thm:1} (a)---(c) and Corollary \ref{cor:smallp0} whenever $0 < \lambda < \lambda_0$; though we suppress this in the notation, these quantities depend on the connections $A_{\mathrm{triv}}$ and $A_{\mathrm{st}}$, as well as the isomorphisms $\rho_i$. In particular, the connection $A \defeq  J (\xi)$ is irreducible and satisfies 
$$s(A) = - \sigma \pi \xi \indent \indent \textrm{and} \indent \indent  \vert \kappa(E, A \vert_{\End X}) -  \ell \vert < \eps/2.$$

It suffices to show that the $\lambda_i$ and $\rho_i$ can be chosen so that $\sigma \pi \xi  = 0$, since this implies that $A$ is $\mASD$. For this, let $X'$ be the complement in $X$ of the $L \lambda_0^{1/2}$-balls around the $x_i$; we assume $\lambda_0$ is small enough so these balls do not intersect and are contained in $X_0$. Note that the bundles $E$ and $E_{\mathrm{triv}}$ are canonically identified over $X'$, and so over $X'$ we can compare 2-forms on $E_{\mathrm{triv}}$ with 2-forms on $E$. The self-dual 2-form $\sigma \pi \xi$ vanishes if and only if the integral
\eqncount\begin{equation}\label{eq:x'int}
\intd{X'} \langle \omega \wedge  \sigma \pi \xi \rangle  = 0
\end{equation}
vanishes for all $\omega \in H^+(X, \mathrm{ad}(A_{\mathrm{triv}})) = H^+(X, \bb{R}) \otimes \mathfrak{g}$. 

\medskip

\noindent \emph{Claim}: 
 \eqncount\begin{equation}\label{eq:wts}
\intd{X'} \langle \omega \wedge  \sigma \pi \xi \rangle  = q_\omega^\ell(\big\{ (\lambda_i ,\rho_i)\big\}_i)  + O(\lambda^3)
\end{equation}
\emph{where} 
$$q_\omega^\ell(\big\{ (\lambda_i ,\rho_i)\big\}_i) \defeq \sumdd{i = 1}{\ell} \lambda_i^2 \mathrm{tr}(\rho_i \omega(x_i)).$$ 
\medskip

Here $\mathrm{tr}(\rho_i \omega(x_i)) \in \bb{R}$ is the pairing of $\rho_i$ and $\omega(x_i)$ as described in \cite[Equation (5.3)]{donaldsonb}. We will prove this claim below, but first we will show how it is used to finish the proof of the theorem. From the discussion leading up to the claim, we are interested in the simultaneous system of equations
\eqncount\begin{equation}\label{eq:done?}
q_\omega^\ell(\big\{ (\lambda_i ,\rho_i)\big\}_i)  = 0 \hspace{1cm} \forall \omega \in H^+(X, \mathrm{ad}(A_{\mathrm{triv}})).
\end{equation}
When $\ell  = 2$, the argument of \cite[Section V(ii)]{donaldsonb} carries over verbatim to show that the solutions set of the system (\ref{eq:done?}) is non-empty and cut out transversely, whenever $\max(\lambda_1, \lambda_2)$ is sufficiently small. This uses the assumption $\omega_0(x_1), \omega_0(x_2) \neq 0$. Note that Donaldson's argument uses $b^+(X) = 1$. (Alternatively, the reader could follow the original argument of Taubes \cite[Prop.~7.1]{taubes3}, but our notation is more inline with that of \cite{donaldsonb}.) When $\ell > 2$, it was pointed out by Taubes \cite[Prop.~6.2]{taubes3} that by taking $\max(\lambda_3, \ldots, \lambda_\ell)$ sufficiently small relative to $\max(\lambda_1, \lambda_2)$, any transverse zero of $q_\omega^2$ implies the existence of a transverse zero of $q_\omega^\ell$. In summary, for each $\ell \geq 2$, there are $ \lambda_0' > 0$ and $\mu \in (0, 1)$ so that the system (\ref{eq:done?}) has a nonempty, transverse solution set, for all $\lambda_1, \ldots, \lambda_\ell > 0$ with
$$\max(\lambda_1, \lambda_2) < \lambda_0' \indent \indent \textrm{and} \indent \indent \max(\lambda_3, \ldots, \lambda_\ell) < \mu  \max(\lambda_1, \lambda_2).$$
For any such $\lambda_1, \ldots, \lambda_\ell$, since $q_\omega^\ell$ is $O(\lambda^2)$, it then follows from the transversality of $q_\omega^\ell = 0$ and the identity (\ref{eq:wts}) that the solution sets to (\ref{eq:x'int}) and (\ref{eq:done?}) are diffeomorphic, provided $\lambda$ is sufficiently small. In particular, there is a simultaneous zero $\{ (\lambda_i , \rho_i)\}_{i = 1}^\ell $ of the solution set to (\ref{eq:x'int}). For this collection of gluing data, the glued connection $A$ is therefore $\mASD$, as desired.

\medskip

It therefore suffices to verify the above Claim. We will first unpack the notation. Note that the preglued connection $A'$ restricts on $X'$ to equal the trivial flat connection. Let $A'(\lambda_0)$ be the preglued connection defined using $\lambda_0$ at every gluing site, and the same $\rho_i$ as was used to define $A'$ (so the only difference between $A'$ and  $A'(\lambda_0)$ is that the former uses $\lambda_i$ at the gluing site $x_i$, while the latter uses $\lambda_0$ at all gluing sites). Define the map $i$ (and hence $\iota$) using $A'(\lambda_0)$ as a reference connection. Write $\Gamma = \Gamma_{\mathrm{triv}}$ for the trivial connection on $N$. Then we can write $A' = \iota(\Gamma_{\mathrm{triv}}, V') = i(\Gamma_{\mathrm{triv}}) + V' = A'(\lambda_0) + V'$ for some 1-form $V'$. It follows that $V' \vert_{X'} = 0$, and we note also that $A'(\lambda_0) \vert_{X'} = A_{\mathrm{triv}}$. 

Next, recall the map $P: L^p_\delta(\Omega^+) \rightarrow T_{\Gamma} \cH\times L^p_{1, \delta}(\Omega^1)$ from Claim 1 in the proof of Theorem \ref{thm:1}, and write $P \xi = (\eta, V)$. The definition of the map $J = J_{A_{\mathrm{triv}}, A_{\mathrm{st}}}$ gives
$$J(\xi) = i(\exp_{\Gamma}(\eta)) + V' + V$$
where $\exp_\Gamma : T_\Gamma \cH \rightarrow \cH$ is the exponential. The observations of the previous paragraph combine with the formula (\ref{eq:formfors}) to give that the restriction of $s(A)$ takes the following form:
$$s(A) \vert_{X'} = (1- \beta')F^+_{i(\exp_\Gamma(\eta))} + d^+_{i(\exp_\Gamma(\eta))} V + \frac{1}{2} \left[ V \wedge V \right]^+.$$

Returning to the integral (\ref{eq:x'int}), we can use the defining property of $\xi$ and the above identity for $s(A)$ to get
\eqncount\begin{equation}\label{eq:expofx'int}
\begin{array}{rcl}
\intd{X'}  \langle  \omega \wedge \sigma \pi \xi \rangle  & = & - \intd{X'}  \langle \omega \wedge  s(A) \rangle \\
& = & - \intd{X'} (1-\beta') \langle \omega \wedge   F^+_{i(\exp_\Gamma(\eta))} \rangle   - \intd{X'} \langle \omega \wedge  d^+_{i(\exp_\Gamma(\eta))} V \rangle  \\
&& - \frac{1}{2} \intd{X'}  \langle \omega \wedge   \left[ V \wedge V \right]^+ \rangle.
\end{array}
\end{equation}
Focus on the last term on the right. Recall from Lemma \ref{lem:deltadecay} that $\omega$ decays in $\C^0$ like $e^{- {\mu_\Gamma^+} t}$. In particular, $\omega$ is bounded and so
$$\intd{X'}  \left| \langle \omega \wedge   \left[ V \wedge V \right]^+ \rangle \right| \leq C_1 \Vert V \Vert_{L^2(X)}^2 \leq C_1 \Vert V \Vert_{L^2_\delta(X)}^2 \leq C_1 \Vert P \xi \Vert_{L^2_\delta(X)}^2 $$
for some constant $C_1$. By Corollaries \ref{cor:smallp0} and \ref{cor:smallp1}, this term decays like $\lambda^3$:
$$- \frac{1}{2} \intd{X'}  \langle \omega \wedge   \left[ V \wedge V \right]^+ \rangle = O(\lambda^3).$$
We can control the nonlinear parts of the other two terms in (\ref{eq:expofx'int}) similarly. Indeed, use $i(\exp_\Gamma(\eta)) = A_{\mathrm{triv}} + (D i)_\Gamma \eta + O( \eta^2)$ and the expansion formulas for the curvature and covariant derivative, to get
$$\begin{array}{rcl}
\intd{X'}   \langle  \omega \wedge  \sigma \pi \xi \rangle   & =  & - \intd{X'} (1-\beta')  \langle \omega \wedge  d^+ (D i)_\Gamma \eta \rangle  - \intd{X'}  \langle \omega \wedge  d^+ V \rangle   + O( \lambda^3)
\end{array}$$
where $d = d_{A_{\mathrm{triv}}}$. Focus on the first term on the right (there is no analogue of this term in the standard $\ASD$ framework). It follows from the definitions of $\beta'$ and $i$ that $(1-\beta') (d^+ (D i)_\Gamma \eta)$ is supported on $\left[ T- 1/2, T+1/2\right] \times N$. Using the formula (\ref{eq:daformulafordi}), we have
$$- \intd{X'}(1-\beta')   \langle \omega \wedge  d^+ (D i)_\Gamma \eta \rangle   =  - \intd{X'} (1-\beta') (\partial_t \beta'')  \langle \omega \wedge (dt \wedge \eta)^+ \rangle  = 0 $$
where the last equality uses the facts that (i) $\eta \in H^1_{\Gamma_{\mathrm{triv}}}$ is in the harmonic space on $N$, and (ii) elements of $H^+(X, \mathrm{ad}(A_{\mathrm{triv}}))$ restrict on each slice to be orthogonal to $H^1_{\Gamma_{\mathrm{triv}}}$. In summary, this gives
$$\begin{array}{rcl}
\intd{X'}   \langle  \omega \wedge  \sigma \pi \xi \rangle   & = & - \intd{X'} \langle \omega \wedge  d^+ V \rangle   + O(\lambda^3)\\
& = &  \intd{\partial X'} \langle \omega \wedge  V \rangle   + O(\lambda^3).
\end{array}$$
What remains is to estimate the integral $\int_{\partial X'} \langle \omega \wedge   V \rangle$. This is an integral taking place at the boundary of the disks centered at the gluing sites $x_1, \ldots, x_\ell$. In particular, this integral is identical to the analogous term that arises when gluing in the standard $\ASD$ setting (e.g., see the top of \cite[p.~328]{donaldsonb}). Then the argument of \cite[pp.~328---331]{donaldsonb} carries over verbatim to give
$$\intd{\partial X'} \langle \omega \wedge   V \rangle   = q^\ell_\omega(\big\{ (\lambda_i , \rho_i)\big\}) + O(\lambda^3).$$
This proves (\ref{eq:wts}).

\end{proof}

\subsection{An $\ASD$ existence result and a proof of Theorem \ref{thm:A}}\label{sec:ProofOfThmA}

Recall from Section \ref{sec:AuxiliaryChoices} the definition of the vector field $\Xi_\Gamma$ on the center manifold. We will be interested in the case where the flat connection $\Gamma$ satisfies the following hypothesis:

\begin{customhyp}{H}\label{hypothesis1}
There is a neighborhood $U \subseteq \mathcal{H}_\Gamma$ of $\Gamma$ so that every $a \in U$ flows under $\Xi_\Gamma$ to a flat connection in $U$.   
\end{customhyp}

\begin{example}\label{ex:ex1}\mbox{}

(a) Recall from Section \ref{sec:AuxiliaryChoices} that $U_\Gamma$ is a neighborhood of $\Gamma$ in the Coulomb slice through $\Gamma$. Suppose the set of flat connections in $U_\Gamma$ is smooth in a neighborhood $U' \subseteq U_\Gamma$ of $\Gamma $ and has the same dimension as $\mathcal{H}_\Gamma$. Then $U \defeq U' \cap \mathcal{H}_\Gamma$ satisfies Hypothesis \ref{hypothesis1}. 

\medskip

(b) The assumption of (a) trivially holds when $\Gamma$ is non-degenerate, since $\mathcal{H}_\Gamma$ consists of a single point. More generally, the assumption of (a) also holds when the Chern--Simons function is Morse--Bott in a neighborhood of $\Gamma$ (though there is no assumption in part (a) about the nondegeneracy of the Hessian in the normal directions). 

\medskip

(c) Suppose $N = T^3$, and let $\Gamma$ be a flat connection on the trivial $\SU(2)$-bundle. If $\Gamma$ is not gauge equivalent to the trivial connection, then $\Gamma$ satisfies the assumption in (a), and hence Hypothesis \ref{hypothesis1}; see \cite[Lemma 14.2(i)]{gompfmrowka}. However, the trivial connection on $T^3$ does \emph{not} satisfy Hypothesis \ref{hypothesis1}.
\end{example}

The main usefulness of Hypothesis \ref{hypothesis1} for us is through the following theorem.

\begin{theorem}\label{prop:exists3}
Consider the situation of Theorem \ref{thm:1}, and assume $A_1$ and $A_2$ are regular. In addition, assume that $\Gamma_1$ and $\Gamma_2$ each satisfy Hypothesis \ref{hypothesis1}. Let $\lambda_0 > 0$ be the constant from Theorem \ref{thm:1}. Then there is some $0 < \lambda_0' \leq \lambda_0$ so that for all $ \lambda  \in(0, \lambda_0')$ the $\mASD$ connection $\mathcal{J}(A_1, A_2)$ guaranteed by Theorem \ref{thm:1} (and hence by Theorems \ref{thm:B} and \ref{thm:C}) is in fact $\ASD$. 
\end{theorem}

\begin{proof}
Fix $0 < \lambda < \lambda_0$ and let $A_\lambda \defeq \mathcal{J}(A_1, A_2)$ be the $\mASD$ connection from Theorem \ref{thm:1} associated to this value of $\lambda$. Recall that $\Gamma = \Gamma_1 \sqcup \Gamma_2$, and so $\Gamma$ satisfies Hypothesis \ref{hypothesis1} since the $\Gamma_k$ do. It follows from (\ref{eq:close}) that $p_T(A_\lambda) \in \cH_{in}$ converges to $\Gamma$ as $\lambda$ approaches 0. In particular, by taking $\lambda$ sufficiently small, Hypothesis \ref{hypothesis1} implies that the $\Xi_\Gamma$-flow line beginning at $p_T(A_\lambda)$ lies in $\cH_{in}$ for all positive time. This implies $i(p_T(A_\lambda))$ is $\ASD$ (see the paragraph just before the statement of Lemma \ref{lem:derofalpha}), and so
$$F^+_{A_\lambda} = F^+_{A_\lambda} - \beta' F^+_{i(p_T(A_\lambda))} = s(A_\lambda) = 0.$$
\end{proof}

Now we can prove our application from the introduction.

\begin{proof}[Proof of Theorem \ref{thm:A}]
Take $\Gamma$ to be the trivial flat connection on the trivial $\SU(2)$-bundle. We will show that the two conditions on $N$ stated in Theorem \ref{thm:A} each imply that $\Gamma$ satisfies Hypothesis \ref{hypothesis1}; it will then be immediate that the $\mASD$ connection guaranteed by Theorem \ref{thm:C} is in fact $\ASD$, as desired.

First assume $N$ is a circle bundle over a surface with positive Euler class. Then \cite[Corollary 13.2.2]{\mmr} implies that $\Gamma$ satisfies the condition of Example \ref{ex:ex1} (a), and thus Hypothesis \ref{hypothesis1}.

Now assume that $b_1(N)  \leq 1$. If $b_1(N) = 0$, then $H^1_\Gamma = H^1(N) \otimes \frak{g} = 0$ and so $\Gamma$ is nondegenerate. Thus, $\Gamma$ again satisfies Hypothesis \ref{hypothesis1}, but this time by Example \ref{ex:ex1} (b).

Finally, suppose $b_1(N) = 1$. We will show here that $\Gamma$ satisfies the condition of Example \ref{ex:ex1} (a). Since $b_1(N) = 1$, there is a loop $\gamma: S^1 \rightarrow N$ and a harmonic 1-form $\eta \in \Omega^1(N, \bb{R})$ so that $\int_\gamma \eta = 1$. For each $\xi \in \frak{g}$, let
$$a_\xi \defeq \Gamma + \xi \otimes \eta .$$
We claim that $a_\xi$ lies in the center manifold $\mathcal{H}_\Gamma$ for all sufficiently small $\xi$. To see this, first note that $a_\xi$ is flat, since
$$F_{a_\xi} = F_\Gamma + \xi \otimes  (d \eta) + \frac{1}{2} [\xi, \xi] \otimes \eta \wedge \eta = 0.$$
This connection also lies in the Coulomb gauge slice for $\Gamma$, since
$$d^*_\Gamma( a_\xi - \Gamma) = \xi \otimes d^* \eta = 0.$$
Recall the map $\Theta$ and the vector field $\nabla f_\Gamma$ from Section \ref{sec:TheCenterManifold}. Since $F_{a_\xi} = 0$, we have
$$\Theta(a_\xi) = 0.$$
Thus $\nabla_{a_\xi} f_\Gamma = 0$. One of the defining features of $\mathcal{H}_\Gamma$ is that it contains all zeros of $\nabla f_\Gamma$ that are sufficiently close to $\Gamma$, so this proves the claim.

It thus follows that there is some $\epsilon > 0$ so that the map
$$B_\epsilon(0) \subseteq \frak{g} \longrightarrow \mathcal{H}_\Gamma \hspace{2cm} \xi \longmapsto a_\xi$$
is well-defined. It is clearly an immersion, so a dimension count implies that it must be a local diffeomorphism; this uses the fact that $b_1(N) = 1$. This establishes the condition of Example \ref{ex:ex1} (a). 
\end{proof}

\section{Partial compactification---the Taubes boundary}\label{sec:PartialCompactification}

Here we give a more global formulation of the result of Theorem \ref{thm:exist1} in the case where $G = \SU(2)$ and $\ell = 1$. Fix a closed set $X'_0$ contained in the interior of the compact part $X_0$. Let $A_{\mathrm{st}}$ be the standard one-instanton on the $\SU(2)$-bundle $E_1 \rightarrow S^4$ with $c_2(E_1)\left[S^4 \right] = 1$. For $x \in X'_0$, let $X_x$ be the connected sum of $X$ and $S^4$ obtained by gluing $x \in X$ to the north pole in $S^4$. Similarly, glue the trivial $\SU(2)$-bundle on $X$ to $E_1 \rightarrow S^4$ and let $E_x \rightarrow X_x$ be the resulting bundle. Let $A'_x = A'(A_{\mathrm{triv}},A_{\mathrm{st}}) $ be the preglued connection on $E_x$, where $A_{\mathrm{triv}}$ is the trivial connection on $X$. Note that in the present situation, all auxiliary gluing data can be chosen to be independent of $x$. For example, the fiber isomorphism $\rho$ of Section \ref{sec:Setupforgluing} can be taken to be independent of $x$ since we are starting with the trivial bundle on $X$.

\subsection{The Taubes Boundary}
Fix $\epsilon  > 0$, and let $\mathcal{T}(\eps)$ be thickening data with this choice of $\eps$, as in Section \ref{sec:b^+=0}. By Theorem \ref{thm:1}, there are $\epsilon_0, \lambda_0>0$ so the following holds: For all $0 < \epsilon < \epsilon_0$ and $0 < \lambda \leq \lambda_0$, there is an irreducible, regular $\mASD$ connection 
$$A(x, \lambda) \defeq \mathcal{J}(A_{\mathrm{triv}},A_{\mathrm{st}}) \in  \A^{1,p}(\mathcal{T}(\eps))$$ 
with the property that $A(x, \lambda)  - A_{\mathrm{triv}} \vert_{X \backslash \mathrm{nbhd}(x)}$ goes to zero in $\lambda$ in the sense of (\ref{eq:close}). This $\epsilon_0$ depends only on the trivial flat connection on the 3-manifold $N$; hence $\epsilon_0$ is independent of $x$. Since $X'_0$ is compact, we can assume this $\lambda_0$ is independent of $x$ as well. 

We want to allow $x$ to vary, and for this, we form the space
$$\E' \defeq \left\{(x, \lambda, A) \in X'_0 \times (0, \lambda_0] \times \A^{1,p}(\mathcal{T}(\epsilon)) \; \Big| \; A \in \rM_{reg}(\mathcal{T}(\eps),A(x, \lambda)) \right\}.$$ 
Let $\Pi': \E' \rightarrow X'_0 \times (0, \lambda_0]$ be the projection to the first two factors. Then the assignment $\Psi'(x, \lambda) \defeq (x, \lambda, A(x, \lambda) )$ defines a section of $\Pi'$. Just as in Theorem \ref{thm:2}, there is an open neighborhood $\mathcal{U}' \subseteq \E'$ of the image of $\Psi'$ so that the restriction $\Pi' \vert_{\mathcal{U}'}$ is a locally-trivial $\C^m$-fiber bundle over $X'_0 \times (0, \lambda_0]$. By construction, the fiber over $(x, \lambda)$ is an open subset of $\rM_{reg}(\mathcal{T}(\eps),A(x, \lambda))$ containing $A(x, \lambda)$.

\begin{remark}
Here we describe a sense in which Theorem \ref{thm:exist1} can be viewed as a local version of this fiber bundle construction. Fix a small neighborhood $U_x \subseteq X_0$ around $x$. The gluing procedure of Section \ref{sec:Setupforgluing} identifies this with a small neighborhood of the north pole in $S^4$. The standard description \cite[Ch. 6]{freed-uhlenbeck1} of the $\ASD$ moduli space $\M_1(S^4, \SU(2))$ gives an embedding $S^4 \times (0, \lambda_0] \rightarrow \M_1(S^4, \SU(2))$ with the $S^4$-component specifying the center of mass and $(0, \lambda_0]$ parametrizing the scale of the curvature; here the {energy-density of the} curvature is concentrating, as $\lambda$ approaches 0, to a Dirac delta measure supported at the center of mass. Combining these, we have a diffeomorphism 
$$f: U_x \times (0, \lambda_0] \longrightarrow \mathcal{V}  \subseteq \M_1(S^4, \SU(2))$$
onto an open set $\mathcal{V}$. It follows from this construction that there is a local trivialization of the fiber bundle $\Pi' \vert_{\mathcal{U}'}$ relative to which $\Psi'$ takes the form $(y, \lambda) \mapsto (y, \lambda, \Phi(f(y, \lambda)))$ where $\Phi$ is the map of Theorem \ref{thm:exist1}. In fact, by possibly shrinking $U_x$ further, this local trivialization can be taken to be over the full cylinder $U_x \times (0, \lambda_0]$; this due to the fact that the constructions in the proof of Theorem \ref{thm:2} can be taken to be uniform in $\lambda$. This construction is exploiting a coupling between the parameter $\lambda$ and the ``scale'' parameter for the concentration of instantons on $E_1 \rightarrow S^4$; see \cite[p.~323]{donaldson-kronheimer1} for a related discussion.
\end{remark}

Now we consider the behavior of this section $\Psi'$ near $\lambda = 0$. For this, suppose $(x_n, \lambda_n) \in X'_0 \times (0, \lambda_0]$ is a sequence with $\lambda_n \rightarrow 0$; we will call $\Psi'(x_n, \lambda_n)$ a \emph{bubbling sequence in $X_0'$}. By passing to a subsequence, we may assume the $x_n$ converge to some $x_\infty \in X'_0$. It follows from a straight-forward Uhlenbeck-type compactness argument and (\ref{eq:close}) that, after passing to a subsequence, the associated connections $A(x_n, \lambda_n)$ converge \emph{weakly} to the ideal connection $(A_{\mathrm{triv}}, x_0)$ in the sense that the {energy densities $\vert F_{A(x_n, \lambda_n)}\vert^2$} converge in measure to the delta measure supported at $x_0$, and
$$\limd{n \rightarrow \infty} \Big\| \iota^{-1}(A(x_n, \lambda_n)) - \iota^{-1}(A_{\mathrm{triv}}) \Big\|_{L^2_2(N) \times L^p_{1, \delta}(X \backslash B_r(x_0))} = 0$$
for all $r > 0$; see \cite[Section 4.4.1]{donaldson-kronheimer1} for the analogous $\ASD$ case.

Following the lead of \cite[Section 4.4.1]{donaldson-kronheimer1}, the discussion of the previous paragraph can be framed geometrically as follows. Consider the set
$$I(\mathcal{U}') \defeq \mathcal{U}' \cup (X'_0 \times \left\{A_{\mathrm{triv}}\right\})$$ 
which we view as coupling the connections in $\mathcal{U}' \subseteq \E'$ into the same space as the above-mentioned ideal connections. We can extend $\Pi'\vert_{\mathcal{U}'}$ to a map $I(\Pi'): I(\mathcal{U}') \rightarrow X'_0 \times [0, \lambda_0 ]$ by declaring it to send $(x, A_{\mathrm{triv}})$ to $(x, 0)$. Give $I(\mathcal{U}')$ any topology (more below) for which the map $I(\Pi')$ is continuous and so that the notion of weak convergence from the previous paragraph implies convergence in $I(\mathcal{U}')$; we assume also that this topology is first countable. Then the observations of the previous paragraph imply the section $\Psi'$ extends continuously over $X'_0 \times \left\{0 \right\}$ to a section $\overline{\Psi}'$ of $I(\Pi')$. It is due to this that we may view $I(\mathcal{U}')$ as a ``partial compactification'' for $\mASD$ connections: The bubbling sequences in $X_0'$ converge in $I(\mathcal{U}')$.

\subsection{Compactification Issues}
We end this section with several comments about the construction of the partial compactification $I(\mathcal{U}')$, as well as some of its limitations. This partial compactification is constructed only so that bubbling sequences in $X_0'$ converge---our assumptions on the topology on $I(\mathcal{U}')$ do not necessarily imply subsequential convergence of other types of sequences. The simple reason for this is that we do not yet know how such sequences behave, and what additional limiting objects we would need to include in $I(\Pi')$ to ensure their subsequential convergence. What we are presently lacking is a sufficiently strong version of Uhlenbeck's compactness theorem for $\mASD$ connections. In the end, such a theorem would need to (at least) address the following:

\begin{enumerate}[leftmargin=*]
\item[(a)] \emph{Bubble formation on the end}: To what extent is the $\mASD$ condition preserved under Uhlenbeck limits where the curvature concentrates at a point in $\End  X$? More fundamentally, is the connections space $\A^{1,p}(\mathcal{T}_\Gamma)$ suitably closed under such limits? This is related to (c) below. 

\smallskip
\noindent
From the gluing perspective, we avoided these questions altogether by only gluing at points in the compact part where $\mASD$ connections are $\ASD$; that is, $I(\mathcal{U}')$ only corresponds to the points in the ``Taubes boundary'' that corresponds to bubbles in $X'_0 \subseteq X \backslash \End  X$ {(and relative to a \emph{fixed} gluing parameter $\rho$; see Remark \ref{rem:G2point} (b))}. A more thorough investigation would require not only an understanding of the $\mASD$ condition under Uhlenbeck limits, but also an understanding of how to glue at points on $\End  X$.

\item[(b)] \emph{Energy escaping down the end}: One example of this is bubbling on the end, as discussed in (a). Another example is where a non-trivial amount of energy escapes down the end. This is familiar in the $\ASD$ setting, where compactification can be achieved by including spaces of translationally-invariant $\ASD$ connections on $\bb{R} \times N$ (spaces of ``Floer trajectories''); see \cite{floer1,donaldson10}. In the $\mASD$ setting, one would likely need to include spaces of $\mASD$ connections on $\bb{R} \times N$ to account for energy escaping. The details of this appear to be subtle, since the energy values of such connections are not governed by topological quantities, as is the case in the $\ASD$ setting. (In the discussion above, where we considered sequences in the image of $\Psi'$, we were able to exclude non-trivial energy on the end by appealing to (\ref{eq:close}).) 

\item[(c)] \emph{Failure of the slice-wise gauge fixing condition}: In the definition of the space $\A^{1,p}(\mathcal{T}_\Gamma)$ from Section \ref{sec:TheSpaceOfConnections}, we restricted attention to connections that restrict on each time slice $\left\{t \right\} \times N$, for $t \geq T$, to be gauge equivalent to a connection in the gauge slice $U_\Gamma$. This is an open condition in the space of all $L^p_{1, loc}$ connections, and we do not see a reason why this condition should be retained through limits of $\mASD$ connections. 
\end{enumerate}

It is clear from these observations that $I(\mathcal{U}')$ is by no means the end of the story when it comes to compactification. It is due to this that we have avoided defining a specific topology on $I(\mathcal{U}')$ above, choosing instead to axiomatize a minimal set of desirable properties. 

\begin{remark}\label{rem:endglue} A challenging problem is to establish a general ``gluing theorem {on the ends}" for \mASD-connections $(X_1, A_1)$ and $(X_2, A_2)$ with ``matching boundary conditions" on two given $4$-manifolds with cylindrical ends. Such a theorem was 
obtained in \cite{morgan-mrowka94}, for \ASD-connections assuming that their flat limits in the common $3$-manifold $N$ end where irreducible smooth points in the representation variety of $N$. In this case, the gluing was unobstructed. More generally, in order to glue two \mASD-connections, 
the matching conditions should at least include 
(i) an identification $\End X_1 \cong \End X_2 \cong N \times [0, \infty)$, (ii) the same flat reference connection $\Gamma$ on $N$ and center manifold $\cH_\Gamma$, and (iii) the same flow lines $\hat\omega_h$ in $\cH_\Gamma$, where $h = p_T(A_1) = p_T(A_2)$. We would expect the glued-up connections to provide an {embedded submanifold} of connections on the closed $4$-manifold $X = X_1 \cup X_2$ {obtained by} identifying along their cylindrical ends. {In this general $\mASD$ case, the glued-up connection would presumably satisfy some version of the $\mASD$ equation on $X$ that equals the $\ASD$ equation on the complement of the neck. How do the $\ASD$ connections on $X$ compare to these ``$\mASD$ connections'' on $X$? For example, do these ``$\mASD$ connections'' on $X$ form some sort of local thickening of the $\ASD$ moduli space, as is the case for cylindrical end manifolds? Since $X$ is compact, it seems likely that the $\ASD$ operator differs from the ``$\mASD$ operator'' by a compact operator. Can this operator be scaled in some way to show that $\ASD$ and $\mASD$ spaces on $X$ are, in some sense, cobordant?}
\end{remark}

\providecommand{\bysame}{\leavevmode\hbox to3em{\hrulefill}\thinspace}
\providecommand{\MR}{\relax\ifhmode\unskip\space\fi MR }
\providecommand{\MRhref}[2]{%
  \href{http://www.ams.org/mathscinet-getitem?mr=#1}{#2}
}
\providecommand{\href}[2]{#2}

\end{document}